\documentclass[11pt,a4paper]{article}

\usepackage{graphicx}
\usepackage{amsmath}
\usepackage{amsfonts}
\usepackage{amssymb}
\usepackage{enumerate}
\usepackage{lscape}
\usepackage{longtable}
\usepackage{rotating}
\usepackage{multirow}
\usepackage{color}   
\usepackage{bbm}
\usepackage{url}
\usepackage{subfigure}
\usepackage{rotating}
\usepackage{titlesec}
\usepackage[titletoc]{appendix}

\usepackage{mathtools,xparse}

\usepackage{bm}
\usepackage{hhline}
\usepackage{filecontents}
\usepackage{pgfplots}
\usepackage{float}
\usepackage{pgfplotstable}
\usepackage{array} 
\usepackage{longtable}

\usepackage{import}

\usepackage{multirow}

\usepackage{nth}

\usepackage{hyperref}
\hypersetup{
     colorlinks   = true,
     citecolor    = blue
}

\usepackage{grffile}

\newtheorem{theorem}{Theorem}[section]

\usepackage[ruled,vlined,noline]{algorithm2e}

\newtheorem{corollary}{Corollary}[section]

\newtheorem{definition}{Definition}[section]

\newtheorem{lemma}{Lemma}[section]

\newtheorem{proposition}{Proposition}[section]

\newtheorem{assumption}{Assumption}[section]

\newenvironment{proof}[1][Proof]{\textbf{#1.} }{\ \rule{0.5em}{0.5em} \vspace{1ex}}

\def\real{\mathbb{R}}
\newcommand{\eps}{\epsilon}
\newcommand{\epsg}{\epsilon_g}
\newcommand{\epsH}{\epsilon_H}
\newcommand{\low}{\mathrm{low}}
\newcommand{\up}{\mathrm{up}}
\newcommand{\ecurv}{\epsilon^{1/2}}

\DeclareMathOperator{\pfun}{\varrho}

\renewcommand{\Pr}{\mathbb{P}}

\newcommand{\E}[1]{\mathbb{E}\left[#1\right]}

\def\skn{\pi}
\def\cch{\kappa_H}
\def\ccg{\kappa_g}
\def\sta{\mathrm{sta}}
\def\Tj{T_{\epsilon,J}}

\def\smallnumerics{0}

\title{A Subsampling Line-Search Method with Second-Order Results}

\author{
E. Bergou \thanks{MaIAGE, INRAE, Universit\'e Paris-Saclay, 78350 Jouy-en-Josas, France
 ({\tt elhoucine.bergou@inra.fr}). King  Abdullah  University  of  Science  and  Technology  (KAUST),  Thuwal,  Saudi  Arabia. 
 This author received support from the AgreenSkills+ fellowship programme which has received funding from the EU's Seventh Framework Programme under grant agreement No FP7-609398 (AgreenSkills+ contract).
}
\and
Y. Diouane\thanks{ISAE-SUPAERO, Universit\'e de Toulouse, 31055 Toulouse Cedex 4, France
 ({\tt youssef.diouane@isae.fr}).
}
\and
V. Kunc \thanks{Department of Computer Science, Faculty of Electrical Engineering, Czech Technical University in Prague ({\tt kuncvlad@fel.cvut.cz}). Support for this author was provided by the CTU SGS grant no. SGS17/189/OHK3/3T/13.}
\and
V. Kungurtsev \thanks{Department of Computer Science, Faculty of Electrical Engineering, Czech Technical University in Prague ({\tt vyacheslav.kungurtsev@fel.cvut.cz}). Support for this author was provided by the OP VVV project CZ.02.1.01/0.0/0.0/16\_019/0000765 
"Research Center for Informatics".}
\and
C. W. Royer \thanks{LAMSADE, CNRS, Universit\'e Paris-Dauphine, Universit\'e PSL, 75016 PARIS, FRANCE ({\tt clement.royer@dauphine.psl.eu}). Support for this author was partially provided by Subcontract 3F-30222 from Argonne National Laboratory.} 
}

\newcommand{\revised}[1]{\textcolor{black}{#1}}
\def\skn{\pi}
\def\cch{\kappa_H}
\def\ccg{\kappa_g}
\def\sta{\mathrm{sta}}
\def\Tj{T_{\epsilon,J}}

\begin{document}

\maketitle
{\small
\begin{abstract}
In many contemporary optimization problems such as those arising in machine 
learning, it can be computationally challenging or even infeasible to evaluate 
an entire function or its derivatives. This motivates the use of stochastic 
algorithms that sample problem data, which can jeopardize the guarantees 
obtained through classical globalization techniques in optimization such as a 
line search. Using subsampled function values is particularly challenging for 
the latter strategy, which relies upon multiple evaluations. For nonconvex 
data-related 
problems, such as training deep learning models, one aims 
at developing methods that converge to second-order stationary points quickly,
i.e., escape saddle points efficiently. This is particularly difficult to 
ensure when one only accesses subsampled approximations of the objective and its 
derivatives.

In this paper, we describe a stochastic algorithm based on negative curvature and 
Newton-type directions that are computed for a subsampling model of the objective. 
A line-search technique is used to enforce suitable decrease for this model, and 
for a sufficiently large sample, a similar amount of reduction holds for the true 
objective. We then present worst-case complexity guarantees for a notion of 
stationarity tailored to the subsampling context. Our analysis encompasses 
the deterministic regime, and allows us to identify sampling requirements for 
second-order line-search paradigms. As we illustrate through real data 
experiments, these worst-case estimates need not be satisfied for our method to 
be competitive with first-order strategies in practice.
\end{abstract}
}%

\begin{center}
\end{center}

\section{Introduction} \label{sec:intro}

In this paper, we aim to solve
\begin{equation}\label{eq:prob1}
\min_{x \in \real^n} f(x) :=  \frac{1}{N}\sum_{i=1}^N f_i(x),
\end{equation}
where the objective function $f$ is not necessarily convex  and the components 
$f_i$ are assumed to be twice-continuously differentiable on $\real^n$. 
We are interested in problems in which the number of components $N \ge 1$ is 
extremely large, so that it becomes computationally 
infeasible to evaluate the entire function, its gradient or its Hessian. 
 
To overcome this issue, we consider the use of \emph{subsampling techniques} 
to compute stochastic estimates of the objective function, its gradient and its 
Hessian. Given a random set $\mathcal{S}$ sampled from $\{1, \ldots, N \}$ and a 
point $x \in \real^n$, we use the following estimates of $f$ and its 
derivatives:
\begin{equation} \label{eq:subsampmodel}
\hat f(x;\mathcal{S}) := \frac{1}{|\mathcal{S}|} \sum_{i\in \mathcal{S}} f_i(x), \quad
g(x;\mathcal{S}) := \frac{1}{|\mathcal{S}|} \sum_{i\in \mathcal{S}} \nabla f_i(x), \quad 
H(x;\mathcal{S}) := \frac{1}{|\mathcal{S}|} \sum_{i\in \mathcal{S}} \nabla^2 f_i(x),
\end{equation}
where $|\mathcal{S}|$ 
denotes the cardinal number of the sampling set $\mathcal{S}$. We are interested in iterative 
minimization procedures that choose a new random sampling set $\mathcal{S}$ at every 
iteration. 

The standard subsampling optimization procedure for such a problem is the (batch) stochastic gradient 
descent (SGD) method, wherein the gradient is estimated by one component gradient $\nabla f_i$ 
(or a batch of component), and a step is taken in the negative of this direction. 
The SGD framework is ubiquitous in a variety of applications, including 
large-scale machine learning problems, particularly those arising from the training of deep neural net 
architectures~\cite{bottou2018reviewml}. However, it is known to be sensitive to nonconvexity, 
particularly in the context of training deep neural nets. For such problems, it has indeed been observed 
that the optimization landscape for the associated (nonconvex) objective exhibits a significant number of 
saddle points, around which the flatness of the function tends to slow down the convergence of 
SGD~\cite{dauphin2014identifying}. This behavior is typical of first-order methods, despite the fact 
that those schemes almost never converge to saddle points~\cite{lee2019first}. By incorporating 
second-order information, one can guarantee that saddle points can be escaped from at a favorable rate: 
various algorithms that provide such guarantees 
while only requiring gradient or Hessian-vector products have been proposed in the literature~\cite{agarwal2017fasterrate,allen2018natasha,liu2017noisy,xu2018first}. 
Under certain accuracy conditions, which can be satisfied with arbitrarily high probability by controlling 
the size of the sample, these methods produce a sequence of iterates that converge to a local minimizer at a 
certain rate. Alternatively, one can extract second-order information and escape 
saddle points using accelerated gradient techniques in the stochastic 
setting~\cite{tripuraneni2018stochastic}. The results are also in high probability, with a priori tuned 
stepsizes. Noise can 
be used to approximate second-order information as well~\cite{xu2018first}. 
Recent proposals~\cite{xu2019newton,yao2018inexact} derive high probability convergence results 
with second-order steps (e.g., Newton steps) based on sampled derivatives and exact objective values, by means of trust-region and 
cubic regularization frameworks.
%
\textcolor{black}{Stochastic subsampling Newton methods, 
including~\cite{berahas2017newtonsketchsubsampled, bollapragada2019subsamplednewton, byrd2011stochhessian, erdogdu2015subsampled, pilanci2017newtonsketch, roosta2019subsampled,xu2016subsampled},
exploit second-order information to accelerate the convergence while typically using
a line search on exact function values to ensure global convergence.}
In the general stochastic optimization setting, 
a variety of algorithms have been extended to handle access to  sampled derivatives, and 
possibly function values: of particular interest 
to us are the algorithms endowed with complexity guarantees. When exact function values can be computed, one can employ strategies based on 
line search~\cite{cartis2018probamodels}, cubic 
regularization~\cite{cartis2018probamodels,kohler2017sub} or trust region~\cite{curtis2019stochtr,gratton2018trprobamodels} to compute a step of suitable length. 
Many algorithms building on SGD require the tuning of the step size parameter (also called 
learning rate), which can be cumbersome without knowledge of the Lipschitz constant. On the contrary, 
methods that are based on a globalization technique (line search, trust region, quadratic or cubic 
regularization) can control the size of the step in an adaptive way, and are thus less sensitive to 
parameter tuning.

In spite of their attractive properties with respect to the step size, globalized techniques are 
challenging to extend to the context of inexact function values. Indeed, these methods traditionally 
accept new iterates only if they produce a sufficient reduction of the objective value. 
Nevertheless, inexact variants of these schemes have been a recent topic of interest in the literature.
In the context of stochastic optimization, several trust-region algorithms that explicitly deal with 
computing stochastic estimates of the function values have been described~\cite{blanchet2019convergence,chen2018stochastic,larson2016stochtr}. In the specific case of least-squares problems, both approaches (exact and inexact 
function values) have been incorporated within a 
Levenberg-Marquardt framework~\cite{bergou2018stochlevenberg,bergou2016lmprobamodels}.
The use of stochastic function estimates in a line-search framework (a process that heavily 
relies on evaluating the function at tentative points) has also been the subject of very recent 
investigation. A study based on proprietary data~\cite{kungurtsev2016} considered an inexact 
Newton and negative curvature procedure using each iteration's chosen
mini-batch as the source of function evaluation sample in the line search. A stochastic 
line-search technique was introduced in~\cite{mahsereci2017probalinesearches}, where extensive 
experiments matching performance to pre-tuned SGD were presented.
An innovating 
technique based on a backtracking procedure for steps generated by a limited memory
BFGS method for nonconvex problems using first-order information was recently proposed
in~\cite{bollapragada2018progressive}, and first-order convergence results were derived.
Finally, contemporary to the first version of this paper, a stochastic 
line-search framework was described by~\cite{paquette2018stochls}.
Similarly to our scheme, this algorithm computes stochastic estimates
for function and gradient values, which are then used within a line-search 
algorithm. However, the two methods differ in their inspiration and results: we provide more 
details about these differences in the next paragraph, and throughout the paper when relevant.

In this paper, we propose a line-search scheme with second-order guarantees based on subsampling 
function and derivative evaluations. Our method uses these subsampled values to compute Newton-type and 
negative curvature steps. Although our framework bears similarities with the approach 
of~\cite{paquette2018stochls}, the two algorithms are equipped with different analyzes, each 
based on their own arguments from probability theory. The method 
of~\cite{paquette2018stochls} is designed with first-order guarantees in mind (in particular, the 
use of negative curvature is not explored), and its complexity results are particularized to the 
nonconvex, convex and strongly convex cases; our work presents a line-search method that is 
dedicated to the nonconvex setting, and to the derivation of second-order results.
Earlier work on second-order guarantees for subsampling methods often focused on complexity 
bounds holding with a high probability (of accurate samples being taken at each iteration), 
disallowing poor outlier estimates of the problem function. 
Our results are complementary as we establish a rate of convergence to points satisfying approximate second-order optimality conditions in expectation.
\revised{In addition to theoretical results, we test an implementation of our method on 
contemporary neural network training tasks, and compare it with a 
stochastic gradient approach using the same amount of sampling. Our results indicate that the 
proposed method performs well when these problems involve a large amount of data but a small 
number of parameters. Although the case of a large number of parameters require a careful 
implementation, we provide insights regarding the promises of our second-order technique.
}

We organize this paper as follows. In Section~\ref{sec:algo}, we describe our proposed 
approach based on line-search techniques. In Section~\ref{sec:expdecrease}, we derive 
bounds on the amount of expected decrease that can be achieved at each iteration by our 
proposed approach. Section~\ref{sec:cvwcc} gives  the  global convergence rate of our 
method under appropriate assumptions, followed by a discussion about the required properties and their 
satisfaction in practice. A numerical study of our approach is provided in Section~\ref{sec:numerics}.
A discussion of conclusions and future research is given in Section~\ref{sec:conc}.

Throughout the paper, $\|.\|$ denotes the Euclidean norm. A vector $v \in \real^n$ will be called a unit 
vector if $\|v\|=1$. Finally, $\mathbbm{I}_n$ denotes the identity matrix of size $n$.
\section{Subsampling line-search method} \label{sec:algo}
In this section, we introduce a line-search algorithm dedicated to solving the 
unconstrained optimization problem~\eqref{eq:prob1}: the detailed framework is 
provided in Algorithm~\ref{alg:alas}. 
At each iteration $k$, our method computes a random sampling set $\mathcal{S}_k$, and the associated estimates $g_k:=g(x_k;\mathcal{S}_k)$ and $H_k:=H(x_k;\mathcal{S}_k)$ of $\nabla f(x_k)$ and $\nabla^2 f(x_k)$, respectively. 
Since we have computed $\mathcal{S}_k$, the model $\hat f_k(\cdot) := \hat f(\cdot;\mathcal{S}_k)$ of the function $f$ is also defined: the estimates $g_k:=g(x_k;\mathcal{S}_k)$ and $H_k:=H(x_k;\mathcal{S}_k)$ define the quadratic Taylor expansion of $\hat f_k$ around $x_k$, which we use to compute a search direction $d_k$.
The form of this direction is set based on the norm of $g_k$ as well as the minimum eigenvalue of $H_k$, denoted by $\lambda_k$. The process is described through Steps 2-5 of Algorithm~\ref{alg:alas}, and involves an optimality tolerance $\epsilon$. When $\lambda_k < -\epsilon$, the Hessian estimate is indefinite, and we choose to use a negative curvature direction, as we know that it will offer sufficient reduction in the value of $\hat f_k$ (see the analysis of Section~\ref{subsec:expdecrease:generalres}). When $\lambda_k \ge \|g_j\|^{1/2}$, the quadratic function defined by $g_k$ and $H_k$ is (sufficiently) positive definite, and this allows us to compute and use a Newton direction. Finally, when $\lambda_k \in \left[-\epsilon^{1/2},\|g_k\|^{1/2}\right]$, we regularize this quadratic by an amount of order $\epsilon^{1/2}$, so that we fall back into the previous case: we thus compute a regularized Newton direction, for which we will obtain desirable decrease properties.
Once the search direction has been determined, and regardless of its type, a backtracking line-search strategy is applied to select a step size $\alpha_k$ that decreases the model $\hat f_k$ by a sufficient 
amount~(see condition \eqref{eq:suff:cond3}). This condition is instrumental in obtaining good complexity 
guarantees.

Aside from the use of subsampling, Algorithm~\ref{alg:alas} differs from the original method of~\cite{royer2018complexity} in two major ways. First, we only consider three types of search direction, as opposed to five in the original method of~\cite{royer2018complexity}. Indeed, in order to simplify 
the upcoming theoretical analysis, we do not allow for selecting gradient-based steps, i.e. steps that 
are colinear with the negative (subsampled) gradient. Note that such steps do not affect the 
complexity guarantees, but have a practical value since they are cheaper to compute. For this reason, we 
present the algorithm without their use, but we will re-introduce them in our numerical experiments.
\revised{Secondly, our strategy for choosing among the three different forms for the direction differs
slightly from~\cite{royer2018complexity} in the ``if'' condition in Step 4 of the algorithm and in the regularization used for the regularized Newton, see equation~\eqref{eq:regnewton} . In Section~\ref{sec:cvwcc}, we will show that such algorithmic modifications will lead also to  results that are equivalent to those of the original deterministic method~\cite{royer2018complexity}.} Nevertheless, and for the reasons already mentioned in the first point, we will also revert to the original rule in our practical implementation (see Section~\ref{subsec:implement}).

\begin{algorithm}[h!]
	\SetAlgoLined
	\DontPrintSemicolon 
	\BlankLine
	\textbf{Initialization}: Choose $x_0 \in \real^n$, $\theta \in (0,1), \eta > 0$, $\epsilon>0$.\;
	\For{$k=0,1,...$}{
		\begin{enumerate}
			\item Draw a random sample set $\mathcal{S}_k \subset \{1,\dots,N\}$, and compute the 
			associated quantities $g_k:=g(x_k;\mathcal{S}_k), H_k:=H(x_k;\mathcal{S}_k)$. Form the 
			estimation $\hat f_k$ as a function of the variable $s$:
			\begin{equation} \label{eq:modelsamples}
				\hat f_k(x_k+s) := \hat f(x_k+s;\mathcal{S}_k).
				\vspace{-3ex}
			\end{equation}
			\item Compute $\lambda_k$ as the minimum eigenvalue of the Hessian estimate $H_k$.\\ 
			If $\lambda_k \ge -\ecurv$ and $\|g_k\|=0$ set $\alpha_k=0,\ d_k=0$ and go to Step 7.
			\item If $\lambda_k < -\ecurv$, 
                              compute a negative eigenvector $v_k$ such that
			\begin{equation} \label{eq:negcurv}
				H_k v_k = \lambda_k v_k,\ \|v_k\| = -\lambda_k,\ v_k^\top g_k \le 0,
				\vspace{-3ex}
			\end{equation}
			set $d_k=v_k$ and go to the line-search step.
             \item If $\lambda_k>\|g_k\|^{1/2}$, compute a Newton direction $d_k$ solution of
			\begin{equation} \label{eq:newton}
				H_k d = -g_k,
				\vspace{-3ex}
			\end{equation}
			go to the line-search step.
			\item If $d_k$ has not yet been chosen, compute it as a regularized Newton direction, 
			solution of
			\begin{equation} \label{eq:regnewton}
				\left(H_k+(\|g_k\|^{1/2}+\ecurv) \mathbbm{I}_n\right)d_k = -g_k,
\vspace{-3ex}
			\end{equation}
			and go to the line-search step.
			\item \textbf{Line-search step} Compute the minimum index $j_k$ such that the 
			step length\\ $\alpha_k:=\theta^{j_k}$ satisfies the decrease condition:
			\begin{equation} \label{eq:suff:cond3}
				\hat f_k(x_k+\alpha_k d_k) - \hat f_k(x_k) \; \le \; -\frac{\eta}{6}\alpha_k^3\|d_k\|^3.
				\vspace{-3ex}
			\end{equation}
			\item Set $x_{k+1}=x_k+\alpha_k d_k$.
			\item Set $k=k+1$.
		\end{enumerate}
	} 
\caption{A Line-search Algorithm based on Subsampling (ALAS). \label{alg:alas}}
\end{algorithm}

Three comments about the description of Algorithm~\ref{alg:alas} are in order. First, we observe that the 
method as stated is not equipped with a stopping criterion. Apart from budget considerations, one might 
be tempted to stop the method when the derivatives are suggesting that it is a second-order stationary 
point. However, since we only have access to subsampled versions of those derivatives, it is possible 
that we have not reached a stationary point for the true function. As we will establish later in the 
paper, one needs to take into account the accuracy of the model, and it might take several iterations to 
guarantee that we are indeed  at a stationary point. We discuss the link between stopping criteria and 
stationarity conditions in Section~\ref{sec:cvwcc}.

Our second remark relates to the computation of a step. When the subsampled gradient $g_k$ is zero and 
the subsampled matrix $H_k$ is positive definite, we cannot compute a descent step using first- or 
second-order information, because the current iterate is second-order stationary for the subsampled model. 
In that situation, and for the reasons mentioned in the previous paragraph, we do not stop our method, but rather take a zero 
step and move on to a new iteration and a new sample set. After a certain number of such 
iterations, one can guarantee that a stationary point has been reached with 
high probability (see Section~\ref{sec:cvwcc}). 

\revised{
The third remark is about the size of the random sample that we did not specify. The algorithm supports adaptive sample size. For the sake  of simplicity and clarity, in our theoretical analysis (see Theorem 1)  we focus on using a lower bound which is independent from $k$. This bound on the sample size depends on $\epsilon$ but is limited by the full sample size (see Condition  (\ref{eq:samplesizebounddec})).
One may derive a sharper lower bound on the sample size, but such a bound will 
involve random, iteration-dependent quantities, which introduces a number of 
measurability issues. We thus chose to focus on constant bounds on the sample 
size.
 }
\section{Expected decrease guarantees with subsampling} 
\label{sec:expdecrease}
In this section, we derive bounds on the amount of expected decrease at each iteration. When the 
current sample leads to good approximations of the objective and derivatives, we are able to guarantee 
decrease in the function for any possible step taken by Algorithm~\ref{alg:alas}. By controlling the 
sample size, one can adjust the probability of having a sufficiently good model, so that the guaranteed 
decrease for good approximations will compensate a possible increase for bad approximations on average.
\subsection{Preliminary assumptions and definitions} 
\label{subsec:expdecrease:assum}
Throughout the paper, we will study Algorithm~\ref{alg:alas} under the following assumptions.

\begin{assumption}\label{as:f:lower}
The function $f$ is bounded below by $f_{\low} \in \real$.
\end{assumption}

\begin{assumption}\label{as:f:H}
The functions $f_i$ are twice continuously differentiable, with Lipschitz continuous 
gradients and Hessians, of respective Lipschitz constants $L_{i}$ and $L_{H,i}$.
\end{assumption}
A consequence of Assumption~\ref{as:f:H} is that $f$ is twice continuously differentiable, Lipschitz, 
with Lipschitz continuous first and second-order derivatives. This property also holds for 
$m(\cdot;\mathcal{S})$, regardless of the value of $\mathcal{S}$ (we say that the property holds 
\emph{for all realizations of $\mathcal{S}$}). In what follows, we will always consider that 
$m(\cdot;\mathcal{S})$ and $f$ have $L$-Lipschitz continuous 
gradients and $L_H$-Lipschitz continuous Hessians, where
$L:=\max_i L_i$ and $L_H:=\max_i L_{H,i}$.
Another corollary of Assumption~\ref{as:f:H} is that there exists a finite 
positive constant $U_H$ such that
$
U_H \ge \max_{i=1,\dots,N} \|\nabla^2 f_i(x_k)\|
$
for every $k$.
The value $L=U_H$ is a valid one, however we make the distinction between the two by analogy 
with previous works~\cite{royer2018complexity,xu2019newton}.
In the same spirit, we make the following additional assumption.
\begin{assumption}\label{as:finiteSGvar}
There exists a finite positive constant $U_g$ such that,
\[
	U_g \ge \max_{i=1,\dots,N} \|\nabla f_i(x_k)\| \quad \forall k,
\]
for all realizations of the algorithm.
\end{assumption}

\revised{
Assumption~\ref{as:finiteSGvar} guarantees that every trial step will be bounded in norm, and that the 
possible increase of $f$ produced by this step will also be bounded. Note that such assumption is less restrictive than assuming that all $f_i$'s are 
Lipschitz continuous, or that there exists a compact set that contains the sequence of iterates.
The satisfaction of such assumption can be ensured in practice by restarting the method whenever the algorithm detects unboundedness of the iterates, which never occurred in our numerical experiments.
Besides, because our sampling set is finite, the set of possibilities for each iterate is bounded (though it grows at a combinatorial pace). 
}

In the rest of the analysis, for every iteration $k$, we let 
$\skn_k := \tfrac{|\mathcal{S}_k|}{N}$ 
denote the \emph{sample fraction} used at every iteration. Our objective is to identify threshold 
values on $\skn_k$ that lead to (expected) decrease in the objective.

Whenever the sample sets in Algorithm~\ref{alg:alas} are drawn at random, the subsampling process
introduces randomness in an iterative fashion at every iteration. As a result, Algorithm~\ref{alg:alas} 
results in a stochastic process $\{x_k,d_k, \alpha_k, g_k, H_k, \hat f_k(x_k), \hat f_k(x_k+\alpha_k d_k) \}$ (we 
point out that the sample fractions need not be random). To lighten the notation throughout the paper, 
we will use these notations for the random variables and their realizations. Most of our analysis will be 
concerned with random variables, but we will explicitly mention that realizations are considered when 
needed.
Our goal is to show that under certain conditions on the sequences $\{g_k\}$, $\{H_k\}$, $\{\hat f_k(x_k)\}$, 
$\{\hat f_k(x_k+\alpha_k d_k)\}$ the resulting stochastic process has desirable convergence properties in 
expectation.

Inspired by a number of definitions in the model-based literature for stochastic or subsampled 
methods~\cite{bandeira2014trproba,chen2018stochastic,larson2016stochtr,liu2017noisy}, we introduce a 
notion of sufficient accuracy for our model function and its derivatives.

\begin{definition}\label{def:accurate}
Given a realization of Algorithm~\ref{alg:alas} and an iteration index $k$, the model 
$\hat f_k: \real^n \mapsto \real$ is said to be \emph{$(\delta_f,\delta_g, \delta_H)$-accurate} 
with respect to $(f,x_k,\alpha_k,d_k)$ when
\begin{equation} \label{eq:accurate:funcvals}
	|f(x_k)-\hat f_k(x_k)|\le \delta_f\quad \text{and} \quad |f(x_{k}+\alpha_k d_k) -\hat f_k(x_{k}+\alpha_k d_k)|\le \delta_f,
\end{equation}
\begin{equation} \label{eq:accurate:gradient}
	\|\nabla f(x_k)-\nabla \hat f_k(x_k) \|\le \delta_g \quad \text{and} \quad 
	\|\nabla f(x_{k}+\alpha_k d_k) - \nabla \hat f_k (x_{k}+\alpha_k d_k)\|\le \delta_g,
\end{equation}
\begin{equation} \label{eq:accurate:hessian}
\|\nabla^2 f(x_k)-\nabla^2 \hat f_k(x_k) \|\le \delta_H,
\end{equation} 
where 
$\delta_f,\delta_g,\delta_H$ are 
nonnegative constants.
\end{definition}

Condition~\eqref{eq:accurate:funcvals} is instrumental in establishing decrease guarantees for 
our method, while conditions~\eqref{eq:accurate:gradient} and~\eqref{eq:accurate:hessian} 
play a key role in defining proper notions of stationarity (see Section~\ref{sec:cvwcc}). 
Since we are operating with a sequence of random samples and models, we need a probabilistic equivalent 
of Definition~\ref{def:accurate}, which is given below.

\begin{definition}\label{def:accurate:proba}
Let $p \in (0,1]$, $\delta_f \ge 0$, $\delta_g \ge 0$ and $\delta_H \ge 0$. A sequence of functions 
$\{\hat f_k\}_k$ is called $p$-probabilistically  $(\delta_f,\delta_g, \delta_H)$-accurate for 
Algorithm~\ref{alg:alas} if the events 
$$
	I_k := \left\{\mbox{$\hat f_k$ is $(\delta_f,\delta_g, \delta_H)$-accurate with respect to 
	$(f,x_k,\alpha_k,d_k)$}\right\} \text{ satisfy  }
p_k := \Pr\left(I_k | \mathcal{F}_{k-1}\right) \ge p,
$$
where $\mathcal{F}_{k-1}$ is the $\sigma$-algebra generated by the sample sets 
$\mathcal{S}_0,\mathcal{S}_1, \ldots,\mathcal{S}_{k-1}$, and we define \\
$\Pr(I_0 | \mathcal{F}_{-1}) := \Pr(I_0)$.
\end{definition}

Observe that if we sample the full data at every iteration (that is, $\mathcal{S}_k=\{1,\dots,N\}$ for 
all $k$), the resulting model sequence satisfies the above definition for any $p \in [0,1]$ and any 
positive values $\delta_f,\delta_g,\delta_H$. Given our choice of model~\eqref{eq:subsampmodel}, the 
accuracy properties are directly related to the random sampling set, but we will follow the existing 
literature on stochastic optimization and talk about accuracy of the models. We will however express 
conditions for good convergence behavior based on the sample size rather than the probability of 
accuracy.

In the rest of the paper, we assume that the estimate functions of the problem form a probabilistically 
accurate sequence as follows.

\begin{assumption}\label{as:accurate:in:prob} 
The sequence $\{\hat f_k\}_k$ produced by Algorithm \ref{alg:alas} is $p$-probabilistically 
$\delta$-accurate, with $\delta:=(\delta_f,\delta_g, \delta_H)$ and $p \in (0,1]$. 
\end{assumption}
%

We now introduce the two notions of stationarity that will be considered in our analysis.

\begin{definition} \label{def:stationary}
Consider a realization of Algorithm~\ref{alg:alas}, and let $\epsilon_g,\epsilon_H$ be 
two positive tolerances. We say that the $k$-th iterate $x_k$ is 
$(\epsilon_g,\epsilon_H)$-\emph{model stationary} if
	\begin{equation} \label{eq:accuratenonstatiomodel}
		\min\left\{ \|g_k\|,\|g(x_{k+1},\mathcal{S}_k)\| \right\} \le \epsilon_g \quad 
		\mbox{and} \quad \lambda_k \ge -\epsilon_H.
	\end{equation}
Similarly, we will say that $x_k$ is $(\epsilon_g,\epsilon_H)$-\emph{function 
stationary} if 
	\begin{equation} \label{eq:accuratenonstatiocondf}
		\min\left\{\|\nabla f(x_k)\|,\|\nabla f(x_{k+1})\|\right\} \le \epsilon_g \quad 
		\mathrm{or} \quad \lambda_{\min}(\nabla^2 f(x_k)) \ge -\epsilon_H.
	\end{equation}
\end{definition}

Note that the two definitions above are equivalent whenever the model consists of the full function, 
i.e. when $\mathcal{S}_k=\{1,\dots,n\}$ for all $k$. 
We also observe that the definition of model stationarity involves the norm of the vector
\begin{equation} \label{eq:nextgraditk}
	g_k^+ \; := g(x_k+\alpha_k d_k ; \mathcal{S}_k).
\end{equation}
The norm of this ``next gradient" is a major tool for the derivation of complexity results in 
Newton-type methods~\cite{cartis2011arccomplexity,royer2018complexity}. In a subsampled setting, a 
distinction between $g_k^+$ and $g_{k+1}$ is necessary because these two vectors are 
computed using different samples.

Our objective is to guarantee convergence towards a point satisfying a function stationarity 
property~\eqref{eq:accuratenonstatiocondf}, yet we will only have control on achieving model 
stationarity. The accuracy of the models will be instrumental in relating the two properties, 
as shown by the lemma below.

\begin{lemma} \label{lemma:modeltruestationarity}
Let Assumptions \ref{as:f:lower} and~\ref{as:f:H} hold. Consider a realization of the method 
that reaches an iterate $x_k$ such that $x_k$ is $(\epsg,\epsH)$-model stationary.
Suppose further that the model $\hat f_k$ is $(\delta_f,\delta_g,\delta_H)$-accurate with 
\begin{equation} \label{eq:conddeltagdeltah}
	\delta_g \le \ccg\epsg \quad \mbox{and} \quad \delta_H \le \cch\epsH
\end{equation}
where $\ccg$ and $\cch$ are positive, deterministic constants independent of $k$. Then, $x_k$ is a\\ 
$((1+\ccg)\epsg,(1+\cch)\epsH)$-function stationary point.
\end{lemma}

\proof{Proof of Lemma~\ref{lemma:modeltruestationarity}.}
Let $x_k$ be an iterate such that $
	\min\{\|g_k\|,\|g_k^+\|\} \le \epsg  \quad \mathrm{and} \quad \lambda_k \ge -\epsH$.
Looking at the first property, suppose that $\|g_k\| \le \epsg$. In that case, we have:
\begin{equation*}
	\|\nabla f(x_k)\|  \le \|\nabla f(x_k) - g_k\| + \| g_k\|  
	\le  \delta_g +  \epsg  \le (\ccg+1) \epsg.
\end{equation*}
A similar reasoning shows that if $\|g_k^+\| \le \epsg$, we obtain $\|\nabla f(x_{k+1})\| \le 
 (\ccg+1) \epsg$; thus, we must have 
\[
	\min\left\{ \|\nabla f(x_k)\|,\|\nabla f(x_{k+1})\|\right\} \le (1+\ccg)\epsg.
\]
Consider now a unit eigenvector $v$ for $\nabla^2f(x_k)$ associated with 
$\lambda_{\min}(\nabla^2 f(x_k))$, one has
\[
\lambda_k - \lambda_{\min}(\nabla^2 f(x_k)) \le v^\top H_k v - v^\top \nabla^2 f(x_k) v
\le \|H_k - \nabla^2 f(x_k)\| \|v\|^2 \le \delta_H.
\]
Hence, by using \eqref{eq:conddeltagdeltah}, one gets
\begin{equation*}
	\lambda_{\min}(\nabla^2 f(x_k))  = \left( \lambda_{\min}(\nabla^2 f(x_k)) - \lambda_k\right)  
	+ \lambda_k  \ge  -\delta_H - \epsH \ge -(\cch+1) \epsH.
\end{equation*}
Overall, we have shown that
\[
	\min\left\{\|\nabla f(x_k)\|,\|\nabla f(x_{k+1})\|\right\} \le (1+\ccg)\epsg \quad 
	\mbox{and} \quad \lambda_{\min}(\nabla^2 f(x_k))  \ge -(1+\cch)\epsH,
\]
and thus $x_k$ is also a $((1+\ccg)\epsg,(1+\cch)\epsH)$-function stationary point.
\endproof

The reciprocal result of Lemma~\ref{lemma:modeltruestationarity}, which can be proven in ac 
similar fashion will also be of interest to us.
\begin{lemma} \label{lm:truemodelnonstationary}
	Consider a realization of Algorithm~\ref{alg:alas} and the associated $k$-th iteration. 
	Suppose that $x_k$ is \emph{not} $\left((1+\ccg)\epsg,(1+\cch)\epsH\right)$-function 
	stationary, and that the model $\hat f_k$ is $\delta=(\delta_f,\delta_g,\delta_H)$-accurate 
	with $\delta_g \le \ccg \epsg$ and $\delta_H \le \cch \epsH$ where $\ccg$ and $\cch$ are positive constants. Then, $x_k$ is \emph{not} $(\epsg,\epsH)$-model stationary.
\end{lemma}
\subsection{A general expected decrease result} 
\label{subsec:expdecrease:generalres}
In this section, we study the guarantees that can be obtained 
(in expectation) for the various types of direction considered by our method. 
By doing so, we identify the necessary requirements on our sampling 
procedure, as well as on our accuracy threshold for the model values.

In what follows, we will make use of the following constants:
\begin{equation*}
c_{nc} := \frac{3\theta}{L_H+\eta},~c_{n} :=  \min\left\{\left[\frac{2}{L_H}\right]^{1/2},\left[\frac{3\theta}{L_H+\eta}\right]\right\},~ c_{rn} := \min\left\{\frac{1}{1+\sqrt{1+L_H/2}},\left[\frac{6\theta}{L_H+\eta}\right]\right\}, 
\end{equation*}
\begin{equation*}
\bar{j}_{nc}:= \left[\log_{\theta} \left( \frac{3}{L_H+\eta}\right) \right]_{+},~
\bar{j}_{n}:= \left[\log_{\theta} \left( \sqrt{\frac{3}{L_H+\eta}} 
\frac{\epsilon^{1/2}}{\sqrt{U_g}}\right) \right]_{+},~\bar{j}_{rn}:= 
\left[\log_{\theta} \left( \frac{6}{L_H+\eta}\frac{\epsilon}{U_g}\right)\right]_{+},
\end{equation*}
where $\epsilon$ is the tolerance used in Algorithm~\ref{alg:alas}.
Those constants are related to the line-search steps that can be performed at 
every iteration of Algorithm~\ref{alg:alas}. As long as the current iterate is not an approximate 
stationary point of the model, we can bound the number of such steps independently of $k$. To 
formalize this property, we introduce the following events for any $k \in \mathbb{N}$:
\revised{
\begin{equation}\label{eq:nonmodelstatevents}
E^1_k:=\left\{\|g_k\| > \epsilon\right\},\quad E^+_k:=\left\{\|g^+_k\| > \epsilon\right\}, \quad E^2_k:=\left\{\lambda_k < -\epsilon^{1/2} \right\}.
\end{equation}
We also define
\begin{equation}\label{eq:nonfuncstatevents}
\mathcal{E}^1_k:=\left\{\|\nabla f(x_k)\| > (1+\ccg)\epsilon\right\}  ~ \mbox{and}~ \mathcal{E}^2_k:=\left\{\lambda_{\min}(\nabla^2 f(x_k)) < -(1+\cch)\epsilon^{1/2} \right\}.
\end{equation}
By Lemma~\ref{lemma:modeltruestationarity}
i.e., under the event $I_k$, it holds that $E^1_k\cap E^2_k$ imply $\mathcal{E}^1_k\cap\mathcal{E}^2_k$.
}

\revised{Assuming event $(E^1_k\cap E^+_{k}) \cup E^2_k$ occurs, we can provide a lower bound on the step returned by the 
line-search process. This is the purpose of the following lemma. Its proof mainly relies on arguments from the deterministic case~\cite{royer2018complexity} and is provided in the appendix.}

\begin{lemma} \label{lm:bound:alpha_kd_k}
	Let Assumptions \ref{as:f:lower} and \ref{as:f:H} hold for a 
	realization of Algorithm~\ref{alg:alas}. Consider an iteration $k$ such 
	that \revised{$(E^1_k \cap E^+_{k}) \cup E^2_k$ occurs.} 
        Then, the backtracking line search terminates 
	with the step length $\alpha_k =\theta^{j_k}$, with 
	$j_k \le \bar{j} +1$ and
	\begin{equation} \label{bound:alpha_kd_k}
		\alpha_k \|d_k\| \; \ge \;  c \;\epsilon^{1/2},
	\end{equation}
	where $c := \min\{c_{nc},c_n,c_{rn}\}$ and 
	$\bar{j}  := \max\{\bar{j} _{nc},\bar{j} _n,\bar{j}_{rn}\}$.
\end{lemma}

\revised{Since we are using estimates of the true function and its derivatives, the decrease guaranteed by Lemma~\ref{lm:bound:alpha_kd_k} may not reflect on the true function values. 
To address this issue, we provide below a deterministic upper bound on the 
norm of any step computed by our method. The proof is again to be found in the appendix.}

\begin{lemma}\label{lem:bound_on_dk}
Let Assumption~\ref{as:f:H} hold for a realization of 
Algorithm~\ref{alg:alas}. Then, for any index $k$,
\begin{equation}\label{eq:boundonde}
\|d_k\|\le \max\{U_H,U_g^{1/2}\}.
\end{equation}
\end{lemma}
\medskip

\revised{
We will now state and prove our main result on expected decrease of our method by combining the 
results of Lemmas~\ref{lm:bound:alpha_kd_k} and \ref{lem:bound_on_dk}. To this end, 
we introduce the following function on $[0,\infty) \times [0,1]$:}
\revised{
\begin{equation} \label{eq:rhofunction}
	\pfun(t,q) \; := \;		\max\left\{0, 1-\frac{q \frac{\eta t^3}{12}}{2(1-q) U_L}\right\},\quad 
	\mbox{where\ $U_L:=U_g\max\{U_H,U_g^{1/2}\}+\frac{L}{2}\max\{\revised{U^2_H},U_g\}$}.
\end{equation} 
}
\revised{With the convention that $\pfun(t,1)=0\ \forall t \ge 0$, the function $\pfun$ 
is well-defined} with values in $[0,1]$, and 
decreasing in its first and second arguments.

%
\revised{
We are now ready to establish a guarantee of expected decrease.}
To this end, define $T_{\epsilon}$ as the first iteration index $k$ of 
Algorithm~\ref{alg:alas} for which 
\[
	\min\{\|\nabla f(x_k)\|,\|\nabla f(x_{k+1})\|\} \le (1+\ccg)\epsilon  \quad \mbox{and} \quad 
	\lambda_{\min}\left(\nabla^2 f(x_k)\right) \ge -(1+\cch)\epsilon^{1/2}.
\]
We shall calculate the expected decrease for any iteration $k$ such that $T_\epsilon>k$. 


\begin{theorem}\label{th:gendecrease}
Let Assumptions \ref{as:f:lower} and \ref{as:f:H} hold. Suppose also that 
Assumption~\ref{as:accurate:in:prob} holds with $\delta=(\delta_f,\delta_g,\delta_H)$ satisfying
	\begin{equation} \label{eq:bounddeltaepsgendecrease}
		\delta_f \le \frac{\eta}{24}c^3\epsilon^{3/2},\quad \delta_g \le \ccg\epsilon, \quad 
		\delta_H \le \cch \epsilon^{1/2}
	\end{equation}
where $\epsilon>0$, $\ccg \in (0,1)$, $\cch \in (0,1)$ and $c$ is chosen as in 
Lemma~\ref{lm:bound:alpha_kd_k}. 
Then, if the sample fraction $\skn_k$ is chosen such that
	\begin{equation} \label{eq:samplesizebounddec}
		\skn_k \; \ge \; \pfun(c\epsilon^{1/2},p),
	\end{equation}
where $\pfun$ is given by~\eqref{eq:rhofunction}, then
	\revised{
		\begin{eqnarray} \label{eq:gendecrease2}
	& &	\mathbb{E}\left[f(x_k+\alpha_k d_k) - f(x_k)\ |\  \mathcal{F}_{k-1}, \mathcal{E}^1_k\cup\mathcal{E}^2_k\right]
		\nonumber \\ 
	&  &	\le -p \frac{\eta c^3}{24}\epsilon^{3/2} \mathbb{P}\left(E^+_{k}
		|\ \mathcal{F}_{k-1},E^{1}_k\cup E_k^{2},I_k\right)
		+{\small \min\left\{0,\frac{2-p}{1-p}\right\}}\frac{\eta c^3}{24}\epsilon^{3/2}  \mathbb{P}\left(\overline{E^+_{k}} |
		\mathcal{F}_{k-1},E^{1}_k\cup E_k^{2},I_k\right).
	\end{eqnarray}
}
\end{theorem}
\proof{Proof of Theorem~\ref{th:gendecrease}.}	
	By definition, one has that:
	\revised{
		{\small{
	\begin{eqnarray} \label{eq:good_bad_combined}
		& & 
		\E{f(x_{k+1}) - f(x_k)\ |\ \mathcal{F}_{k-1},\mathcal{E}^{1}_k\cup \mathcal{E}_k^{2}}  \nonumber \\ 
		&  &=  \tilde p_k \E{f(x_{k+1})- f(x_k)\ |\ \mathcal{F}_{k-1},\mathcal{E}^{1}_k\cup \mathcal{E}_k^{2},I_k}  + 
		 (1- \tilde p_k) \E{f(x_{k+1})- f(x_k)\ |\ \mathcal{F}_{k-1},\mathcal{E}^{1}_k\cup \mathcal{E}_k^{2},\overline{I_k}\,} \nonumber \\ 
		&  &=  \tilde p_k \E{f(x_{k+1})- f(x_k)\ |\ \mathcal{F}_{k-1},E^{1}_k\cup E_k^{2},I_k}  + 
		 (1- \tilde p_k) \E{f(x_{k+1})- f(x_k)\ |\ \mathcal{F}_{k-1},\mathcal{E}^{1}_k\cup \mathcal{E}_k^{2},\overline{I_k}\,} 
	\end{eqnarray}}}
	}
	in which $I_k$ is the event corresponding to the model being $\delta$-accurate and 
	\revised{$\tilde p_k :=\Pr(I_k|\mathcal{F}_{k-1},\mathcal{E}^{1}_k\cup \mathcal{E}_k^{2})$, and we have used the equivalence of $E^1_k\cup E^2_k$ and
        $\mathcal{E}^1_k\cup\mathcal{E}^2_k$ under $I_k$.}
     \revised{
      Recall that $\mathcal{E}^{1}_k$ and $\mathcal{E}^2_k$ only depend on the history of the algorithm prior to iteration $k-1$, thus the $\sigma$-algebra generated by $\mathcal{E}^{1}_k\cup \mathcal{E}_k^{2}$ is included in $ \mathcal{F}_{k-1}$.
    }
    We can then apply a result from probability theory~\cite[Theorem 5.1.6]{durrett2010probatheory} 
    stating that for any random variable $X$ (not necessarily belonging to $\mathcal{F}_{k-1}$),
        we have: 
	\revised{
	\begin{equation} \label{eq:inclusionsigalg}
		\E{X |\mathcal{F}_{k-1}, \mathcal{E}^{1}_k\cup \mathcal{E}_k^{2} } = 
		\E{\E{X|\mathcal{F}_{k-1}} | \mathcal{F}_{k-1}, \mathcal{E}^{1}_k\cup \mathcal{E}_k^{2}}
	\end{equation}
	}
	Therefore, denoting by $\mathbf{1}(I_k)$ the indicator variable of the random event $I_k$, we have
	\revised{
	{\small{
	\begin{eqnarray*}
		\tilde p_k= \Pr\left( I_k | \mathcal{F}_{k-1}, \mathcal{E}^{1}_k\cup \mathcal{E}_k^{2} \right) 
		&= &\E{ \mathbf{1}(I_k) |\mathcal{F}_{k-1}, \mathcal{E}^{1}_k\cup \mathcal{E}_k^{2} } 
		= \E{\E{\mathbf{1}(I_k) |\mathcal{F}_{k-1}} | \mathcal{F}_{k-1} , \mathcal{E}^{1}_k\cup \mathcal{E}_k^{2}}  \\
		&=& \E{ \Pr(I_k | \mathcal{F}_{k-1}) | \mathcal{F}_{k-1} , \mathcal{E}^{1}_k\cup \mathcal{E}_k^{2}} 
		= \E{ p_k | \mathcal{F}_{k-1}, \mathcal{E}^{1}_k\cup \mathcal{E}_k^{2}} =p_k \ge p,
	\end{eqnarray*}
	}}
	where the last equality comes from $p_k \in \mathcal{F}_{k-1}$.
	}
	
	\revised{
	It thus suffices to bound the two terms in \eqref{eq:good_bad_combined} independently to bound 
	the expected change in the function values. 
	We begin by the term corresponding to the occurrence of $I_k$ in \eqref{eq:good_bad_combined}, 
	and use the following decomposition:
	\begin{eqnarray*}
		& &\E{f(x_{k+1})- f(x_k)\ |\ \mathcal{F}_{k-1},E^{1}_k\cup E_k^{2},I_k} 
		\\
		&= & \E{f(x_{k+1})- f(x_k)\ |\ \mathcal{F}_{k-1},E^{1}_k\cup E_k^{2},I_k,E_{k}^+} 
		\mathbb{P}\left(E^+_{k} | \mathcal{F}_{k-1},E^{1}_k\cup E_k^{2},I_k\right) \\
		&& +\E{f(x_{k+1})- f(x_k)\ |\ \mathcal{F}_{k-1},E^{1}_k\cup E_k^{2},I_k, 
		\overline{E_{k}^+}}\,
		\mathbb{P}\left(\overline{E_k^{+}} | \mathcal{F}_{k-1},E^{1}_k\cup E_k^{2},I_k\right). 
	\end{eqnarray*}
	}
	\revised{When the event $E_k^{\sta}= (E^{1}_k\cap E_k^{+}) \cup E_{k}^2\cup I_k$ occurs, we can bound the first term
        of the right hand side as follows:
	\begin{eqnarray} 
		f(x_k+\alpha_k d_k)-f(x_k)
		&  =&f(x_k+\alpha_k d_k)-\hat f_k(x_k+\alpha_k d_k)+\hat f_k(x_k+\alpha_k d_k)-\hat f_k(x_k)+\hat f_k(x_k)-f(x_k) 
		\nonumber \\ 
		&\le  & 2\delta_f + \hat f_k(x_k+\alpha_k d_k)-\hat f_k(x_k) \nonumber \\
		&\le  & 2\delta_f -\frac{\eta}{6}\alpha_k^3\|d_k\|^3 \nonumber \\
		& \le & \frac{\eta c^3}{12}\epsilon^{3/2}-\frac{\eta}{6}\alpha_k^3\|d_k\|^3
		\label{eq:gendecproofIk:delta}\\
		&\le &-\frac{\eta}{12}\alpha_k^3\|d_k\|^3 \le -\frac{\eta c^3}{12}\epsilon^{3/2}, 
		\label{eq:gendecproofIk:adk}
	\end{eqnarray}
	%
	where~\eqref{eq:gendecproofIk:delta} comes from the 
	bound~\eqref{eq:bounddeltaepsgendecrease} on $\delta_f$, 
	and~\eqref{eq:gendecproofIk:adk} follows from Lemma~\ref{lm:bound:alpha_kd_k}.
	}
	
	\revised{
	If we now condition on $E^{1}_k\cup E_k^{2},I_k$ and consider that 
	$E_{k}^+$ does not occur, we can no longer apply
	Lemma~\ref{lm:bound:alpha_kd_k}, but the derivation up 
	to~\eqref{eq:gendecproofIk:delta} still holds since the method does not 
	stop at iteration $k$.
	}
	\revised{
	Overall, we obtain 
	\begin{eqnarray}\label{eq:gendecproofIk}
		& &\E{f(x_{k+1})- f(x_k)\ |\ \mathcal{F}_{k-1},E^{1}_k\cup E_k^{2},I_k} 
		\nonumber \\
		&\le &-\frac{\eta c^3}{12}\epsilon^{3/2} \mathbb{P}\left(E^+_{k} | 
		\mathcal{F}_{k-1},E^{1}_k\cup E_k^{2},I_k\right)
		+ \frac{\eta c^3}{12}\epsilon^{3/2} \mathbb{P}\left(\overline{E^+_{k}} |
		\mathcal{F}_{k-1},E^{1}_k\cup E_k^{2},I_k\right).
	\end{eqnarray}
	}	
	We now turn to the second case in (\ref{eq:good_bad_combined}), for which we exploit the 
	following decomposition:
	\begin{equation*}
		f(x_k+\alpha_k d_k) - f(x_k) = \skn_k\left(\hat f_k(x_k+\alpha_k d_k) - \hat f_k(x_k)\right) 
		+ \left(1-\skn_k\right)
		\left(f_{\mathcal{S}^c_k}(x_k+\alpha_k d_k) - f_{\mathcal{S}^c_k}(x_k) \right),
	\end{equation*}
	with $f_{\mathcal{S}^c_k} = \tfrac{1}{N-|\mathcal{S}_k|} \sum_{i \notin \mathcal{S}_k} f_i$.	
	Using the decrease condition~\eqref{eq:suff:cond3} to bound the first term, and a first-order 
	Taylor expansion to bound the second term, we obtain:
	\begin{eqnarray*} 
	f(x_k+\alpha_k d_k) - f(x_k) &\le &  -\skn_k\frac{\eta}{6}\alpha_k^3 \|d_k\|^3 + 
		\left(1-\skn_k\right) \frac{1}{N-|\mathcal{S}_k|}
		\sum_{i \notin \mathcal{S}_k} \left\{\alpha_k \nabla f_i(x_k)^\top d_k 
		+ \frac{L_{i}}{2}\alpha_k^2 \|d_k\|^2\right\} \nonumber \\ 
		& \le& -\skn_k\frac{\eta}{6}\alpha_k^3 \|d_k\|^3 +  \left(1-\skn_k\right) 
		\max_{i \notin \mathcal{S}_k} \left\{\alpha_k \nabla f_i(x_k)^\top d_k 
		+ \frac{L_{i}}{2}\alpha_k^2 \|d_k\|^2\right\} \nonumber \\
		& \le& -\skn_k\frac{\eta}{6}\alpha_k^3 \|d_k\|^3 +  \left(1-\skn_k\right) 
		\left\{\alpha_k \|\nabla f_{i_k}(x_k)\|\|d_k\| + \frac{L_{i_k}}{2}\alpha_k^2 \|d_k\|^2\right\},
	\end{eqnarray*}
	where $i_k \in \arg\max_{i \notin \mathcal{S}_k} \left\{\alpha_k \nabla f_i(x_k)^\top d_k + 
	\frac{L_{i}}{2}\alpha_k^2 \|d_k\|^2\right\}$ and we used the Lipschitz continuity assumption 
	on the functions $f_i$'s. Introducing the constants $U_g$ and $L$ to remove the dependencies 
	on $i_k$, we further obtain:		
	\begin{eqnarray*}
		f(x_k+\alpha_k d_k) - f(x_k) & \le & 
		-\skn_k\frac{\eta}{6}\alpha_k^3 \|d_k\|^3 +  \left(1-\skn_k\right) 
		 \left\{ U_g \alpha_k \|d_k\| + \frac{L}{2}\alpha_k^2 \|d_k\|^2\right\}.
	\end{eqnarray*}
	Using now Lemma~\ref{lem:bound_on_dk}, we arrive at:
	\begin{eqnarray*}
		f(x_k+\alpha_k d_k) - f(x_k) & \le &  \left(1-\skn_k\right) 
		 \left\{U_g\max\{U_H,U_g^{1/2}\}+\frac{L}{2}\max\{U^2_h,U_g\}\right\} =\left(1-\skn_k\right) U_L,
	\end{eqnarray*}
	where $U_L$ is defined as in~\eqref{eq:rhofunction}, and we use~\eqref{eq:samplesizebounddec}
	to obtain the last inequality.

	 \revised{
	 	Putting all cases together yields:
	\begin{eqnarray}
	\label{eq:finalth1}
	& &\mathbb{E}\left[f(x_k+\alpha_k d_k) - f(x_k) | \mathcal{F}_{k-1},\mathcal{E}^{1}_k\cup \mathcal{E}_k^{2}\right] \nonumber \\
	&\le &-\tilde{p}_k\frac{\eta c^3}{12}\epsilon^{3/2} \mathbb{P}\left(E^+_{k} | 
		\mathcal{F}_{k-1},E^{1}_k\cup E_k^{2},I_k\right)
		+ \tilde{p}_k \frac{\eta c^3}{12}\epsilon^{3/2} \mathbb{P}\left(\overline{E^+_{k}} |
		\mathcal{F}_{k-1},E^{1}_k\cup E_k^{2},I_k\right)
		+ (1-\tilde p_k) \left(1-\pi_k \right) U_L \nonumber \\
	& = & \left[-\tilde p_k \frac{\eta c^3}{12}\epsilon^{3/2}  + 
	(1-\tilde p_k) \left(1-\pi_k\right) U_L\right]\mathbb{P}\left(E^+_{k} | 
		\mathcal{F}_{k-1},E^{1}_k\cup E_k^{2},I_k\right)  \nonumber \\ 
	& &~~+
	\tilde{p}_k \frac{\eta c^3}{12}\epsilon^{3/2} \mathbb{P}\left(\overline{E^+_{k}} |
		\mathcal{F}_{k-1},E^{1}_k\cup E_k^{2},I_k\right) +
	(1-\tilde p_k) \left(1-\pi_k\right) U_L (1-\mathbb{P}\left(E^+_{k} | 
		\mathcal{F}_{k-1},E^{1}_k\cup E_k^{2},I_k\right) )   \nonumber\\
	& \le & \left[-p\frac{\eta c^3}{12}\epsilon^{3/2}  + 
	(1-p) \left(1-\pi_k\right) U_L\right]\mathbb{P}\left(E^+_{k} | 
		\mathcal{F}_{k-1},E^{1}_k\cup E_k^{2},I_k\right)  \nonumber\\ 
	& &~~+
	\frac{\eta c^3}{12}\epsilon^{3/2} \mathbb{P}\left(\overline{E^+_{k}} |
		\mathcal{F}_{k-1},E^{1}_k\cup E_k^{2},I_k\right) +
	\left(1-\pi_k\right) U_L (1-\mathbb{P}\left(E^+_{k} | 
		\mathcal{F}_{k-1},E^{1}_k\cup E_k^{2},I_k\right) ) \nonumber\\ 
			& \le & \left[-p\frac{\eta c^3}{12}\epsilon^{3/2}  + 
	(1-p) \left(1-\pi_k\right) U_L\right]\mathbb{P}\left(E^+_{k} | 
		\mathcal{F}_{k-1},E^{1}_k\cup E_k^{2},I_k\right)  \nonumber\\ 
			& &~~+
	\left[\frac{\eta c^3}{12}\epsilon^{3/2}+ \left(1-\pi_k\right) U_L \right] \mathbb{P}\left(\overline{E^+_{k}} |
		\mathcal{F}_{k-1},E^{1}_k\cup E_k^{2},I_k\right).
	\end{eqnarray} 
	}
	\revised{
	Finally, the bound~\eqref{eq:samplesizebounddec} guarantees that 
	\[
		1 - \pi_k \le \frac{p\eta c^3}{24(1-p)U_L}\epsilon^{3/2}
		\; \Longrightarrow \;
		 \left\{
	\begin{array}{ll}
				-p \frac{\eta c^3}{12}\epsilon^{3/2} + 
		(1-p) \left(1-\pi_k\right) U_L \le -p \frac{\eta c^3}{24}\epsilon^{3/2}; \\
		& \\
		\frac{\eta c^3}{12}\epsilon^{3/2}+ \left(1-\pi_k\right) U_L  \le \left(\frac{2-p}{1-p}\right)\frac{\eta c^3}{24}\epsilon^{3/2} .
	\end{array}
	\right.
	\]
	when $\pi_k < 1$ and $p<1$.
	Plugging these inequalities into~\eqref{eq:finalth1} leads to the desired result. The result trivially holds also for $\pi_k=1$ or $p=1$. 
	}

\endproof 

\revised{We now comment on the assumptions needed to establish 
Theorem~\ref{th:gendecrease}. First, we observe that the accuracy 
requirements~\eqref{eq:bounddeltaepsgendecrease} can arise as a direct 
consequence of $\pi_k$ being sufficiently high (see the online companion for a 
full illustration). However, it also encompasses the use of inexact values on 
top of sampling, one particular case being the use of approximate values and 
derivatives in a deterministic framework (i.e. $\pi_k=1$ and inexact values 
are used).
Secondly, we point out that the result of Theorem~\ref{th:gendecrease} is 
established under a uniform bound on the sampling fraction $\pi_k$: although 
this suggests to use a constant value for $\pi_k$ (which we do in our 
experiments), this does not preclude from using an adaptive value for 
$\pi_k$. Indeed, any value satisfying~\eqref{eq:samplesizebounddec} would suit 
our purpose, thus the choice of $\pi_k$ could be made adaptively by increasing 
the sample size. Moreover,  the proof of Theorem~\ref{th:gendecrease} helps in 
identifying iteration-dependent quantities that could be used as bounds on 
$\pi_k$, e.g. via~\eqref{eq:finalth1}. Still, such (random) quantities are  
challenging to estimate in practice, and it is unclear whether their 
manipulation in the upcoming convergence analysis can lead to complexity 
guarantees such as those presented in the next section: we thus elected to 
focus on a constant lower bound for the sample size.
}

\section{Global convergence rate and complexity analysis} 
\label{sec:cvwcc}

\revised{In this section, we build on Theorem~\ref{th:gendecrease} to derive a 
global rate of convergence towards an approximate stationary point in 
expectation for Algorithm~\ref{alg:alas}}. More precisely, we seek an 
$((1+\kappa_g)\epsilon,(1+\kappa_H)\epsilon^{1/2})$-function stationary point 
in the sense of Definition~\ref{def:stationary}, that is, an iterate 
satisfying:
\begin{equation} \label{eq:epspoint}
	\min\left\{\|\nabla f(x_k)\|,\|\nabla f(x_{k+1})\|\right\} \le 
	(1+\ccg)\epsilon \quad \mbox{and} \quad 
	\lambda_{\min}(\nabla^2 f(x_k)) \ge -(1+\cch) \epsilon^{1/2}.
\end{equation} 
Since our method only operates with a (subsampling) model of the objective, we are only able to 
check whether the current iterate is an $(\epsilon,\epsilon^{1/2})$-\emph{model} stationary point 
according to Definition~\ref{def:stationary}, i.e.  an iterate $x_k$ such that 
$
	\min\{\|g_k\|,\|g_k^+\|\} \le \epsilon \quad \mbox{and} \quad 
	\lambda_k \ge -\epsilon^{1/2}.
$
Compared to the general setting of Definition~\ref{def:stationary}, we are using 
$\epsilon_g=\epsilon_H^2=\epsilon$. This specific choice of first- and second-order tolerances has been 
observed to yield optimal complexity bounds for a number of algorithms, in the sense that the dependence 
on $\epsilon$ is minimal (see e.g.~\cite{carmon2018accelnonconvex,royer2018complexity}). The rules 
defining what kind of direction (negative curvature, Newton, etc) is chosen at every iteration of Algorithm~\ref{alg:alas} implicitly rely on this choice.

Our goal is to relate the model stationarity property~\eqref{eq:accuratenonstatiomodel} to its function 
stationarity counterpart~\eqref{eq:epspoint}. For this purpose, we first establish a general result 
regarding the expected number of iterations required to reach function stationarity. We then introduce a 
stopping criterion involving multiple consecutive iterations of model stationarity: with this criterion, 
our algorithm can be guaranteed to terminate at a function stationary point with high probability. 
Moreover, the expected number of iterations until this termination occurs is of the same order of 
magnitude as the expected number of iterations required to reach a stationary point.
\subsection{Expected iteration complexity} 
\label{subsec:cvwcc:iterwcc}
The proof of our expected complexity bound relies upon two arguments from martingales and stopping time 
theory. The first one is a martingale convergence result~\cite[Theorem 1]{robbins1971convergence}, which 
we adapt to our setting in Theorem~\ref{th:mart}.
\begin{theorem}\label{th:mart}
Let $(\Omega,\Sigma,\mathbb{P})$ be a probability space and $\{\Sigma^k\}_k$ be a 
sequence of sub-sigma algebras of $\Sigma$ such that $\Sigma^k \subset \Sigma^{k+1}$.
If $\zeta^k$ is a positively valued sequence of random variables on $\Sigma$, and if 
there exists a deterministic sequence $\nu^k\ge 0$ such that 
$
\mathbb{E}[\zeta^{k+1}|\Sigma^k] + \nu^k \le \zeta^k
$, then $\zeta^k$ converges to a $[0,\infty)$-valued random variable almost surely and 
$\sum_k \nu^k <\infty$.
\end{theorem}

At each iteration, Theorem~\ref{th:gendecrease} guarantees a certain expected decrease for 
the objective function. Theorem~\ref{th:mart} will be used to show that such a decrease cannot hold 
indefinitely if the objective is bounded from below.

The second argument comes from stopping time analysis (see, e.g., 
\cite[Theorem 6.4.1]{ross1996stochastic}) and is given in Theorem~\ref{th:stop}. The notations have 
been adapted to our setting.
\begin{theorem}\label{th:stop}
Let $T$ be a stopping time for the process $\{Z_k,k\ge 0\}$ and let $\bar Z_k=Z_k$ for $k\le T$ and 
$\bar Z_k=Z_T$ for $k>T$. If either one of the three properties hold: (i) $\bar Z_k$ is uniformly bounded; (ii) $T$ is bounded; or (iii) $\mathbb{E}[T]<\infty$ and there is an $R<\infty$ such that 
	$\mathbb{E}[|Z_{k+1}-Z_k|\,|\,Z_0,...,Z_k]<R$.
Then, $\E{\bar{Z}_k} \rightarrow \E{Z_T}$. 

Moreover, $\mathbb{E}[Z_T] \ge \mathbb{E}[Z_0]$ (resp. 
$\mathbb{E}[Z_T] \le \mathbb{E}[Z_0]$) if $\{Z_k\}$ is a submartingale (resp. a supermartingale).
\end{theorem}

Theorem~\ref{th:stop} enables us to exploit the martingale-like property of 
Definition~\ref{def:accurate:proba} in order to characterize the index of the first stationary point 
encountered by the method.
Using both theorems along with Theorem~\ref{th:gendecrease}, we bound the expected 
number of iterations needed by Algorithm \ref{alg:alas} to produce an approximate function stationary 
point for the model.

\begin{theorem}\label{th:complexity0}
Let Assumptions \ref{as:f:lower}, \ref{as:f:H} and \ref{as:accurate:in:prob} hold,
with $\delta=(\delta_f,\delta_g,\delta_H)$ satisfying~\eqref{eq:bounddeltaepsgendecrease}. Suppose 
that for every index $k$, the sample size $\skn_k$ satisfies~\eqref{eq:samplesizebounddec}.
Denote $T_{\epsilon}$ to be the first iteration index $k$ of 
Algorithm~\ref{alg:alas} for which 
\[
	\min\{\|\nabla f(x_k)\|,\|\nabla f(x_{k+1})\|\} \le (1+\ccg)\epsilon  \quad \mbox{and} \quad 
	\lambda_{\min}\left(\nabla^2 f(x_k)\right) \ge -(1+\cch)\epsilon^{1/2}.
\]
Then, $T_{\epsilon}<\infty$ almost surely, and 
\begin{equation} \label{eq:complexity0} 
	\E{T_\epsilon}\; \le \frac{\left(f(x_0)- f_{\low}\right)}{c_{\epsilon}}+1,~~~\mbox{where}~ 
		c_{\epsilon} = p  \hat{c} \epsilon^{3/2}~~\mbox{and}~~\hat{c}=\tfrac{\eta}{24}c^3.
\end{equation}
\end{theorem}
\proof{Proof of Theorem~\ref{th:complexity0}.}
	We first observe that 
	\revised{both $x_{k+1}$ and the sample size 
	$\mathcal{S}_k$ belong to $\mathcal{F}_k$, implying that}
	$\{T_{\epsilon} = k\} \in \mathcal{F}_k$ for all $k$ and $T_{\epsilon}$ is 
	indeed a stopping time.

	We first show that the event $\{T_{\epsilon}=\infty\}$ has a zero 
	probability of occurrence. To this end, we suppose that for every iteration 
	index $k$, we have $k < T_{\epsilon}$. \revised{Recalling the definitions of 
	$\mathcal{E}^1_k:=\{\|\nabla f(x_k)\|> (1+\ccg)\epsilon\}$
    and $\mathcal{E}^2_k:=\{\lambda_{\min}\left(\nabla^2 f(x_k)\right) < -(1+\cch)\epsilon^{1/2}\}$,  
    having $k < T_{\epsilon}$ implies that the events 
	$(\mathcal{E}_0^1\cap \mathcal{E}_1^1)\cup\mathcal{E}_0^2,\dots,(\mathcal{E}_k^1\cap \mathcal{E}_{k+1}^1)\cup\mathcal{E}_{k}^2$} occur, where we recall that 
	\revised{both $\mathcal{E}_j^1$ and $\mathcal{E}_j^2$ belong to} $\mathcal{F}_{j-1}$. We thus 
	define the following filtration:
	\revised{
	\begin{equation} \label{eq:filtrationwcc}
		\mathcal{T}_{0} = \mathcal{F}_{-1} \cap \left(\mathcal{E}_0^1\cup\mathcal{E}_0^2\right),\quad \mathcal{T}_k = \mathcal{F}_{k-1} 
		\cap \left((\mathcal{E}_0^1\cup\mathcal{E}_0^2)\cap\dots\cap(\mathcal{E}_k^1\cup\mathcal{E}_{k}^2)\right) \, \forall k \ge 1,
	\end{equation}
	}
	where we use $\mathcal{F} \cap E$ to denote the trace $\sigma$-algebra of the event $E$ on the 
	$\sigma$-algebra $\mathcal{F}$, i.e., $\mathcal{F} \cap E = \{E \cap F \ :\ F\in\mathcal{F}\}$. 
	For every $k\ge 0$ and any event $A$, we thus have $\mathcal{T}_k \subset \mathcal{T}_{k+1}$ and (by the same argument used to 
	establish~\eqref{eq:inclusionsigalg})\revised{, we obtain}
	\[
		\E{A | \mathcal{T}_k} = \E{\E{A| \mathcal{F}_{k-1},
		\revised{\mathcal{E}_k^1\cup\mathcal{E}_{k}^2)}}|\mathcal{T}_k}.
	\] 
	
	\revised{
	If $T_{\epsilon}=\infty$, then the assumptions of Theorem~\ref{th:gendecrease} are satisfied 
	for all iterations $k$. In particular, for all $k$, the events $\mathcal{E}_k^1$, $\mathcal{E}_{k+1}^1$ 
	and $\mathcal{E}_k^2$ occur. Moreover,  under the event $I_k$, $\mathcal{E}^i_k$ and $E^i_k$, $i=1,2$ occur also and are equivalent, hence 
	$$
	 \mathbb{P}\left(\overline{E^+_{k}} |
		\mathcal{F}_{k-1},E^{1}_k\cup E_k^{2},I_k\right)=1 -\mathbb{P}\left(E^+_{k}
		|\ \mathcal{F}_{k-1},E^{1}_k\cup E_k^{2},I_k\right)=0 
	$$
	Thus, from Theorem~\ref{th:gendecrease}, one gets
			\begin{eqnarray*}
	\mathbb{E}\left[f(x_k+\alpha_k d_k) - f(x_k)\ |\  \mathcal{F}_{k-1}, \mathcal{E}^1_k\cup\mathcal{E}^2_k\right]	&\le & -c_{\epsilon}.
	\end{eqnarray*}
	Thus, since $x_k \in \mathcal{T}_k$, we obtain
		\[ \E{f(x_{k+1})\ |\ \mathcal{T}_k} +  c_{\epsilon} \le 
		 f(x_k),
	\]
	Subtracting $f_{\low}$ on both sides, we obtain
	\begin{equation} \label{eq:stoppingtime0flowdec}
		\E{f(x_{k+1})-f_{\low}\ |\ \mathcal{T}_k} +  c_{\epsilon} \; \le \; f(x_k)-f_{\low}.
	\end{equation}
	}
	
	\revised{
	As a result, we can apply Theorem~\ref{th:mart} with $\alpha^k = f(x_k)-f_{\low} \ge 0$, 
	$\Sigma^k = \mathcal{T}_k$ and $\nu^k=c_{\epsilon} > 0$: we thus obtain that 
	$\sum_{k=0}^{\infty} c_{\epsilon} < \infty$, which is obviously false. This implies that 
	$T_{\epsilon}$ must be finite almost surely.
	}

	Consider now the sequence of random variables given by,
	\[
		R_k= f(x_{\min(k,T_{\epsilon})})+\max\left(\min(k,T_{\epsilon})-1,0\right) c_{\epsilon}.
	\]
	\revised{For any $k<T_{\epsilon}$, we have occurrence of 
	$\mathcal{E}^1_k\cup \mathcal{E}^2_k$. As in the proof of Theorem~\ref{th:gendecrease}, we 
	use the fact that $(I_k,E^1_k\cup E^2_k)$ and $(I_k,\mathcal{E}^1_k\cup \mathcal{E}^2_k)$ have 
	the same logical value for $k<T_{\epsilon}$, hence 
	$\mathbb{P}\left(\overline{E^+_{k}} |
	\mathcal{F}_{k-1},E^{1}_k\cup E_k^{2},I_k, k<T_{\epsilon}\right)=0$. This leads to
	\begin{eqnarray*}
		\mathbb{E}[R_{k+1}\ |\ \mathcal{T}_k, k<T_{\epsilon}]  
		&= 
		\mathbb{E}[R_{k+1}\ |\ \mathcal{F}_{k-1},\mathcal{E}^1_k\cup\mathcal{E}^2_k, k<T_{\epsilon}] 
		= 
		\mathbb{E}[f(x_{k+1})\ |\ \mathcal{T}_k]+k\,c_{\epsilon} 
		\le f(x_{k})-c_{\epsilon}+k\,c_{\epsilon} \le R_k.
	\end{eqnarray*}
	Therefore, $\mathbb{E}[R_{k+1}\ |\ \mathcal{T}_k, k<T_{\epsilon}] \le R_k$ 
	while $R_{k+1}=R_k$ when $k\ge T_{\epsilon}$, implying that $R_k$ is a supermartingale.
	Moreover, this supermartingale has bounded expected increments, since}
	\begin{eqnarray*}
		\E{|R_{k+1}-R_k|\ |\ \mathcal{T}_k}  & = &\; 
		\E{|f(x_{k+1})+(k+1)c_{\epsilon}-f(x_k)-kc_{\epsilon}|\ |\ \mathcal{T}_k} \\
		& \le & \; \E{|f(x_{k+1})-f(x_k)|+c_{\epsilon}\ |\ \mathcal{T}_k} 
		\le c_{\epsilon}+\max\left(c_{\epsilon},f_{\text{max}}-f_{\text{low}}\right) <\infty
	\end{eqnarray*}
	\revised{Noting that} $T_{\epsilon} < \infty$ 
	almost surely, we satisfy the assumptions of Theorem~\ref{th:stop}: it thus holds that 
	$\mathbb{E}[R_{T_{\epsilon}}]\le \mathbb{E}[R_0]$, \revised{leading to}
	\[
		f_{\low} + (\E{T_{\epsilon}}-1)c_{\epsilon} \le \E{R_{T_{\epsilon}}} 
		\le \mathbb{E}[R_0] = f(x_0).
	\]
	\revised{Re-arranging the terms leads to}
	\[
		\E{T_\epsilon}\; \le \frac{\left(f(x_0)- f_{\low}\right)}{c_{\epsilon}}+1,
	\]
	which is the desired result.
\endproof

The result of Theorem~\ref{th:complexity0} gives a worst-case complexity bound on the expected number of
iterations until a function stationarity point is reached. This does not provide a practical stopping 
criterion for the algorithm, because only \emph{model} stationarity can be tested for during an 
algorithmic run (even though in the case of $\skn_k=p=1$, both notions of stationarity are 
equivalent). Still, by combining Theorem~\ref{th:complexity0} and 
Lemma~\ref{lemma:modeltruestationarity}, we can show that after at most 
$\mathcal{O}(p\epsilon^{-3/2})$ iterations on average, if the model is accurate, 
then the corresponding iterate will be function stationary. 

In our algorithm, we assume that accuracy is only guaranteed with a certain probability at 
every iteration. As a result, stopping after encountering an iterate that is model stationary 
only comes with a weak guarantee of returning a point that is function stationary. In 
developing an appropriate stopping criterion, we wish to avoid such ``false positives".
To this end, one possibility consists of requiring model stationarity to 
be satisfied for a certain number of successive iterations. Our approach is motivated by the 
following result\revised{, proved in the online companion.}

\begin{proposition}\label{prop:stop} 
Under the assumptions of Theorem~\ref{th:complexity0}, suppose that Algorithm~\ref{alg:alas} 
reaches an iteration index $k+J$ such that for every $j \in \{k,k+1,\dots,k+J\}$, 
$
	\min\left\{\|g_j\|,\|g_j^+\|\right\} \le \epsilon  \quad \mathrm{and} \quad \lambda_j \ge -\epsilon^{1/2}. 
$
Suppose further that $\delta_g$ and $\delta_H$ satisfy~\eqref{eq:conddeltagdeltah}, 
and that the sample sizes are selected independently of the current iterate.
Then, with probability at least $1-(1-p)^{J+1}$, where $p$ is the lower bound on $p_k$ given by 
Assumption~\ref{as:accurate:in:prob}, one of the iterates $\{x_k,x_{k+1},\dots,x_{k+J}\}$ is 
$((1+\ccg)\epsilon,(1+\cch)\epsilon^{1/2})$-function stationary.
\end{proposition}

The result of Proposition~\ref{prop:stop} is only of interest if we can ensure that 
such a sequence of model stationary iterates can occur in a bounded number of iterations. This will be 
the case provided we reject iterates for which we cannot certify sufficient decrease, hence the following 
additional assumption.

\revised{
\begin{assumption} \label{as:modifalgo}
	In Algorithm~\ref{alg:alas}, Step 7 is replaced by:\\
	\emph{7'. If $\min\{\|g_k\|,\|g_k^+\|\} < \epsilon$ and $\lambda_k > -\epsilon^{1/2}$, set 
	$x_{k+1}=x_k$, otherwise set $x_{k+1}=x_k+\alpha_k\,d_k$.}\\
	Moreover, the sample size is selected independently of the current iterate so 
	that $\Pr\left(I_k | \mathcal{F}_{k-1}\right) = p$ where $p$ is defined in 
	Assumption~\ref{as:accurate:in:prob}.
\end{assumption}
}

This algorithmic change allows us to measure stationarity at a given iterate based on a series of 
samples. Note that our method now requires two gradient evaluations, but that those involve 
the same sample. Under a slightly stronger set of assumptions, Proposition~\ref{prop:complexityJ} then 
guarantees that sequences of model stationary points of arbitrary length will occur in expectation.
\revised{The proof of this result can be found in the online companion.}

\begin{proposition}\label{prop:complexityJ}
Let Assumptions \ref{as:f:lower}, \ref{as:f:H}, \ref{as:accurate:in:prob} and 
\ref{as:modifalgo} hold,
where $\delta$ satisfies~\eqref{eq:bounddeltaepsgendecrease} 
with $\kappa_g,\kappa_H \in (0,1)$, and $p_k=p\ \forall k$. 
For a given $J \in \mathbb{N}$, define $\Tj^m$ as the first iteration index of 
Algorithm~\ref{alg:alas} for which 
\begin{equation} \label{eq:Jp1modelstationarity}
	\min\{\|g_k\|,\|g^+_k\|\} < \epsilon  \quad \mbox{and} \quad \lambda_k > -\epsilon^{1/2}, \quad
	\forall k \in \{\Tj^m,\Tj^m+1,...,\Tj^m+J\}.
\end{equation}
Suppose finally that for every index $k$, the sample size $\skn_k$ satisfies
\begin{equation} \label{eq:Jcomplexitysamplecond}
	\forall k, \quad \skn_k \ge \pfun(c\hat{\epsilon}^{1/2},p),
\end{equation}
where $\hat{\epsilon} = \min\{\tfrac{1-\kappa_g}{1+\kappa_g},
\tfrac{(1-\kappa_H)^2}{(1+\kappa_H)^2}\}\epsilon$. 
(Note that $\pfun(c\hat{\epsilon}^{1/2},p) \ge \pfun(c\epsilon^{1/2},p)$.)

Then, $\Tj^m<\infty$ almost surely, and 
\begin{equation} \label{eq:complexityJ}
	\mathbb{E}[\Tj^m] \; \le \; \frac{\left(f(x_0)- f_{\low}\right)}{c_{\hat{\eps}}}+J+1 ~~~\mbox{where}~	
	c_{\hat{\eps}} = p  \hat{c} \hat{\eps}^{3/2} ~~\mbox{and}~~\hat{c}=\tfrac{\eta}{24}c^3.
\end{equation}
\end{proposition}

With the result of Proposition~\ref{prop:complexityJ}, we are guaranteed that there will exist 
consecutive iterations satisfying model stationarity in expectation. Checking for stationarity over 
successive iterations thus represents a valid stopping criterion in practice. If an estimate of the 
probability $p$ is known, one can even choose $J$ to guarantee that the probability of computing a 
stationary iterate is sufficiently high.

\revised{
To end this section, we establish a bound 
in expectation on the number of evaluations of $f_i$ needed to reach 
a stationary point. Note 
that we must account for the additional objective evaluations induced by the 
line-search process (see~\cite[Theorem 8]{royer2018complexity} for details).
}

\begin{corollary} \label{coro:evalcomplexityJevals}
Let Assumptions \ref{as:f:lower}, \ref{as:f:H}, \ref{as:accurate:in:prob} and 
\ref{as:modifalgo} hold,
where $\delta$ satisfies~\eqref{eq:bounddeltaepsgendecrease} 
with $\kappa_g,\kappa_H \in (0,1)$, and $p_k=p$ for every $k$. Suppose
also that
$\skn_k \ge \pfun(c\hat{\epsilon}^{1/2},p)$, then the expected 
number of evaluations of $\nabla f_i$ and $\nabla^2 f_i$ are respectively 
bounded above by 
\begin{equation} \label{eq:evalcomplexityJderivatives}
\frac{3 \left(f(x_0)- f_{\low}\right)}{p\hat{c}}\epsilon^{-3/2}+1 
\end{equation}
while the expected number of function evaluations is bounded above by
$$
(1+\max\{j_{nc},j_n,j_{rn}\}) p^{-(J+1)}\left[\frac{\left(f(x_0)- f_{\low}\right)}{p\hat{c}} 
		\,\hat{\epsilon}^{-3/2}+J+1 \right].	
$$
\end{corollary}

\subsection{Sample and evaluation complexity for uniform sampling} 
\label{sec:samplewcc}

\subsubsection{Comparison with the deterministic line-search method:}

In general, we can consider that at each iteration we must perform $|\mathcal{S}_k|n$ computations to 
evaluate the sample gradient, $|\mathcal{S}_k|n^2$ to compute the sample Hessian, and evaluate a function 
$\bar{j}|\mathcal{S}_k|$ times during the line search. Per Lemma~\ref{lm:bound:alpha_kd_k}, the maximum 
number of line-search iterations $\bar{j}$ can be considered of the order $\log (1/\epsilon)$. 
As a result, each iteration requires $|\mathcal{S}_k|(n+n^2+\log(1/\epsilon))$ computations. 
For our sampling rate to yield improvement over the deterministic (fully sampled) method, we 
require $N\log(1/\epsilon)$ to be much larger than $|\mathcal{S}_k|n^2$, implying that the 
appropriate regime for this algorithm is one where the number of variables is considerably less than 
the number of data points. 

To illustrate the theoretical results, we now comment on the two main requirements made in the analysis of 
the Section 4 are related to the function value accuracy $\delta_f$ and the sample size $\skn_k$. 
More precisely, for a given tolerance $\epsilon$, we required in~Theorem~\ref{th:complexity0} 
that:
\begin{enumerate}
\item $\delta_f \le \tfrac{\eta}{24}c^3\epsilon^{3/2},\ \delta_g \le \ccg \epsilon,\ \delta_H \le \cch \epsilon^{1/2}$.
\item $\skn_k \; \ge \; \pfun(c\epsilon^{1/2},p). $  
\end{enumerate}

In this section, we provide estimates of the minimum number of samples necessary to achieve those two 
conditions in the case of a uniform sampling strategy. To facilitate the exposure, we discard the 
case $p=1$, for which the properties above trivially hold, and focus on $p \in (0,1)$. Although we 
focus on the properties required for Theorem~\ref{th:complexity0}, note that a similar analysis holds for the requirements of~Proposition~\ref{prop:complexityJ}.

For the rest of the section, we suppose that the set $\mathcal{S}_k$ is formed of $n\bar{\skn}$ 
indexes chosen uniformly at random with replacement. where $\bar{\skn}$ is independent of $k$. 
That is, for every $i \in \mathcal{S}_k$ and every $j=1,\dots,N$, we have $\Pr(i=j) = \frac{1}{N}$.
This case has been well studied in the case of subsampled Hessian~\cite{xu2019newton} and 
subsampled gradient~\cite{roosta2019subsampled}. The next theorem derives the result regarding the required conditions on $\skn_n$ to ensure the event $I_k$. The theorem uses standard arguments (e.g., ~\cite[Lemma 16]{xu2019newton} and ~\cite[Lemma 2]{roosta2016subsampled}); for completeness, its proof is given in the appendix.
\begin{theorem} \label{theo:suffsampleunif}
	Let~Assumptions 1 and 2 hold. For any $p \in (0,1)$, let $\hat p= \frac{p+3}{4}$. Suppose that 
	the sample fractions $\skn_k$ of Algorithm 1 are chosen to satisfy 
	$\skn_k \ge \skn(\epsilon)$ for every $k$, where
	\begin{equation*} \label{eq:suffsampleunif}
		\skn(\epsilon) {:=} \frac{1}{N}\max\left\{
		N\,\pfun(c\epsilon^{1/2},\hat p), 
		\frac{9216 f_{\up}^2}{\eta^2 c^6\epsilon^{3}}\ln\left(\tfrac{2}{1-\hat p}\right),		
		\frac{U_g^2}{\ccg^2\epsilon^2}\left[1+\sqrt{8\ln\left(\tfrac{1}{1-\hat p}\right)}\right]^2,
		\frac{16L^2}{\cch^2\epsilon}\ln\left(\tfrac{2N}{1-\hat p}\right)\right\}.
	\end{equation*}
	Then, the model sequence is $p$-probabilistically  
	$(\delta_f,\delta_g,\delta_H)$-accurate with $\delta_f=\tfrac{\eta}{24}c^3\epsilon^{3/2}$, 
	$\delta_g=\ccg\epsilon$ and $\delta_H=\cch\epsilon^{1/2}$. Moreover, all the results from~Section~\ref{subsec:cvwcc:iterwcc} hold.
\end{theorem}

We observe that explicit computation of these bounds would require estimating $L, L_H, U_H,$ and 
$U_g$. If the orders of magnitude for the aforementioned quantities are available, they can be 
used for choosing the sample size. In addition, note that the bound presented is global, i.e., it should hold for
every iterate $k$, however, it is also possible to consider these problem constants as they hold over
a region, and thus when $x_k$ is in this region $\skn_k$ can be potentially adjusted accordingly
in an adaptive fashion, if some local estimates for these constants could be made available. 

\subsubsection{Comparison with other sample complexities:}

By comparing it with sample complexity results available in the literature, we can position our method 
within the existing landscape of results, and get insight about the cost of second-order requirements, as 
well as that of using inexact function values.

When applied to nonconvex problems, a standard stochastic gradient approach with fixed step size has a 
complexity in $\mathcal{O}(\epsilon^{-4})$ (both in terms of iterations and gradient evaluations) for 
reaching approximate first-order optimality~\cite{ghadimi2013stoch}. Modified SGD methods that take 
curvature information can significantly improve over that bound, and even possess second-order guarantees. 

Typical guarantees are provided in terms of function stationarity (though this condition cannot be 
checked in practice), and hold with high probability.
Our results hold in expectation, but involve a model stationary condition that we can check at every 
iteration. Moreover, it is possible to convert a rate in expectation into 
high-probability rates (of same order), following for instance an argument used for trust-region 
algorithms with probabilistic models~\cite{gratton2018trprobamodels}.

It is also interesting to compare our complexity orders with those obtained 
by first-order methods in stochastic optimization. Stochastic trust-region 
methods~\cite{chen2018stochastic,larson2016stochtr} require $\mathcal{O}
\left(\Delta_k^{-4}\right)$ samples per iteration, where $\Delta_k$ is the 
trust-region radius and serves as an approximation of the norm of the 
gradient. The line-search algorithm of~\cite{paquette2018stochls} guarantees 
sufficient accuracy in the function values if the sample size is of order 
$\mathcal{O}\left(\alpha_k^{-2}\|g_k\|^{-4}\right)$ (we use our notations for 
consistency). By comparison our method requires $\mathcal{O}(\eps^{-3/2})$ 
samples, where $\eps^{-3/2}$ is used as a proxy for $\alpha_k^{-3}\|d_k\|^3$. 
Our sample complexity can thus be higher than that of other methods, in that 
it does not depend on the iteration level. On the other hand, this makes our 
approach well-defined, and enables the derivation of second-order guarantees, 
while previously proposed methods such as that of Paquette and 
Scheinberg~\cite{paquette2018stochls} is only concerned with first-order 
guarantees.

Finally, we discuss sample bounds for Newton-based methods in a subsampling 
context. In~\cite{xu2019newton}, it was shown that the desired complexity 
rate of $\mathcal{O}(\epsilon^{-3/2})$ is achieved with probability $1-p$, 
if for every $k$, the sample size (for approximating the Hessian) satisfies:
\[
|\mathcal{S}_k|\ge \frac{16 U^2_H}{\epsilon}\ln\left(\frac{2N}{p}\right)
\]
The same result is used in~\cite{liu2017noisy}. In both cases, the 
\revised{sample sizes are inversely proportional to the desired tolerance 
squared.} 
\revised{
In~\cite{yao2018inexact}, high probability results were derived with the 
accuracy of the function gradient and Hessian estimates $\delta_g$ and 
$\delta_H$ being bounded by $\epsilon$ and $\epsilon^{1/2}$, respectively. 
The resulting sampling bounds are 
$\mathcal{O}(\epsilon^{-2})$.}
In our setting, we additionally require the function accuracy to be of order 
$\epsilon^{3/2}$, where $\epsilon$ is the gradient tolerance, yet the 
\revised{other dependencies on $\epsilon$ in terms of gradient and Hessian 
sampling matches previous bounds in the case of exact function values 
and subsampled derivatives~\cite{xu2019newton,yao2018inexact}. In that 
sense, we can identify the cost induced by assuming that only inexact 
function values are available.} 

\revised{Finally, we observe that our sample size requirements are proportional to $n^2$, while 
the upper bound $U_g$ on the gradients can also grow with $n$.} Thus, unless the setting is a regime 
with a very large number of data samples and a relatively small number of variables, the computed 
required sampling size could approach $N$. It appears that this
is an issue associated with the required worst-case sampling rates for higher-order sampling algorithms
in general, given the literature described above. To the best of our knowledge, tightening these 
bounds in a general setting remains an open area of research, even though improved sample 
complexities can be obtained by exploiting various specific problem structures. 
\revised{Though out of the scope of this work, we remark that practical approaches use 
less samples to approximate second-order information than they use to approximate 
first-order information~\cite{byrd2011stochhessian}.}

\section{Numerical experiments} 
\label{sec:numerics}
In this section, we present a numerical study of an implementation of our proposed framework on 
several machine learning tasks involving both real and simulated data. Our goal is 
to advocate for the use of second-order methods, such as the one described in 
this paper, for certain training problems with a nonconvex optimization 
landscape. We are particularly interested in highlighting notable advantages of going 
beyond first-order algorithms, which are still the preferred approaches in most machine 
learning problems. Indeed, first-order methods have lower computational complexity than 
second-order methods (that is, they require less work to compute a step at each iteration), 
but second-order methods possess better iteration complexity than first-order ones (they 
converge in fewer iterations). In many machine learning settings, lower computational 
complexity \revised{trumps} iteration complexity, and therefore first-order 
algorithms are viewed as superior. Nevertheless, our results suggest that second-order schemes can 
be competitive in certain contexts.

We focus our experiments on architectures with 
a relatively small number of network parameters  (optimization variables). 
We point out that there exist important problems where the number of optimization variables $n$ is indeed small compared to the number of samples $N$. Besides logistic regression and shallow networks, the recent interest in formulating sparser architectures such as \texttt{efficientnet}~\cite{tan2019efficientnet}
for problems training huge amounts of data, in order to lighten memory loads in computation-heavy
and memory-light HPC hardware, indicates
the potential for increased applications of this sort in the future.

Although the most popular variants of SGD in practice such as 
{\tt{ADAM}}~\cite{kingma2014adam} incorporate advanced features, we believe that there is 
value in putting our approach in perspective with a simple framework emblematic of 
the first-order methods used in machine learning. For these reasons, we compare our 
implementation with a vanilla SGD method using a constant
step size (learning rate). We place ourselves in a setting in which the learning rate 
is \revised{tuned} among a set of predefined values corresponding 
to standard choices. Our goal is to highlight the sensitivity of such a method to this 
hyperparameter, and to compare it with our scheme for a given batch size.

\if\smallnumerics
\fi

\subsection{Implementation}
\label{subsec:implement}

\revised{
We first describe the modifications to Algorithm~\ref{alg:alas} that we made in the implementation 
(which we identify as {\tt{ALAS}} in this section).
Algorithm~\ref{alg:alas} differed from the method of~\cite{royer2018complexity} in that we did not consider steps along the negative gradient direction, in order to simplify our theoretical analysis. We have re-introduced these steps in our implementation.
This change for practical performance has proper intuitive justification: if the negative gradient direction is already a good direction of negative curvature, or the Hessian is flat along its direction, then the negative gradient direction already enjoys second-order properties that ensure
it is a good descent direction for the function. In our experiments, these steps can be taken, but 
the vast majority of steps correspond to cases in Algorithm~\ref{alg:alas} (see Figure~\ref{fig:steps_IJCNN} and the online companion for an illustration of this behavior).
Secondly, we modified the condition to select a negative curvature step from 
$\lambda_k < -\epsilon^{1/2}$ to $\lambda_k < -\|g_k\|^{1/2}$. The former was necessary to handle the 
case of a stationary point for which $\|g_k^+\| < \epsilon \le \|g_k\|$ and $\lambda_k > - \eps^{1/2}$ in an appropriate fashion in the analysis of Section 3 in the main text, but this case was rarely encountered in practice.
Thirdly, we replaced the cubic decrease condition by a sufficient decrease quadratic in the norm of the step. This is again for performance reasons, and although theoretical results could be established using this condition, it does not appear possible to obtain the same dependencies in $\epsilon$ as given in~Theorem 4. We recall that the cubic decrease condition has been instrumental in deriving optimal iteration complexity bounds for Newton-type methods~\cite{cartis2011arccomplexity,royer2018complexity}.  The detailed description of  {\tt{ALAS}}  as implemented is provided
in the appendix.
}

In our implementation of {\tt{ALAS}}, we chose the values $\epsilon=10^{-5}$, $\eta = 10^{-2}$, and $\theta=0.9$. The ALAS framework and a standard stochastic gradient descent (SGD) algorithm were both implemented in Python 3.7. \revised{For fair comparison, we compare the performance of {\tt{ALAS}} and {\tt{SGD}} using the same batch size, or percentage of the dataset taken as samples in each iteration ($p_k$ in the theory)}. 
Various Python libraries were used, including JAX~\cite{jax2018github} for efficient compilation and automatic differentiation, as well as NumPy \cite{numpy}, SciPy \cite{scipy} and Pandas \cite{pandas}. \revised{The smallest eigenvalue $\lambda_k$  and the negative curvature step are computed using the ``scipy.sparse.linalg.eigs" routine.}
%
All our experiments were run using Intel Xeon CPU at 2.30~GHz and 25.51~GB of RAM, using a Linux distribution Ubuntu 18.04.3 LTS with Linux kernel 4.14.137.

\subsection{Classification on the IJCNN1 dataset}
\label{subsec:numIJCNN1}
We first tested our algorithms on a binary classification task for the IJCNN1 dataset (with $N=49,990$ 
samples and $n=22$ features per sample). Our goal was to train a neural network for binary classification. We used two different architectures: one had one hidden layer with 4 neurons, the 
other had two hidden layers with 4 neurons each. Both networks used the hyperbolic tangent activation 
function \smash{$\phi(x)=\tfrac{\exp(2x)-1}{\exp(2x)+1}$} and mean-squared error (MSE) loss, resulting in 
a twice continuously differentiable optimization problem, with a highly nonlinear objective. The results we present were obtained using samples corresponding to 20 \%, 10 \%,  5\% and 1\% of the dataset, with a runtime of 10 seconds. 
The dataset was partitioned into disjoint samples of given sizes for each data pass (epoch) and these samples were randomly repartitioned every data pass.
The MSE loss is defined as
\begin{equation}
    l(\mathcal{B}) = \frac{1}{|\mathcal{B}|}\sum_{i \in \mathcal{B}} \left(y_i - \widehat{y_i}\right)^2,
\end{equation}
where $\mathcal{B}$ is the sampled dataset (mini-batch), $|\mathcal{B}|$ the size of the mini-batch,
$y_i \in \{-1, 1\}$ the actual binary label of the sample $i$  and $\widehat{y_i}$ is the predicted label of for the sample $i$.

We tested SGD with several possible values for the learning rate, namely {1, 0.6, 0.3, 0.1, 0.01, 0.001} 
\revised{(the details are available in the online companion), and selected the best variant for 
comparison with ALAS}. Note that the step size in ALAS is selected adaptively through a line search, which makes this method more robust. \revised{Figure~\ref{fig:IJCNN_nn_22_CI} shows the noise and robustness of the performance by reporting the median and 90$\%$ level confidence
interval for the trajectory over 200 runs for each batch size. Note that ALAS is fairly consistent, while the precision and variability of SGD can change in a complex way depending on the learning rate. Figure~\ref{fig:IJCNN_nn_22_acc} plots the training accuracy associated with the trajectories and their variation, which match what we observed on the loss plots.
} 

Figures~\ref{fig:IJCNN_nn_22-4-1} and~\ref{fig:IJCNN_nn_22-4-4-1} compare ALAS and SGD with the classical learning rate value for this task (0.01) on the two networks with 
one and two hidden layers, respectively.  
The results are summarized in Table~\ref{tab:alg_comparison}, where we compare ALAS with the best SGD variant and the standard variant with a learning rate of 0.01. To mitigate the iteration and overall cost of both algorithms, we compared the methods based on the minimal error reached over the last 10 seconds of the run and on the median loss over the last two seconds. In all the tables, the scalar in the parenthesis for SGD entries is the learning rate and the column \textit{median loss [8-10]s} shows the median loss over the last two seconds.
In general, ALAS outperforms SGD on the IJCNN1 task in terms of these metrics. \revised{Another interesting observation is the performance of ALAS appears to be fairly stable as its
trajectory appears monotonic: in fact, the trajectory can exhibit non-monotonic trends (see Section~\ref{subsec:numMNIST} and the online companion), but those are less prominent than for the 
SGD method. Our interpretation is that the ``bad case" of poor estimates yielding an innaccurate noisy step and more poor estimates resulting in its acceptance
is fairly uncommon in the stochastic setting because our model provides a high-order stochastic estimate of the function. On the contrary, SGD can easily yield ascent in the loss function if the gradient samples correspond to components that are far from their mean, i.e. the true gradient.}

\revised{
We further study the behavior of \texttt{ALAS} by first looking at the types of steps taken 
by the method on selected runs for the IJCNN1 task, shown in Figure~\ref{fig:steps_IJCNN}. This distribution is highly dependent of the network architecture but some common patterns were observed. The most frequent type of step is \textit{Regularized Newton}, 
but other steps (in particular, along stochastic gradient directions, which we added in our implementation) occur in a small percentage of cases. Note that negative curvature directions 
are computed, highlighting the nonconvex nature of our problem.
In addition, we illustrate typical numbers of line-search iterations for our problem in 
Figure~\ref{fig:line_search_hist_IJCNN}: in general, this number is relatively mild, but the 
line-search process may require a significant number of function evaluations. Given that a 
worst-case example of this behavior is Figure~\ref{fig:line_search_hist_IJCNN_nn_22-1_100} 
where all samples are used, we attribute this to the architecture rather than to the line-search 
process itself.
\begin{figure}
    \centering
    \subfigure[22-1 100\%]{
    \includegraphics[width=7.5cm, height=4.1cm]{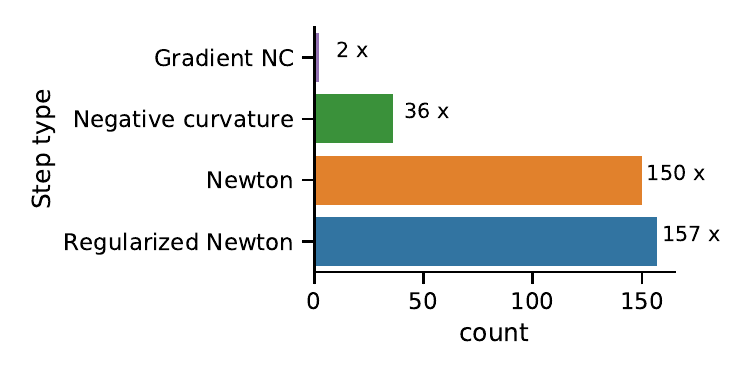}\label{fig:steps_IJCNN_nn_22-1}
    }
    ~ 
    \subfigure[22-4-1 100\%]{
    \includegraphics[width=7.5cm, height=4.1cm]{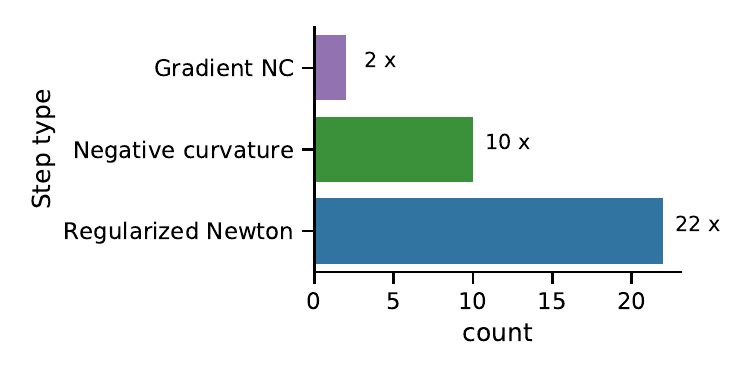}\label{fig:steps_IJCNN_nn_22-4-1}
    }
     \\
    \subfigure[22-4-4-1 1\%]{
    \includegraphics[width=7.5cm, height=4.1cm]{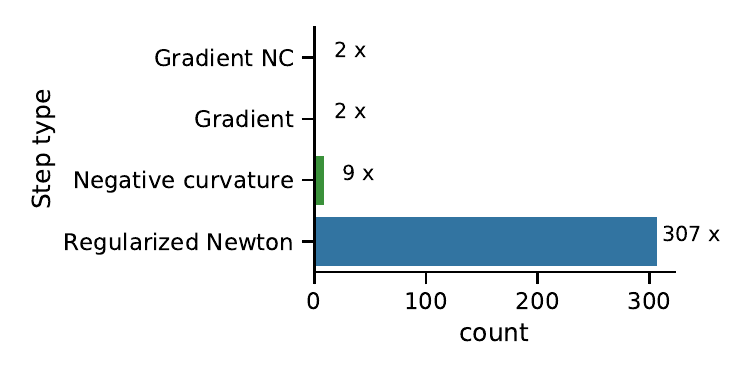}\label{fig:steps_IJCNN_nn_22-4-4-1_1}
    }
    ~ 
    \subfigure[22-4-4-1 100\%]{
    \includegraphics[width=7.5cm, height=4.1cm]{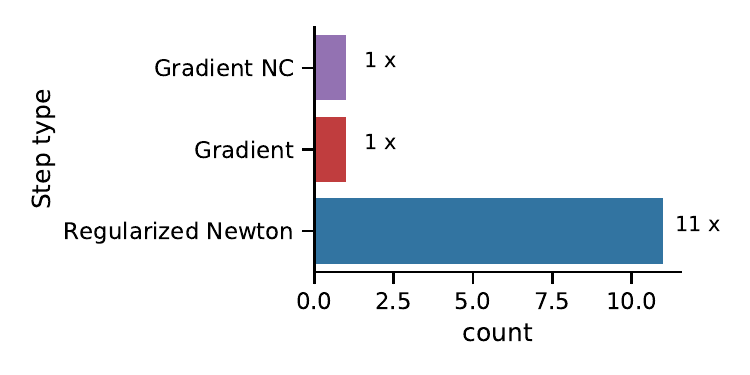}\label{fig:steps_IJCNN_nn_22-4-4-1_100}
    }
    \caption{The step type distribution of a single run of  ALAS algorithm on the IJCNN1 task for different architectures and sampling sizes.}
    \label{fig:steps_IJCNN}
\end{figure}
\begin{figure}
    \centering
    \subfigure[22-1 20\%]{
    \includegraphics[width=5cm, height=3.5cm]{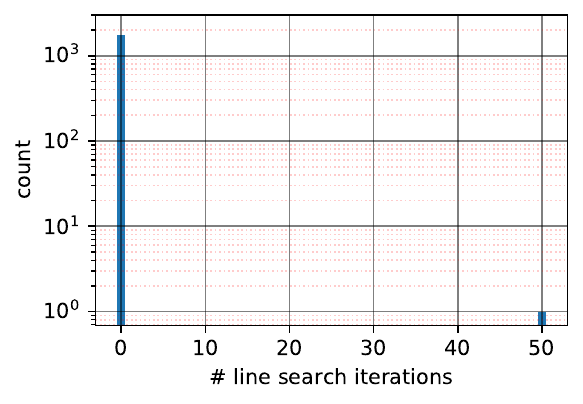}\label{fig:line_search_hist_IJCNN_nn_22-1_20}
    }
    \subfigure[22-1 100\%]{
    \includegraphics[width=5cm, height=3.5cm]{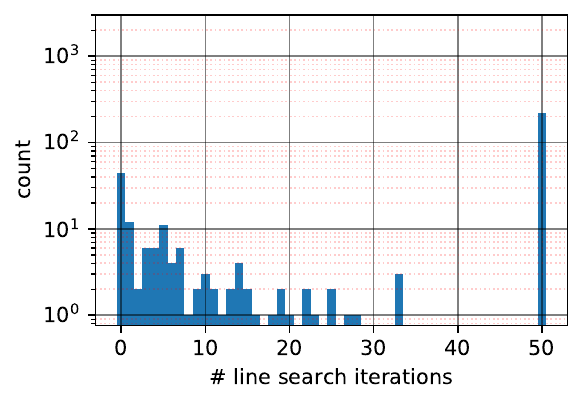}\label{fig:line_search_hist_IJCNN_nn_22-1_100}
    }
    \subfigure[22-4-1 100\%]{
    \includegraphics[width=5cm, height=3.5cm]{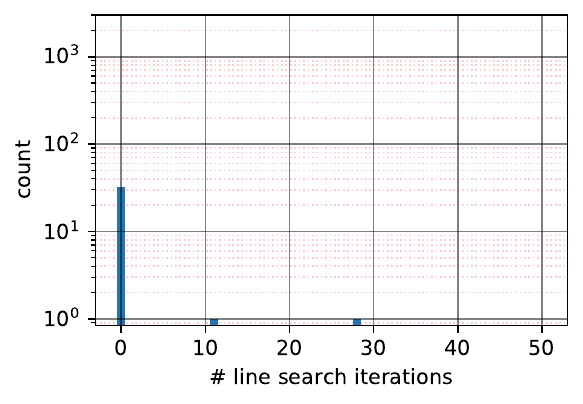}\label{fig:line_search_hist_IJCNN_nn_22-4-1_100}
    }
    \caption{Plots of the number of line-search iterations during each update for the IJCNN1 task for selected runs of the ALAS algorithm with different architectures and sampling sizes. The maximum number of line-search iterations was set to 50.}
    \label{fig:line_search_hist_IJCNN}
\end{figure}
}

\begin{table}[h!]
\centering
{\small
\begin{tabular}{c  c c c c | c c c c c }
\multicolumn{5}{c|}{Layers: 22-4-1} & \multicolumn{5}{c}{Layers: 22-4-4-1} \\ \hline
 alg. & $\skn_k$ & min loss & loss [8-10]s & iter. &  alg. & $\skn_k$ & min loss & loss [8-10]s  & iter.  \\ \hline
 ALAS &1\% & \textbf{0.0463} & \textbf{0.0478} & 551 &  ALAS & 1\% & \textbf{0.0434} & \textbf{0.0451} & 320  \\
 SGD (0.3) & 1\% & 0.0528 & 0.0558 & 12033  &  SGD (0.3) & 1\% & 0.0438 & 0.0454 & 11121\\
 SGD (0.01) & 1\% & 0.0872 & 0.0873 & 11779 &  SGD (0.01) & 1\% & 0.0841 & 0.0842 & 10997 \\ \hline
 ALAS & 5\% & \textbf{0.0449} & \textbf{0.0462} & 238 &  ALAS & 5\% & 0\textbf{.0450} & \textbf{0.0457} & 114 \\
 SGD (0.6) & 5\% & 0.0514 & 0.0569 & 8439 & SGD (0.3) & 5\% & 0.0468 & 0.0500 & 7894\\
 SGD (0.01) & 5\% & 0.0875 & 0.0876 & 8471 & SGD (0.01) & 5\% & 0.0845 & 0.0846 & 7792 \\\hline
 ALAS & 10\% & \textbf{0.0492} & \textbf{0.0510} & 124 &  ALAS & 10\% & \textbf{0.0470} & \textbf{0.0477} &  68\\
 SGD (0.6) & 10\% & 0.0557& 0.0626 & 6293 &  SGD (0.3) & 10\% & 0.0519 &  0.0571 & 5165 \\
 SGD (0.01) & 10\% & 0.0880 & 0.0883 & 5773 &  SGD (0.01) & 10\% & 0.0849 &  0.0850 & 5191\\\hline
 ALAS & 20\% & 0.0611 & \textbf{0.0628} & 76 &  ALAS & 20\% & \textbf{0.0491} & \textbf{0.0498} &  42\\
 SGD (0.6) & 20\% & \textbf{0.0598} & 0.0654 & 4389 &  SGD (0.3) & 20\% & 0.0545 &  0.0587 & 3728\\
 SGD (0.01) & 20\% & 0.0887 & 0.0891 & 4443 & SGD (0.01) & 20\% & 0.0852 &  0.0854 & 3694\\
\hline
\end{tabular}
}
\caption{Results over the time period $t =10 ~s$ on the IJCNN1 task.}
\label{tab:alg_comparison}
\end{table}

\begin{figure}
    \centering
    \subfigure[22-1 5\%]{
        \includegraphics[width=7.5cm, height=4.1cm]{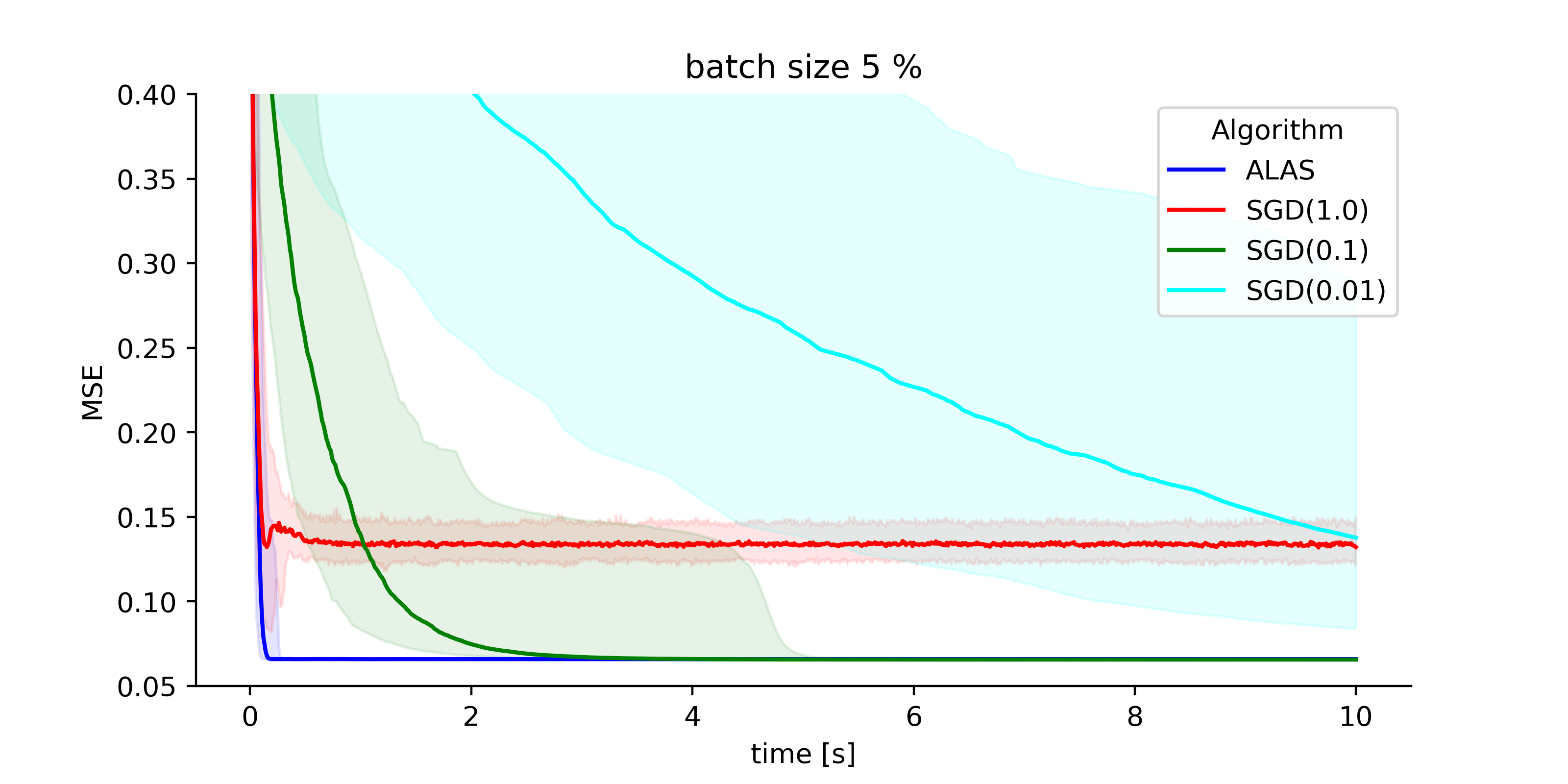}\label{fig:IJCNN_nn_22-1_005_perc}
    }
    ~ 
    \subfigure[22-1 10\%]{
        \includegraphics[width=7.5cm, height=4.1cm]{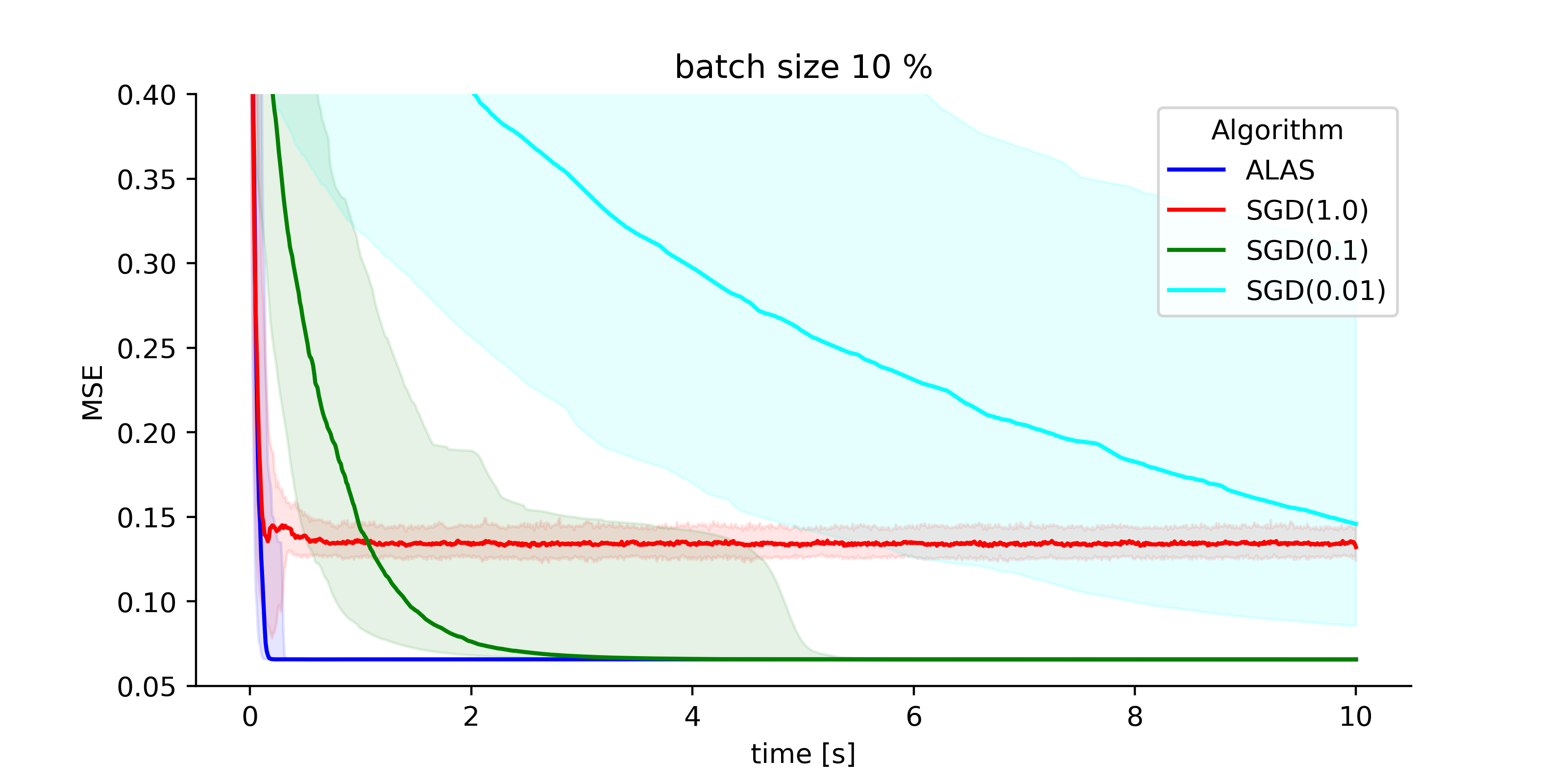}\label{fig:IJCNN_nn_22-1_010_perc}
    }
   \\
    \subfigure[22-1 20\%]{
        \includegraphics[width=7.5cm, height=4.1cm]{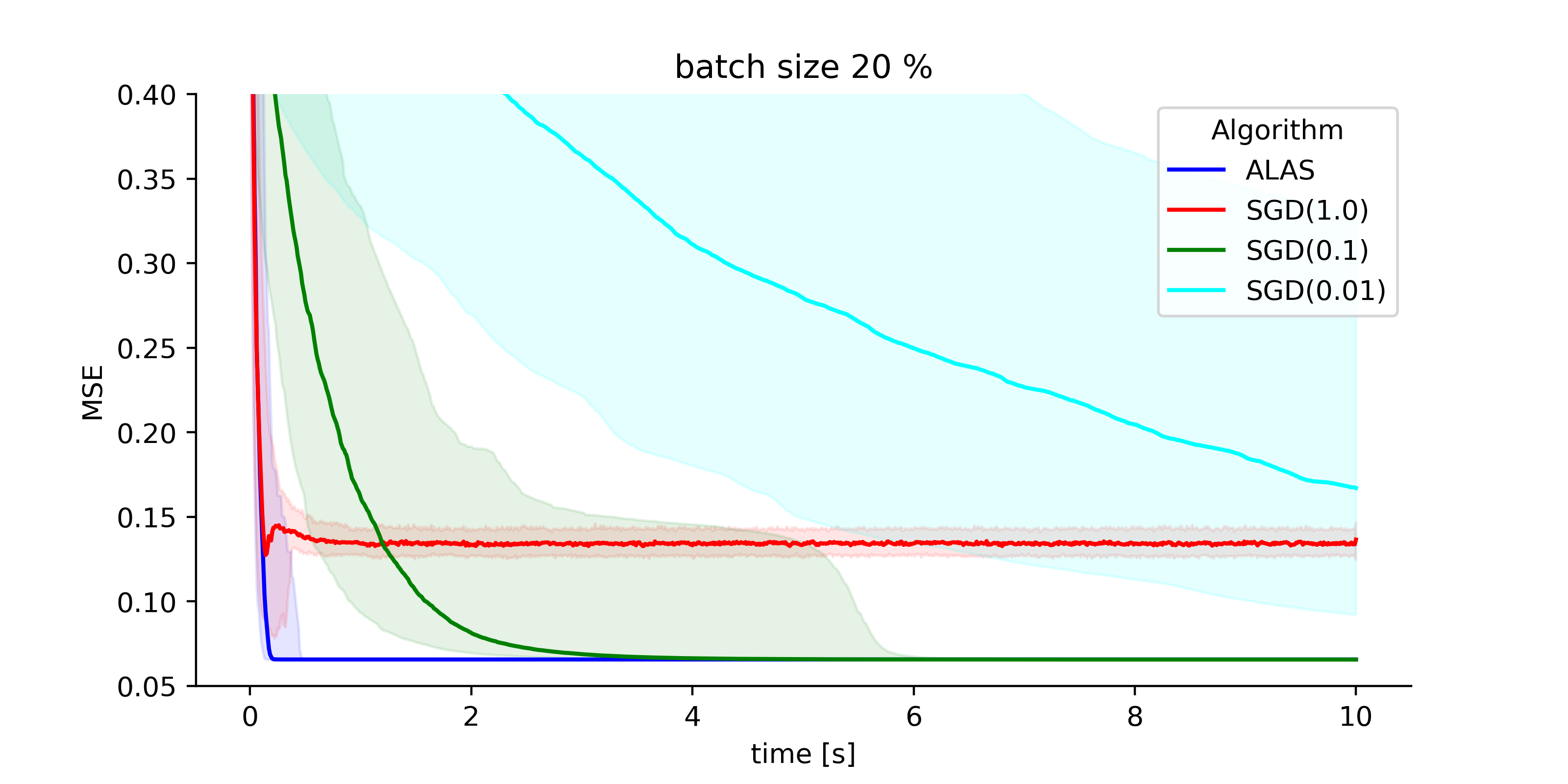}\label{fig:IJCNN_nn_22-1_020_perc}
    }
    ~ 
    \subfigure[22-1 100\%]{
        \includegraphics[width=7.5cm, height=4.1cm]{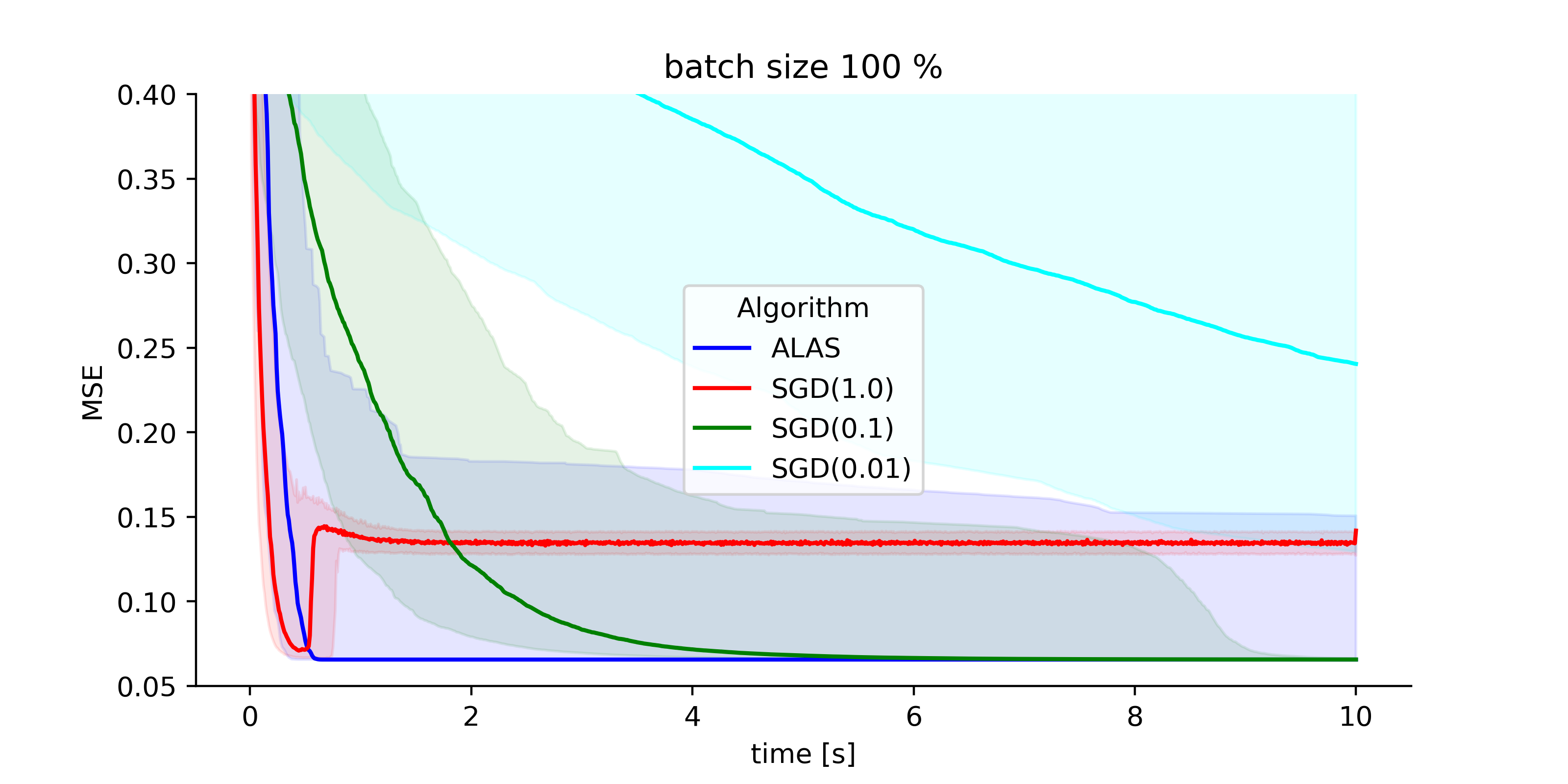}\label{fig:IJCNN_nn_22-1_100_perc}
    }
    \caption{Comparison of  ALAS and SGD (with a default learning rate 0.01) on the IJCNN1 dataset with a simple neural network with 22 input neurons, reporting
the median and $95\%$ confidence interval across 200 runs}
    \label{fig:IJCNN_nn_22_CI}
\end{figure}

\begin{figure}
    \centering
    \subfigure[22-1 5\%]{
        \includegraphics[width=7.5cm, height=4.1cm]{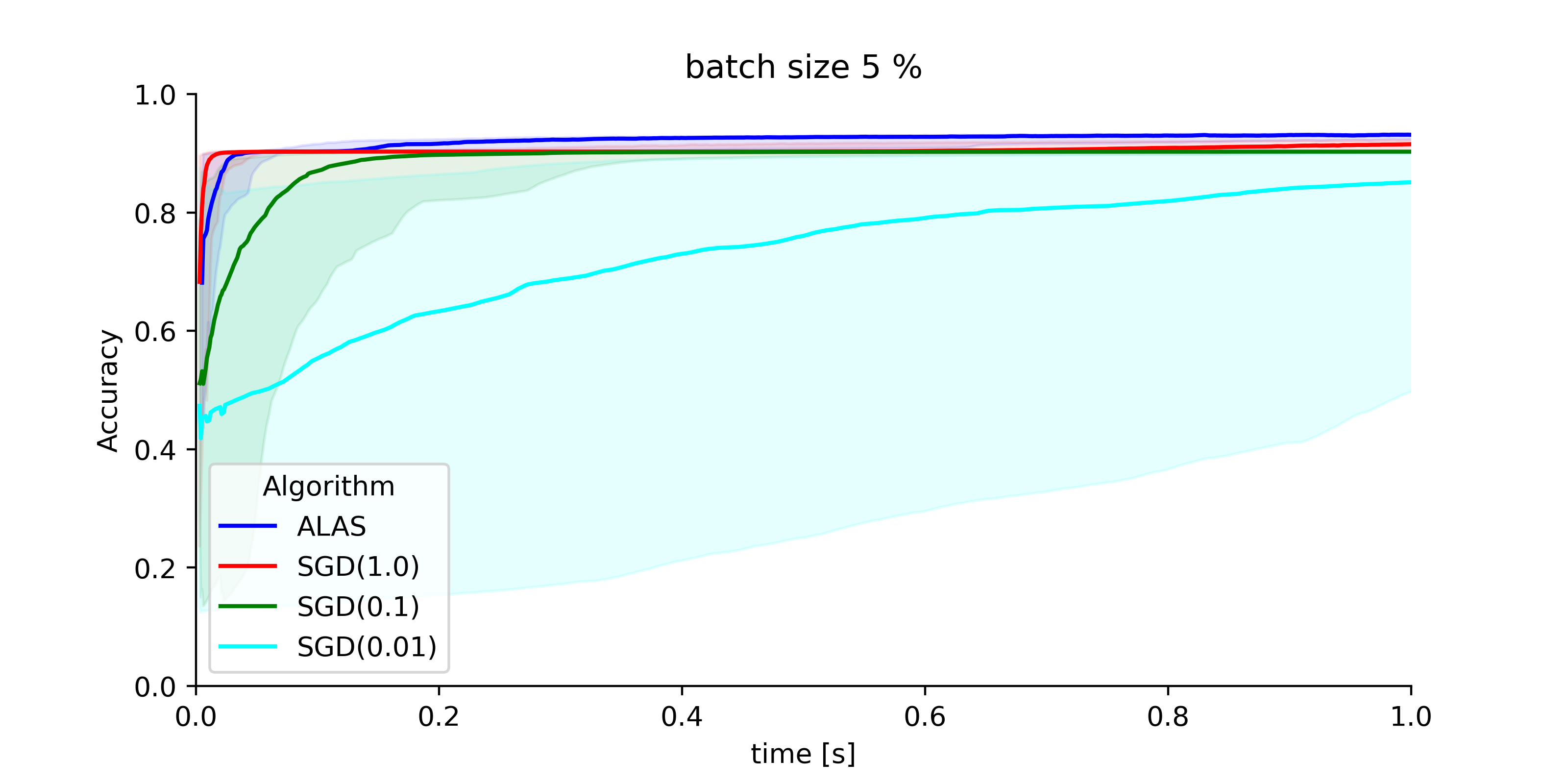}\label{fig:IJCNN_nn_22-1_005_acc_perc}
    }
    ~ 
    \subfigure[22-1 10\%]{
        \includegraphics[width=7.5cm, height=4.1cm]{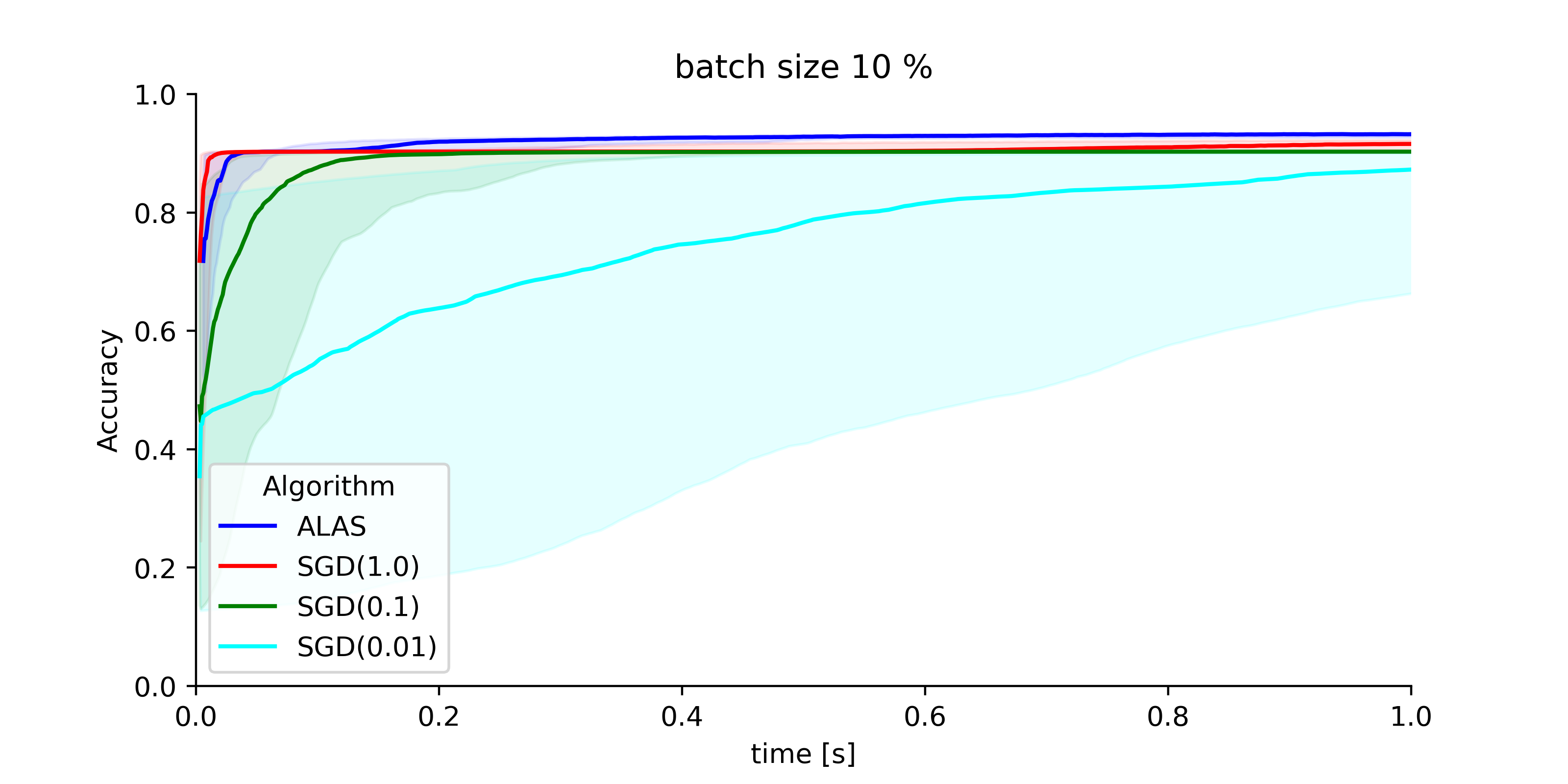}\label{fig:IJCNN_nn_22-1_010_acc_perc}
    }
   \\
    \subfigure[22-1 20\%]{
        \includegraphics[width=7.5cm, height=4.1cm]{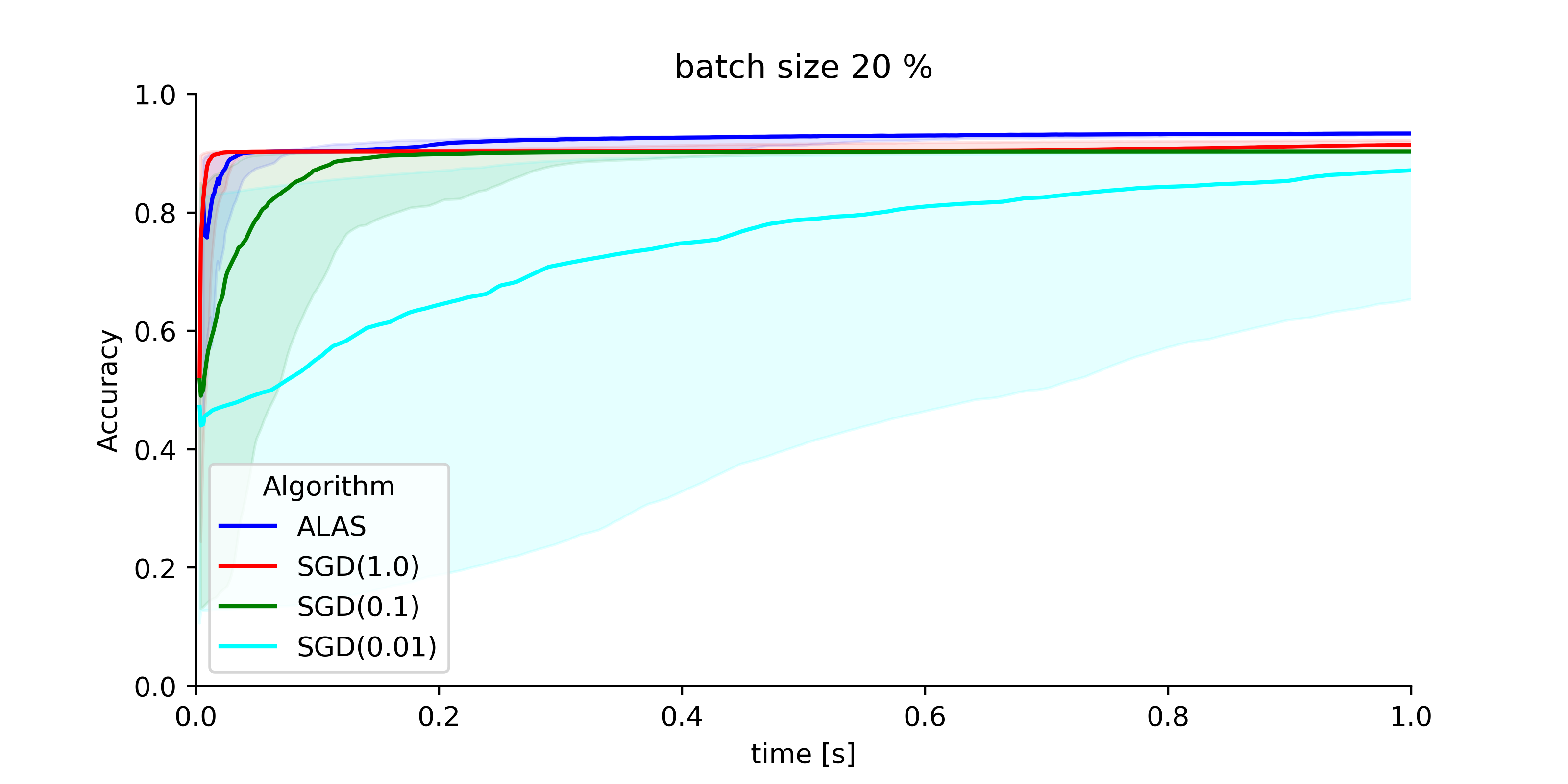}\label{fig:IJCNN_nn_22-1_020_acc_perc}
    }
    ~ 
    \subfigure[22-1 100\%]{
        \includegraphics[width=7.5cm, height=4.1cm]{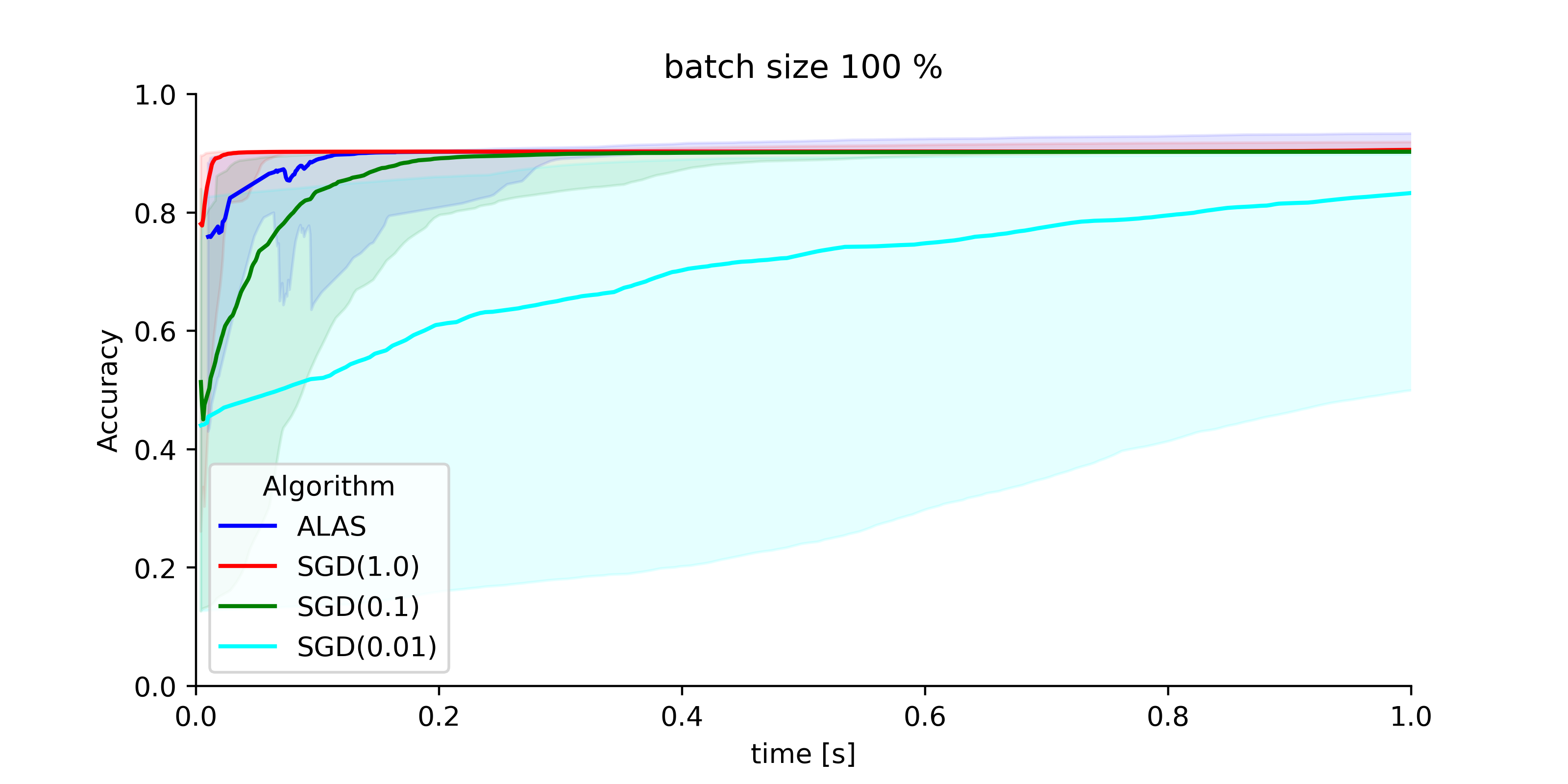}\label{fig:IJCNN_nn_22-1_100_acc_perc}
    }
    \caption{Comparison of levels of training accuracy for the training achieved by 
ALAS and SGD (with a default learning rate 0.01) on the IJCNN1 dataset with a simple neural network with 22 input neurons, reporting
the median and $95\%$ confidence interval across 200 runs}
    \label{fig:IJCNN_nn_22_acc}
\end{figure}

\begin{figure}
    \centering
    \subfigure[22-4-1 5\%]{
        \includegraphics[width=7.5cm, height=4.1cm]{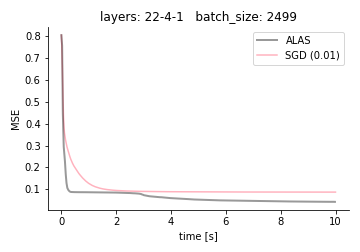}\label{fig:IJCNN_nn_22-4-1_0.05}
    }
    ~ 
    \subfigure[22-4-1 10\%]{
        \includegraphics[width=7.5cm, height=4.1cm]{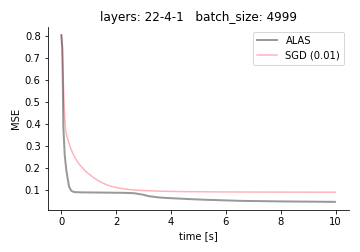}\label{fig:IJCNN_nn_22-4-1_0.10}
    }
   \\
    \subfigure[22-4-1 20\%]{
        \includegraphics[width=7.5cm, height=4.1cm]{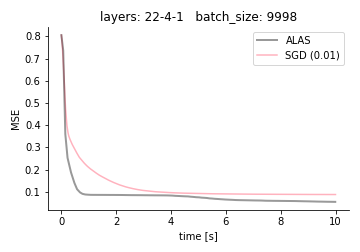}\label{fig:IJCNN_nn_22-4-1_0.20}
    }
    ~ 
    \subfigure[22-4-1 100\%]{
        \includegraphics[width=7.5cm, height=4.1cm]{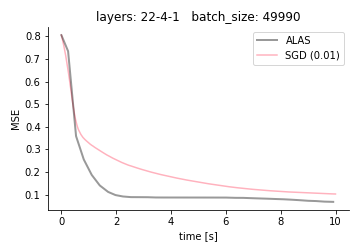}\label{fig:IJCNN_nn_22-4-1_1.00}
    }
    \caption{Comparison of  ALAS and SGD (with a default learning rate 0.01) on the IJCNN1 dataset with a simple neural network with 22 input neurons, 4 neurons in the hidden layer and an output neuron.}
    \label{fig:IJCNN_nn_22-4-1}
\end{figure}

\begin{figure}
    \centering
    \subfigure[22-4-4-1 5\%]{
        \includegraphics[width=7.5cm, height=4.1cm]{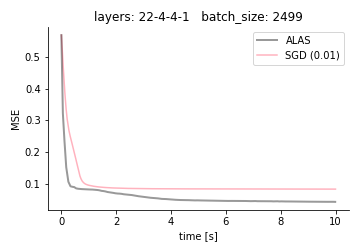}\label{fig:IJCNN_nn_22-4-4-1_0.05}
    }
    ~ 
    \subfigure[22-4-4-1 10\%]{
        \includegraphics[width=7.5cm, height=4.1cm]{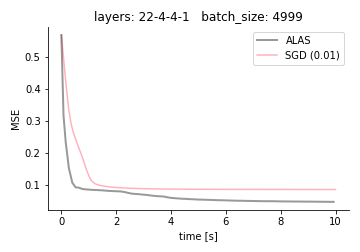}\label{fig:IJCNN_nn_22-4-4-1_0.10}
    }
   \\
    \subfigure[22-4-4-1 20\%]{
        \includegraphics[width=7.5cm, height=4.1cm]{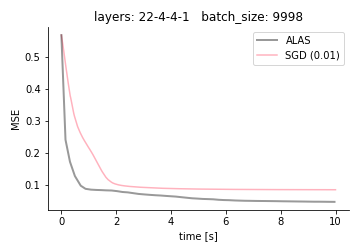}\label{fig:IJCNN_nn_22-4-4-1_0.20}
    }
    ~ 
    \subfigure[22-4-4-1 100\%]{
        \includegraphics[width=7.5cm, height=4.1cm]{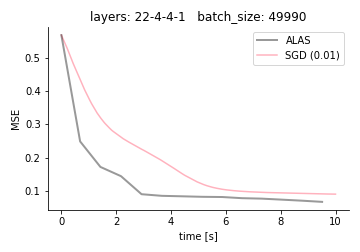}\label{fig:IJCNN_nn_22-4-4-1_1.00}
    }
    \caption{Comparison of  ALAS and SGD (with a default learning rate 0.01)  on the IJCNN1 dataset with a simple neural network with 22 input neurons, two hidden layers each with 4 neurons and an output neuron.}
    \label{fig:IJCNN_nn_22-4-4-1}
\end{figure}

\subsection{Transfer learning using MNIST} 
\label{subsec:numMNIST}
The current implementation of ALAS is very suitable for transfer learning, where only a few layers of a 
neural network need to be trained (i.e. only a few parameters need to be optimized). 
To demonstrate this, we first trained all the parameters in a convolutional neural network (CNN)
on a subset of the MNIST dataset \cite{lecun2010mnist} (classification of digits 0 - 7). 
The CNN had 9 convolution layers 
(with 16 filters 
 for the first layer, 32 filters for the second one, 64 filters for the other layers), 
followed by a global max pooling layer and a single dense classification layer with softmax activation 
function and 25\% dropout for the eight digits (8 output neurons). 
Once this network was trained, the classification layer was removed and the extracted features were used 
for transfer learning to classify the remaining two digits (8 and 9) images of the MNIST dataset, 
that were unseen during the training of the original CNN. The new classification layers were then 
trained using the ALAS and the SGD algorithms (we did the tests using the following  step sizes {1, 0.6, 0.3, 0.1}), with 
sample sizes corresponding to 100\%, 20 \%, 10 \%, and  5\% of the entire data. The total number of samples with the digits 8 and 9 is $N=11800$ and the input to the trained layer are 8 neurons. The optimized function was the MSE.
As shown in Table~\ref{tab:alg_comparison_transfer_1}, ALAS outperformed the SGD variants in most of the 
runs. Note that when the whole dataset was used, i.e. $\mathcal{S}_k=\{1,\dots,N\}$ for all $k$, ALAS 
outperforms the SGD variants by a significant margin for all four neural network architectures.
\begin{table}[h!]
\centering
{\small
\begin{tabular}{c  c c c c | c c c c c }
\multicolumn{5}{c|}{Layers: 8-1} & \multicolumn{5}{c}{Layers: 8-1-1} \\ \hline
 alg. & $\skn_k$ & min loss &  loss [8-10]s & iter.  & alg. & $\skn_k$ & min loss & loss [8-10]s & iter. \\ \hline
ALAS & 5\% &  \textbf{0.2739} & \textbf{0.2752} & 12220 & ALAS & 5\% &  \textbf{0.2721} & \textbf{0.2815} & 11887 \\
 SGD (1.0) & 5\% & 0.2796  & 0.2801 &  20264& SGD (0.3) & 5\% & 0.2796  & 0.2801 &  20264\\ \hline
 ALAS & 10\% & \textbf{0.2732}  & \textbf{0.2740} & 11514 & ALAS & 10\% & \textbf{0.2667}  & \textbf{0.2722} & 10116 \\
 SGD (1.0) & 10\% & 0.2795  & 0.2799 & 20872 & SGD (0.6) & 10\% & 0.2738  & 0.2753 & 20649 \\ \hline
 ALAS & 20\% & \textbf{0.2728}   & \textbf{0.2732} &  10691& ALAS & 20\% & \textbf{0.2657}   & \textbf{0.2704} &  8593\\
 SGD (1.0) & 20\% &  0.2796 & 0.2800 & 20141& SGD (0.6) & 20\% &  0.2736 & 0.2749 & 20327 \\ \hline
 ALAS & 100\% & \textbf{0.2587}  & \textbf{0.2592}&  5118 & ALAS & 100\% & \textbf{0.2362}  & \textbf{0.2363}&  3630\\
 SGD (1.0) & 100\% & 0.2800  & 0.2803 &  17832& SGD (0.3) & 100\% & 0.2751  & 0.2754 &  16697\\ \hline 
 \multicolumn{5}{c|}{Layers: 8-2-1} & \multicolumn{5}{c}{Layers: 8-4-1} \\ \hline
 ALAS & 5\% &  \textbf{0.2649} & 0.2762 & 6964  & ALAS & 5\% &  0.2516 & \textbf{0.2567} & 2245 \\
 SGD (1.0) & 5\% & 0.2667  & \textbf{0.2702} &  20975 & SGD (1.0) & 5\% & \textbf{0.2524}  & 0.2550 &  1356\\ \hline
 ALAS & 10\% & \textbf{0.2632}  & \textbf{0.2674} & 5778 & ALAS & 10\% & 0.2524  & \textbf{0.2550} & 1356 \\
 SGD (1.0) & 10\% & 0.2681  & 0.2705 & 19502 & SGD (1.0) & 10\% & \textbf{0.2505}  & 0.2554 & 19502 \\ \hline
 ALAS & 20\% & \textbf{0.2602}   & \textbf{0.2625} &  3609& ALAS & 20\% & 0.2529   & \textbf{0.2542} &  811\\
 SGD (1.0) & 20\% &  0.2695 & 0.2712 & 16311& SGD (1.0) & 20\% &  \textbf{0.2504} & 0.2555 & 16211 \\ \hline
 ALAS & 100\% & \textbf{0.2295}  & \textbf{0.2301}&  1151& ALAS & 100\% & \textbf{0.2227}  & \textbf{0.2247}&  220\\
 SGD (1.0) & 100\% & 0.2754  & 0.2762 &  7232 & SGD (1.0) & 100\% & 0.2615  & 0.2645 &  6377\\ \hline 
\end{tabular}
}
\caption{Results reached over the given time period $t =10 ~s$ on the transfer learning task. }
\label{tab:alg_comparison_transfer_1}
\end{table}
\subsubsection{Transfer learning without pre-computed features}
\label{subsubsec:no_precomputed}
The experiment above is using pre-computed features, i.e. the original images were run through the pre-trained network once and the saved output was used for the optimization of the top layers. While this technique brings massive speed-ups as it is not necessary to compute the outputs for the same images repeatedly, it can only be used when there are repeated images. This prevents the use of online data augmentation, where the individual images are randomly transformed (e.g. rotations, scaling) for each batch and the output from the pre-trained layers has to be recomputed every time. 

This is a very favorable scenario for the ALAS algorithm as it usually needs much less updates than the SGD even though each update is more costly in terms of operations --- the need to recompute the features every time adds fixed costs to update steps of both SGD and ALAS. We have run the same experiment as above but without the pre-computed features with the time limit 1 hour instead of 10 seconds and with different weight initialization; the results are shown in Table~\ref{tab:alg_comparison_transfer_1b} and in Fig~\ref{fig:TransferMNIST01NBF_8-4-1}. Figures~\ref{fig:TransferMNIST01NBF_nn_8-4-1_0.20xlim} and \ref{fig:TransferMNIST01NBF_nn_8-4-1_0.20ylim} shows details of the run depicted in Figure~\ref{fig:TransferMNIST01NBF_nn_8-4-1_0.20}. The number of iterations was similar for both algorithms (only small differences possibly caused by the different utilization of the cloud server) as the feature evaluation step was the dominant for the same \textit{8-4-1} layer configuration for all sample sizes. For higher number of weights, the ALAS performs comparably less iterations than the SGD as the update step is no longer negligible compared to the feature evaluation using the fixed weight neural network (but that is still quite costly and thus we do not observe order of magnitude differences here) as can be seen in the networks with configurations \textit{8-16-4-1} (only ALAS was run to show the small drop in number of iterations) and \textit{8-32-16-1} (the SGD still performs similar number of iterations as for the \textit{8-4-1} layer configuration).
\begin{table}[h!]
\centering
{\small
\begin{tabular}{c  c c c c | c c c c c }
\multicolumn{5}{c|}{Layers: 8-4-1} & \multicolumn{5}{c}{Layers: 8-16-4-1} \\ \hline
 alg. & $\skn_k$ & min loss &  loss [3400-3600]s & iter.  & alg. & $\skn_k$ & min loss & loss [3400-3600]s & iter. \\ \hline
 ALAS & 10\% & \textbf{0.1805}  & \textbf{0.2483}&  1629& ALAS & 20 \% & 0.1884 & 0.2294 & 748 \\ \hline
 SGD (1.0) & 10\% & 0.2003 & 0.2706 &  1709& SGD (1.0) & 100\% & 0.2255  & 0.2680 &  870\\ \hline 
 SGD (0.1) & 10\% & 0.2144  & 0.2853 &  1683& SGD (0.1) & 100\% & 0.2326  & 0.2796 &  869\\ \hline 
 SGD (0.01) & 10\% & 0.2417  & 0.3159 &  1674& SGD (0.01) & 100\% & 0.2517  & 0.3028 & 878\\ \hline 
\hline
 ALAS & 20\% & \textbf{0.2034}  & \textbf{0.2443}&  871& ALAS & 20 \% & \textbf{0.2076} & \textbf{0.2664} & 562 \\ \hline
 SGD (1.0) & 20\% & 0.2356  & 0.2802 &  878 & SGD (0.1) & 20\%  & 0.2230 & 0.2709 &  887\\ \hline
 SGD (0.1) & 20\% & 0.2488  & 0.2913 &  848& SGD (0.01) & 20\% & 0.2480 & 0.2932 &  876\\ \hline 
 SGD (0.01) & 20\% & 0.3657  & 0.3946 &  900& SGD (1.0) & 20\% & 0.3015 & 1.9834 &  886\\ \hline 
\end{tabular}
}
\caption{Results reached over the given time period $t =3600 ~s$ on the transfer learning task.}
\label{tab:alg_comparison_transfer_1b}
\end{table}

\begin{figure}
    \centering
    \subfigure[8-4-1 100\%]{
        \includegraphics[width=7.5cm, height=4.1cm]{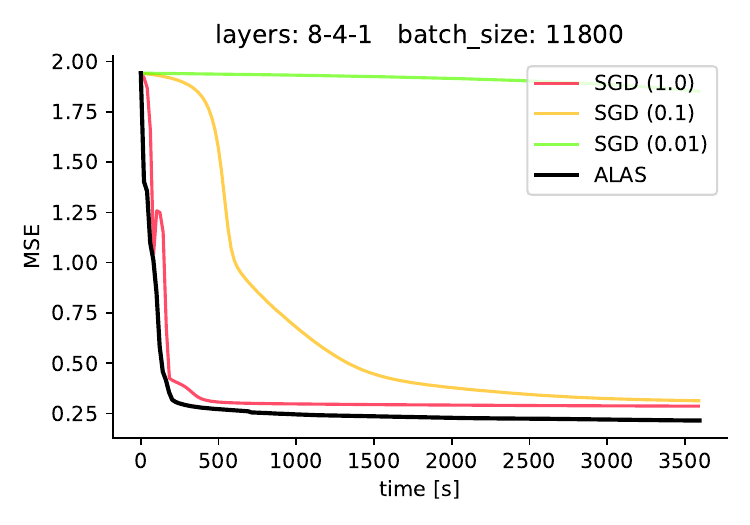}\label{fig:TransferMNIST01NBF_nn_8-4-1_1.00}
    }~
    \subfigure[8-4-1 20\%]{
        \includegraphics[width=7.5cm, height=4.1cm]{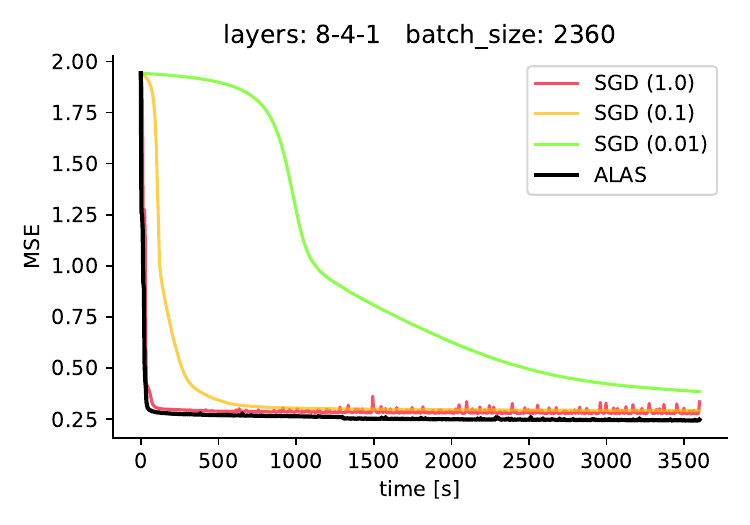}\label{fig:TransferMNIST01NBF_nn_8-4-1_0.20}
    }
    \subfigure[8-4-1 20\% (first 160 s)]{
        \includegraphics[width=7.5cm, height=4.1cm]{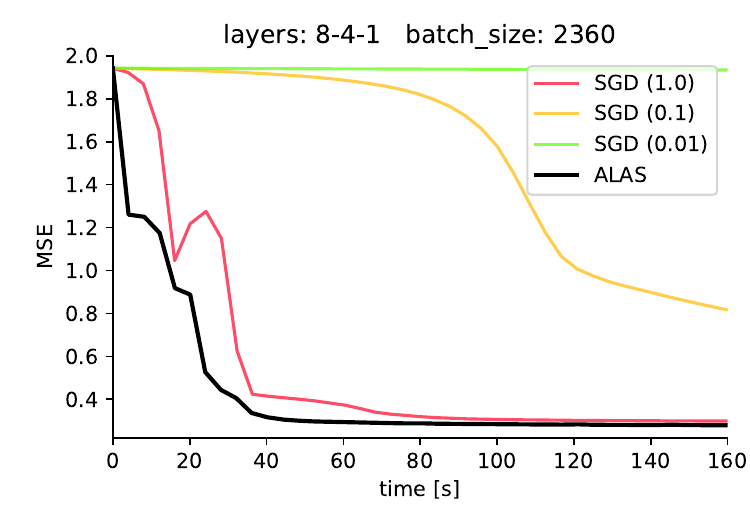}\label{fig:TransferMNIST01NBF_nn_8-4-1_0.20xlim}
    }~
    \subfigure[8-4-1 20\% (detail)]{
        \includegraphics[width=7.5cm, height=4.1cm]{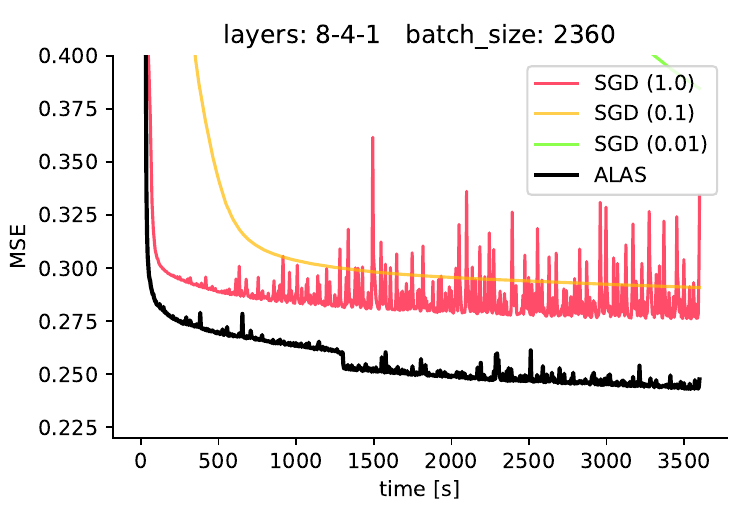}\label{fig:TransferMNIST01NBF_nn_8-4-1_0.20ylim}
    }
    \caption{Comparison of  ALAS and SGD (with various learning rates) on the transfer learning task without pre-computed features with a simple neural network with 8 input neurons, 4 hidden neurons and an output neuron.}
    \label{fig:TransferMNIST01NBF_8-4-1}
\end{figure}

\section{Conclusion} 
\label{sec:conc}
In this paper, we presented a line-search method for stochastic optimization, 
wherein the Hessian, gradient, and function values must be estimated through 
subsampling and cannot be obtained exactly. Using probabilistically accurate 
models, we derived a complexity result on the expected number of iterations 
until an approximate measure of stationary is reached for the current model. 
This result in expectation is complementary to those holding with high 
probability, i.e. with probability of drawing an accurate sample at every 
iteration. 
We also proposed a practical strategy to assess whether the current iterate is 
close to a sample point for the original objective, that does not require the 
computation of the full function. Our numerical experiments showed the potential 
of the proposed approach on several machine learning tasks, including transfer 
learning.

We believe that the results of \revised{this paper encourage further study 
of} second-order algorithms despite the prevailing paradigm of using 
first-order methods for their computationally cheap iterations. 
\revised{In particular, our approach could be helpful in generalizing other 
line-search techniques to the context of subsampled function values while our 
theoretical analysis, that captures the worst-case behavior of the problem, 
can likely be refined to exploit the problem structure.}
%
%
\revised{
Finally, a high-performance implementation of our method would benefit from 
more practical features such as matrix-free operations and iterative linear algebra 
(see preliminary numerical results in~\cite{kungurtsev2016}) as well as 
adaptive sampling batch sizes, following the recent trends in stochastic 
optimization~\cite{bellavia2021stochastic}. A number of difficulties 
arise in extending our complexity analysis to these frameworks, yet developing 
practical schemes with such guarantees is an important and exciting perspective 
for future research.}
\subsection*{Acknowledgments}

{\small We are indebted to Courtney Paquette and Katya Scheinberg for raising an issue with the first version of this paper, that lead to significant improvement of its results.
\revised{We would also like to thank the reviewers of this paper for their insightful comments.}}


\bibliographystyle{plain}
\bibliography{thebib}

\begin{appendices}

\section{Supplementary proofs}
\subsection{Proof of Lemma~\ref{lm:bound:alpha_kd_k}}
 
	We consider in turn the three possible steps that can be taken at iteration $k$, and obtain a lower bound 
	on the amount  $\alpha_k \|d_k\| $ for each of those. 
%
	
	\textbf{Case 1:} $\lambda_k < -\epsilon^{1/2}$ (negative curvature step). In that case, we apply the same reasoning 
	than in~\cite[Proof of Lemma 1]{royer2018complexity} with 
	the model $\hat f_k$ playing the role of the objective, the backtracking 
	line search terminates with the step length $\alpha_k =\theta^{j_k}$, with 
	$j_k \le \bar{j}_{nc} +1$ and $\alpha_k \ge \frac{3 \theta}{L_H+\eta}$. When $d_k$ is 
	computed as a negative curvature direction, one has $\|d_k\| = -\lambda_k >0$. Hence,
	\[
		\alpha_k \|d_k\| \; \ge \;  \frac{3 \theta}{L_H+\eta}[-\lambda_k]
		\; = \; c_{nc} [-\lambda_k] \ge c \epsilon^{1/2}.
	\]

	\textbf{Case 2:} $\lambda_k > \|g_k\|^{1/2}$ (Newton step). Because the stationarity is not achieved 
	and $\lambda_k > 0$ in this case, we necessarily have 
	$\|\tilde{g}_k\|=\min\{\|g_k\|,\|g^+_k\|\} > \epsilon$. From 
	Algorithm~\ref{alg:alas}, we know that $d_k$ is chosen as the Newton step. Hence, using the argument 
	of~\cite[Proof of Lemma 3]{royer2018complexity} with $\epsilon_H = \epsilon^{1/2}$, the 
	backtracking line search terminates with the step length $\alpha_k =\theta^{j_k}$, where
	\begin{equation*}
	j_k \le \left[\log_{\theta} \left( \sqrt{\frac{3}{L+\eta}} \frac{ \epsilon^{1/2}}
	{\sqrt{U_g}}\right) \right]_{+} + 1 = \bar{j}_n + 1,
	\end{equation*}
	thus the first part of the result holds. 
	If the unit step size is chosen, we have 
	by~\cite[Relation (23)]{royer2018complexity} that
	\begin{equation} \label{eq:normunitnewtonstep}
		\alpha_k \|d_k\| = \|d_k\| \ge \left[\frac{2}{L_H}\right]^{1/2} 
		\|g(x_k+\alpha_k d_k;\mathcal{S}_k)\|^{1/2} \ge c  \epsilon^{1/2}.
	\end{equation}		
	
	Consider now the case $\alpha_k<1$. 
	Using~\cite[Relations (25) and (26), p. 1457]{royer2018complexity}, we have:
	\begin{equation*}
		\|d_k\| \ge \frac{3}{L_H+\eta}\epsilon_H = \frac{3}{L_H+\eta} \epsilon^{1/2} \text{ and }
	\end{equation*}
	\begin{equation*}
		\alpha_k \ge \theta \sqrt{\frac{3}{L_H+\eta}}\epsilon_H^{1/2}\|d_k\|^{-1/2} = 
		\theta \sqrt{\frac{3}{L_H+\eta}}\epsilon^{1/4}\|d_k\|^{-1/2}.
	\end{equation*}
	As a result,
	\begin{eqnarray} \label{eq:normnotunitnewtonstep}
		\alpha_k \|d_k\| & \ge &\theta\left[\frac{3}{L_H+\eta}\right]^{1/2}
		\epsilon^{1/4}\|d_k\|^{-1/2}\|d_k\| \ge  \left[\frac{3\theta}{L_H+\eta}\right]\epsilon^{1/2} \ge c  \epsilon^{1/2}.
	\end{eqnarray}

	\textbf{Case 3:} (Regularized Newton step) This case occurs when the conditions for the other two cases fail, 
	that is, when $-\epsilon^{1/2} \le \lambda_k \le \|g_k\|^{1/2}$. We again exploit the fact that 
	the stationarity is not achieved  to deduce that we necessarily have 
	$\|\tilde{g}_k\|=\min\{\|g_k\|,\|g^+_k\|\} > \epsilon$. This in turn implies that 
	$\min\{\|\tilde{g}_k\|\epsilon^{-1/2},\epsilon^{1/2}\} \ge \epsilon^{1/2}$.  
	As in the proof of the previous lemma, we apply the theory of~\cite[Proof of Lemma 4]{royer2018complexity} 
	using $\epsilon_H = \epsilon^{1/2}$. We then know that the backtracking line search terminates with the 
	step length $\alpha_k =\theta^{j_k}$, with 
	\begin{equation*}
		j_k \le \left[\log_{\theta} \left( \frac{6}{L_H+\eta}\frac{\epsilon}{U_g}\right)\right]_{+} + 1 
		= \bar{j}_{rn} +1.
	\end{equation*}
	We now distinguish between the cases $\alpha_k=1$ and $\alpha_k < 1$. If the unit step size is chosen, we 
	can use~\cite[relations 30 and 31]{royer2018complexity}, where $\nabla f(x_k+d_k)$ and $\epsilon_H$ 
	are replaced by $g_k^+$ and $\epsilon^{1/2}$, respectively. This gives 
	\begin{eqnarray*}
		\alpha_k \|d_k\| = \|d_k\| 
		&\ge &\frac{1}{1+\sqrt{1+L_H/2}}\min\left\{\|g_k^+\|/\epsilon^{1/2},\epsilon^{1/2}\right\}.
	\end{eqnarray*}	
    Therefore, if the unit step is accepted, one has by~\cite[equation 31]{royer2018complexity}
    \begin{eqnarray} \label{eq:normunitregnewtonstep}
    		\alpha_k \|d_k\| 
    		&\ge &\frac{1}{1+\sqrt{1+L_H/2}}
    		\min\left\{\|\tilde{g}_k\|\epsilon^{-1/2},\epsilon^{1/2}\right\}.
    \end{eqnarray}
	Considering the case $\alpha_k<1$ and using~\cite[equation 32, p. 1459]{royer2018complexity}, we have:
	\begin{equation*}
		\alpha_k \ge \theta \frac{6}{L_H+\eta}\epsilon_H\|d_k\|^{-1} = 
		\frac{6\theta}{L_H+\eta}\epsilon^{1/2}\|d_k\|^{-1},
	\end{equation*}
	which leads to
	\begin{eqnarray} \label{eq:normnotunitregnewtonstep}
		\alpha_k \|d_k\| & \ge &\frac{6\theta}{L_H+\eta}\epsilon^{1/2}.
	\end{eqnarray}
	Putting~\eqref{eq:normunitregnewtonstep} and~\eqref{eq:normnotunitregnewtonstep} together, we obtain
	\begin{eqnarray*}
		\alpha_k\|d_k\| &\ge &\min\left\{\frac{1}{(1+\sqrt{1+L_H/2})^3},
		\left[\frac{6\theta}{L_H+\eta}\right]^3\right\}\, \min\{\|\tilde{g}_k\| \epsilon^{-1/2}, 
		\epsilon^{1/2}\} = c_{rn}\min\{\|\tilde{g}_k\| \epsilon^{-1/2}, \epsilon^{1/2}\} \ge c  \epsilon^{1/2}.
	\end{eqnarray*}
	
	By putting the three cases together, we arrive at the desired conclusion.

\subsection{Proof of Lemma~\ref{lem:bound_on_dk}}
Indeed, since the lemma trivially holds if $\|d_k\|=0$,  we only 
need to prove that it holds for $\|d_k\|>0$. We consider three disjoint cases:

\textbf{Case 1:} $\lambda_k < -\ecurv$. Then the negative curvature step is taken and $\|d_k\|=|\lambda_k|\le U_H$.

\textbf{Case 2:} $\lambda_k > \|g_k\|^{1/2}$. We can suppose that $\|g_k\|>0$ because 
otherwise $\|d_k\|=0$. Then, $d_k$ is a Newton step with
$$\|d_k\|\le \|H_k^{-1}\|\|g_k\|\le \|g_k\|^{-1/2}\|g_k\|\le \|g_k\|^{1/2}\le U_g^{1/2}.$$

\textbf{Case 3:} $-\ecurv \le \lambda_k \le \|g_k\|^{1/2}$. As in Case 2, we 
suppose that $\|g_k\|>0$ as $\|d_k\|=0$ if this does not hold. 
Then, $d_k$ is a regularized Newton step with
\[
  \|d_k\| = \| (H_k + (\|g_k\|^{1/2} + \epsilon^{1/2})\mathbbm{I}_n)^{-1} g_k\|  
   \le \frac{\|g_k\|}{\lambda_k + \|g_k\|^{1/2} + \ecurv} \le \|g_k\|^{1/2} \le U_g^{1/2}.
\]
where the last inequality uses $\lambda_k + \ecurv \ge 0$ and $\|g_k\|>0$. 

\subsection{Proof of Theorem~\ref{th:complexity0}}

To prove Theorem~\ref{th:complexity0}, we will combine the following three standard lemmas.
\begin{lemma} \label{lemma:uniformsampleHessian} 
Under Assumption 2, consider an iterate $x_k$ of 
	Algorithm 1. For any $p \in (0,1)$, if the 
	sample set $\mathcal{S}_k$ is chosen to be of size
	\begin{equation} \label{eq:unifsamplesizeHessian}
		\bar{\skn} \ge \frac{1}{N}\frac{16L^2}{\delta_H^2}\ln\left(\tfrac{2N}{1-p}\right),
	\end{equation}
	then
	\[
		\Pr\left( \left\|H(x_k;\mathcal{S}_k)-\nabla^2 f(x_k) \right\| \le \delta_H  
		\middle| \mathcal{F}_{k-1} \right) \ge p.
	\]
\end{lemma}

\begin{proof}{Proof of Lemma \ref{lemma:uniformsampleHessian}.}
	See~\cite[Lemma 16]{xu2019newton}; note that here we are using $L_i$ (Lipschitz constant 
	of $\nabla f_i$) as a bound on $\|\nabla^2 f_i(x_k)\|$, and that we are providing a 
	bound on the sampling fraction $\bar{\skn} = \tfrac{|\mathcal{S}_k|}{N}$. 
	See also~\cite[Theorem 1.1]{tropp2015matrix}, considering the norm as related to the 
	maximum singular vector. 
\end{proof}

By the same reasoning as for Lemma~\ref{lemma:uniformsampleHessian}, but in one dimension, we can 
readily provide a sample size bound for obtaining accurate function values. To this end, we define 
\begin{equation} \label{eq:fup}
	f_{\up} \ge \max_{k} \max_{i=1,..,N} f_i(x_k).
\end{equation}
Note that such a bound necessarily exists when the iterates are contained in a compact 
set. Specific structure of the problem can also guarantee such a bound, even tough it may exhibit 
dependencies on the problem's dimension. For instance, in the case of classification and logistic 
regression, one has $f_{\up}=1$, while in the case of (general) regression, one has 
$f_{\up}\le C_1+C_2\|x\|^2=\mathcal{O}(n)$. We emphasize that both of these bounds can be very 
pessimistic.

\begin{lemma} \label{lemma:uniformsamplefunction}
	Under Assumption~\ref{as:f:lower}, 
consider an iterate $x_k$ of 
	Algorithm~\ref{alg:alas}. 
For any $p \in (0,1)$, if the 
	sample set $\mathcal{S}_k$ is chosen to be of size
	\begin{equation} \label{eq:unifsamplesizefunction}
		\bar{\skn} \ge \frac{1}{N}\frac{16 f_{\up}^2}{\delta_f^2}
		\ln\left(\tfrac{2}{1-p}\right),
	\end{equation}
	then
	\[
		\Pr\left( \left|\hat f(x_k;\mathcal{S}_k)- f(x_k) \right| \le \delta_f  
		\middle| \mathcal{F}_{k-1} \right) \ge p.
	\]
\end{lemma}

\begin{proof}{Proof of Lemma \ref{lemma:uniformsamplefunction}.}
	The proof follows that of~\cite[Lemma 4.1]{xu2019newton} by considering 
	$\hat f(x_k;\mathcal{S}_k)$ and $f(x_k)$ as one-dimensional matrices.  
\end{proof}

\begin{lemma} \label{lemma:uniformsamplegrad}
	Under Assumption~\ref{as:f:H}, consider an iterate $x_k$ of 
	Algorithm~\ref{alg:alas}. For any $p \in (0,1)$, if the 
	sample set $\mathcal{S}_k$ is chosen to be of size
	\begin{equation} \label{eq:unifsamplesizegradient}
		\bar{\skn} \ge \frac{1}{N}\frac{U_g^2}{\delta_g^2}
		\left[1+\sqrt{8\ln\left(\tfrac{1}{1-p}\right)}\right]^2,
	\end{equation}
	then
	\[
		\Pr\left( \left\|g(x_k;\mathcal{S}_k)- \nabla f(x_k) \right\| \le \delta_g  
		\middle| \mathcal{F}_{k-1} \right) \ge p.
	\]
\end{lemma}

\begin{proof}{Proof of Lemma \ref{lemma:uniformsamplegrad}.}
	See~\cite[Lemma 2]{roosta2016subsampled}. 
\end{proof}

We now combine the three previous lemmas to obtain an overall result indicating the required $\bar{\skn}$ such that
Lemmas~\ref{lemma:uniformsampleHessian},~\ref{lemma:uniformsamplefunction} and~\ref{lemma:uniformsamplegrad}
simultaneously hold, i.e., the event $I_k$ holds. 


Indeed, let $\delta_f=\tfrac{\eta}{24}c^3\epsilon^{3/2}$, 
	$\delta_g=\ccg\epsilon$, $\delta_H=\cch\epsilon^{1/2}$ and note that,
\[
I_k \equiv I_k^h \cap I_k^g \cap I_k^f,
\]
where 
\[
\begin{array}{l}
I_k^h:= \{\left\|H(x_k;\mathcal{S}_k)-\nabla^2 f(x_k) \right\| \le \delta_H\}\\
I_k^g:= \{\left\|g(x_k;\mathcal{S}_k)- \nabla f(x_k) \right\| \le \delta_g \} \\
I_k^f:= \{\left|\hat f(x_k;\mathcal{S}_k)- f(x_k) \right| \le \delta_f \}.
\end{array}
\]
Using the required conditions on $\skn_n$, Lemmas~\ref{lemma:uniformsampleHessian},~\ref{lemma:uniformsamplefunction} and~\ref{lemma:uniformsamplegrad} imply
$$
\mathbb{P}\left((I_k^f)^c \middle| \mathcal{F}_{k-1}\right) \le 1 - \hat p ~,~~  \mathbb{P}\left((I_k^g)^c \middle| \mathcal{F}_{k-1}\right) \le 1 - \hat p,~~\mbox{and}~~\mathbb{P}\left((I_k^h)^c \middle| \mathcal{F}_{k-1}\right) \le 1 - \hat p.
$$
In the other hand, one has
\begin{equation*}\label{eq:probI}
\begin{array}{l}
\mathbb{P}\left((I_k)^c \middle| \mathcal{F}_{k-1} \right)= 
\mathbb{P}\left((I_k^f)^c \middle| \mathcal{F}_{k-1}\right)+ \mathbb{P}\left((I^g_k)^c \middle| \mathcal{F}_{k-1}\right)+
\mathbb{P}\left((I^g_k)^c \middle| \mathcal{F}_{k-1}\right) -\mathbb{P}\left((I_k^f)^c\cap (I_k^g)^c \middle| \mathcal{F}_{k-1}\right) \\
\qquad\qquad -\mathbb{P}\left((I_k^g)^c\cap (I_k^h)^c \middle| \mathcal{F}_{k-1}\right)
-\mathbb{P}\left((I_k^f)^c\cap (I_k^h)^c \middle| \mathcal{F}_{k-1}\right)+\mathbb{P}\left((I_k^f)^c\cap (I_k^g)^c\cap (I_k^h)^c \middle| \mathcal{F}_{k-1}\right) \\ 
\qquad \le  \mathbb{P}\left((I_k^f)^c \middle| \mathcal{F}_{k-1}\right)+ \mathbb{P}\left((I_k^g)^c \middle| \mathcal{F}_{k-1}\right)+
\mathbb{P}\left((I_k^h)^c \middle| \mathcal{F}_{k-1}\right)  \\ \qquad\qquad
+\mathbb{P}\left((I_k^f)^c \middle| (I_k^h)^c , (I_k^g)^c , \mathcal{F}_{k-1}\right)\mathbb{P}\left((I_k^h)^c \middle| (I_k^g)^c , \mathcal{F}_{k-1}\right)
\mathbb{P}\left((I_k^g)^c \middle| \mathcal{F}_{k-1}\right) \\ 
\qquad
\le \mathbb{P}\left((I_k^f)^c \middle| \mathcal{F}_{k-1}\right)+ 2\mathbb{P}\left((I_k^g)^c \middle| \mathcal{F}_{k-1}\right)+
\mathbb{P}\left((I_k^h)^c \middle| \mathcal{F}_{k-1}\right) \le 4(1-\hat p).
\end{array}
\end{equation*}
Hence,
%
\[
\mathbb{P}(I_k|\mathcal{F}_{k-1}) \ge 1-4(1-\hat p) = p,
\]
meaning that the model sequence is $p$-probabilistically $(\delta_f,\delta_g,\delta_H)$-accurate, 
thus results from Section~\ref{subsec:cvwcc:iterwcc} hold. 
$\square$

\subsection{Proof of Proposition~\ref{prop:stop}}

	In this proof, we will use the notation $\Pr_k(\dots) = \Pr(\cdot|\mathcal{F}_{k-1})$, 
	as well as the random events
	\[
		\begin{array}{lll}
		E &= &\left\{\mbox{One of the iterates in
		$\{x_{k+j}\}_{j=0..J}$ is $((1+\ccg)\epsilon,(1+\cch)\epsilon^{1/2})$-function stationary}
		\right\}, \\
		 & & \\
 		E_j &= &\left\{\mbox{The iterate
 		$x_{k+j}$ is $(\epsilon,\epsilon^{1/2})$-model stationary}\right\} ~~~
 		\forall j=0,\dots,J.
 		\end{array}
	\]
	For every $j=0,\dots,J$, we have $E_j \in \mathcal{F}_{k+j}$ and $I_{k+j} \in \mathcal{F}_{k+j}$ 
	(where $I_j$ is the event introduced in Definition~\ref{def:accurate:proba}). 		
	Moreover, the events $E_j$ and $I_{k+j}$ are conditionally independent:
	\begin{equation} \label{eq:condindeppropstop}
		\forall j=0,\dots,J, \quad \Pr_k\left( E_j \cap I_{k+j} \middle| \mathcal{F}_{k+j-1} \right) \; = \;
		\Pr_k\left( E_j  \middle| \mathcal{F}_{k+j-1} \right) 
		\Pr_k\left( I_{k+j}  \middle| \mathcal{F}_{k+j-1} \right).
	\end{equation}
	This conditional independence holds because 
	$x_{k+j} \in \mathcal{F}_{k+j-1}$, and the model $\hat f_{k+J}$ is constructed 
	independently of $x_{k+j}$ by assumption.
	Using these events, we can reformulate the statement of the theorem as
	\begin{equation*} 
		\Pr_k\left( E \middle| E_0,\dots,E_J\right) \; \ge \; 1-(1-p)^{J+1}.
	\end{equation*}
	Now, by Lemma~\ref{lemma:modeltruestationarity}, 
	\[
		\Pr_k\left( E \middle| E_0,\dots,E_J\right) \; \ge \; 
		\Pr_k\left( \bigcup_{0 \le j \le J} I_{k+j}  \middle| E_0,\dots,E_J \right) 
		= 1-\Pr_k\left( \bigcap_{0 \le j \le J} \bar{I}_{k+j} \middle| E_0,\dots,E_J \right) .
	\]
	Thus,to obtain the desired result, it suffices to prove that
	\begin{equation}\label{eq:goalpropstop}
		\Pr_k\left( \bigcap_{0 \le j \le J} \bar{I}_{k+j} \middle| E_0,\dots,E_J \right) 
		\; \le \; (1-p)^{J+1}.
	\end{equation}
	We now make use of the probabilistically accuracy property. For every $j=0,\dots,J$, we have 
	\begin{equation} \label{eq:condprobapropstop}
		\Pr_k\left(I_{k+j} | A \right) \ge p, 
	\end{equation}
	for any set of events $A$ belonging to the $\sigma$-algebra 
	$\mathcal{F}_{k+j-1}$~\cite[Chapter 5]{durrett2010probatheory}. In particular, for any $j \ge 1$, 
	$\Pr_k\left(I_{k+j}|I_k,\dots,I_{k+j-1},E_0,\dots,E_j\right) \ge p$.
	Returning to our target probability, we have:
	{\small
	\begin{eqnarray*}
	&&\Pr_k\left( \bigcap_{0 \le j \le J} \overline{I}_{k+j} \middle| E_0,\dots,E_J \right) 
		=\frac{\Pr_k\left(\left\{\bigcap_{0 \le j \le J} \bar{I}_{k+j}\right\} \cap E_J 
		\middle| E_0,\dots,E_{J-1} \right)}{\Pr_k\left(E_J \middle| E_0,\dots,E_{J-1} \right)} \\
		&&\quad= \frac{\Pr_k\left(\bar{I}_J \cap E_J \middle| E_0,\dots, E_{J-1},I_k,\dots,I_{k+J-1}\right)
		\Pr_k\left( \cap_{0 \le j \le J-1} \bar{I}_{k+j} \middle|\ E_0,\dots,E_{J-1}\right)}
		{\Pr_k\left(E_J \middle| E_0,\dots,E_{J-1} \right)} \\
		&&\quad=  
		\frac{\Pr_k\left(\bar{I}_J \middle| E_0,\dots, E_{J-1},I_k,\dots,I_{k+J-1}\right) 
		\Pr_k\left( \cap_{0 \le j \le J-1} \bar{I}_{k+j} \middle|\ E_0,\dots,E_{J-1}\right)\Pr_k\left(E_J \middle| E_0,\dots, E_{J-1},I_k,\dots,I_{k+J-1}\right)}
		{\Pr_k\left(E_J \middle| E_0,\dots,E_{J-1} \right)} \\
		&&\quad \le \Pr_k\left(\bar{I}_J \middle| E_0,\dots, E_{J-1},I_k,\dots,I_{k+J-1}\right) 
		\Pr_k\left( \cap_{0 \le j \le J-1} \bar{I}_{k+j} \middle|\ E_0,\dots,E_{J-1}\right),
	\end{eqnarray*}
	}
	where the last equality comes from~\eqref{eq:condindeppropstop}, and the final inequality uses the 
	fact that the events $E_0,\dots,E_{J-1}$ and $I_k,\dots,I_{k+J-1}$ are pairwise independent, thus 
	\[
		\Pr_k\left(E_J \middle| E_0,\dots, E_{J-1},I_k,\dots,I_{k+J-1}\right) \; = \; 
		\Pr_k\left(E_J \middle| E_0,\dots,E_{J-1} \right).
	\]
	Using~\eqref{eq:condprobapropstop}, we then have that 
	\[
		\Pr_k\left(\bar{I}_J \middle| E_0,\dots, E_{J-1},I_k,\dots,I_{k+J-1}\right) = 
		1-\Pr_k\left(I_J \middle| E_0,\dots, E_{J-1},I_k,\dots,I_{k+J-1}\right) \le 
		1-p.
	\]
	Thus,
	\begin{equation} \label{eq:recursivepropstop}
		\Pr_k\left( \bigcap_{0 \le j \le J} \bar{I}_{k+j} \middle| E_0,\dots,E_J \right) 
		\le (1-p)\Pr_k\left( \bigcap_{0 \le j \le J-1} \bar{I}_{k+j} \middle| E_0,\dots,E_{J-1} \right).
	\end{equation}
	By a recursive argument on the right-hand side of~\eqref{eq:recursivepropstop}, we thus arrive 
	at~\eqref{eq:goalpropstop}, which yields the desired conclusion.

\subsection{Proof of Proposition~\ref{prop:complexityJ}}

	As in Theorem~\ref{th:complexity0}, $\Tj^m$ clearly is a stopping time. 
	Moreover, if $\skn_k=1$ for all $k$, then $\Tj^m=T_{\epsilon}$ for every $J$, where 
	$T_{\epsilon}$ is the stopping time defined in Theorem~\ref{th:complexity0}, and therefore the 
	result holds. In what follows, we thus focus on the remaining case.

	Consider an iterate $k$ such that $x_k$ is $(\hat{\epsilon},\hat{\epsilon}^{1/2})$-function 
	stationary and the model $\hat f_k$ is accurate. From the definition of $\hat{\epsilon}$, such an iterate 	
	is also $((1-\kappa_g)\epsilon,(1-\kappa_H)\epsilon^{1/2})$-function stationary and the model $\hat f_k$ is 
	accurate. Then, by a reasoning similar to that of the proof of 
	Lemma~\ref{lemma:modeltruestationarity}, we can show that $x_k$ is $(\epsilon,\epsilon^{1/2})$-model 
	stationary.
	As a result, if $\Tj^m > k$, one of the models $\hat f_k,\hat f_{k+1},\dots,\hat f_{k+J}$ must be inaccurate, which 
	happens with probability $1-p^{J+1}$.

	Let $\Tj$ be the first iteration index for which the iterate is a 
	$(\hat{\epsilon},\hat{\epsilon}^{1/2})$  function stationary point and 
	satisfies~\eqref{eq:Jp1modelstationarity}\revised{, i. e.
	\begin{equation*}
		\min\{\|g_k\|,\|g^+_k\|\} < \epsilon  \quad \mbox{and} \quad \lambda_k > -\epsilon^{1/2}, \quad
		\forall k \in \{\Tj^m,\Tj^m+1,...,\Tj^m+J\}.
	\end{equation*}	
	}
	Clearly $\Tj^m \le \Tj$ (for all realizations of these two 
	stopping times), and it thus suffices to bound $\Tj$ in expectation. By applying 
	Theorem~\ref{th:complexity0} (with $\epsilon$ in the theorem's statement replaced by 
	$\hat{\epsilon}$), one can see that there must exist an infinite number of 
	$(\hat{\epsilon},\hat{\epsilon}^{1/2})$-function stationary points in expectation. More precisely, 
	letting $\{T_{\hat{\epsilon}}^{(i)}\}_{i=1,\dots}$ be the corresponding stopping times indicating the 
	iteration indexes of these points and using Theorem~\ref{th:complexity0},
	we have
	\begin{eqnarray*}
		\E{T_{\hat{\epsilon}}^{(1)}} = \E{T_{\hat{\epsilon}}} &\le 
		&\frac{\left(f(x_0)- f_{\low}\right)}{p \hat{c}} 
		\,\hat{\epsilon}^{-3/2}+1, \\
		\forall i\ge 1,\quad \E{T_{\hat{\epsilon}}^{(i+1)}-T_{\hat{\epsilon}}^{(i)}} 
		&\le &\frac{\left(f(x_0)- f_{\low}\right)}{p\hat{c}} 
		\,\hat{\epsilon}^{-3/2}+1.
	\end{eqnarray*}
	Consider now the subsequence $\{T_{\hat{\epsilon}}^{(i_{\ell})}\}_{\ell=1,\dots}$ such that all 
	stopping times are at least $J$ iterations from each other, i.e., for every $\ell\ge 1$, we have 
	$T_{\hat{\epsilon}}^{(i_{\ell+1})}-T_{\hat{\epsilon}}^{(i_{\ell})}\ge J$. 
	For such a sequence, we get
	\begin{eqnarray*}
		\forall \ell \ge 1,\quad \E{T_{\hat{\epsilon}}^{(i_{\ell+1})}-T_{\hat{\epsilon}}^{(i_{\ell})}} 
		&\le &\frac{\left(f(x_0)- f_{\low}\right)}{p\hat{c}}\,	
		\hat{\epsilon}^{-3/2}+J+1 \triangleq K(\epsilon,J) \\
		\E{T_{\hat{\epsilon}}^{(i_1)}} = \E{T_{\hat{\epsilon}}} 
		&\le &\frac{\left(f(x_0)- f_{\low}\right)}{p\hat{c}} 
		\,\hat{\epsilon}^{-3/2}+1 \le  K(\epsilon,J).
	\end{eqnarray*}

	For every $\ell \ge 1$, we define the event 
	\[
		B_{\ell} \; = \; \bigcap_{j=0}^{J}\, I_{T_{\hat{\epsilon}}^{(i_{\ell})}+j} \; = \; 
		\bigcap_{j=0}^{J} \left\{ m_{T_{\hat{\epsilon}}^{(i_{\ell})}+j} \mbox{\ is\ accurate}\right\}.
	\]
	\revised{ By Assumption~\ref{as:modifalgo}, the samples are generated independently of the current iterate, and for every 
	$k$, $\Pr(I_k|\mathcal{F}_{k-1}) = p$. By the same recursive reasoning as in the proof of 
	Proposition~\ref{prop:stop}, we have that 
	$\Pr\left(B_{\ell}|\mathcal{F}_{T_{\hat{\epsilon}}^{(i_{\ell})}-1}\right) = p^{J+1}$.
	Moreover, by definition of the sequence $\{T_{\hat{\epsilon}}^{(i_{\ell})}\}$, two stopping times in 
	that sequence correspond to two iteration indexes distant of at least $J+1$. Therefore, they also 
	correspond to two separate sequences of $(J+1)$ models that are generated in an independent fashion. 
	We can thus consider $\{B_{\ell}\}$ to be an independent sequence of Bernoulli trials.
	Therefore, the variable $G$ representing the number of runs of $B_{\ell}$ until success follows a 
	geometric distribution with an expectation less than $\frac{1}{p^{J+1}}=p^{-(J+1)}<\infty$.} On the 
	other hand, $\Tj$ is less than the first element of $\{T_{\hat{\epsilon}}^{(i_{\ell})}\}$ for which 
	$B_{\ell}$ happens, and thus $\Tj \leq T_{\hat{\epsilon}}^{(i_G)}$. 
	To conclude the proof, we define 
	\[
		S_G = T_{\hat{\epsilon}}^{(i_G)}, \qquad 
		X_1 = T_{\hat{\epsilon}}^{(i_{1})} = T_{\hat{\epsilon}}^{1}, \quad
		X_{\ell} = T_{\hat{\epsilon}}^{(i_{\ell})} - T_{\hat{\epsilon}}^{(i_{\ell-1})} \quad 
		\forall \ell \ge 2.
	\]
	From the proof of Wald's equation~\cite[Theorem 4.1.5]{durrett2010probatheory} (more precisely, from 
	the third equation appearing in that proof), one has 
	\[
		\E{S_G} = \sum^{\infty}_{\ell=1} \E{X_{\ell}} \Pr({G\ge \ell}).
	\]
	Since $\E{X_\ell} \le K(\epsilon,J)$, one arrives at
	\[
		\E{\Tj^m} \le \E{\Tj}  \le \E{T_{\hat{\epsilon}}^{(i_G)}} \le 
		K(\epsilon,J) \sum^{\infty}_{\ell=1} \Pr({G\ge \ell}) = K(\epsilon,J)  \E{G},
	\] 
	which is the desired result.

\section{More details on numerical results}

\subsection{ALAS algorithm as implemented}
Our implementation of the ALAS algorithm is described in Algorithm~\ref{alg:alasi}. The main differences between Algorithm~\ref{alg:alasi} and Algorithm~\ref{alg:alas} are described in the main paper.

\begin{algorithm}
	\SetAlgoLined
	\setcounter{algocf}{1}
	\DontPrintSemicolon 
	\BlankLine
	\textbf{Initialization}: Choose $x_0 \in \real^n$, $\theta \in (0,1), \eta > 0$, $\epsilon>0$.\;
	\For{$k=0,1,...$}{
		\begin{enumerate}
			\item Draw a random sample set $\mathcal{S}_k \subset \{1,\dots,N\}$, and compute the 
			associated quantities $g_k:=g(x_k;\mathcal{S}_k), H_k:=H(x_k;\mathcal{S}_k)$. Form the model:
			\begin{equation} 
				\hat f_k(x_k+s) := \hat f(x_k+s;\mathcal{S}_k).
			\end{equation}

                        \item Compute $R_k = \tfrac{g_k^T H_k g_k}{\|g_k\|^2}$.
                        \item If $R_k<-\|g_k\|^{1/2}$ then set $d_k=\tfrac{R_k}{\|g_k\|}g_k$ and go to the Line-search step.
                        \item Else if $R_k<\|g_k\|^{1/2}$ and $\|g_k\|\ge \epsilon$ then set $d_k=-\frac{g_k}{\|g_k\|^{1/2}}$ and go to the Line-search step.
			\item Compute $\lambda_k$ as the minimum eigenvalue of the Hessian estimate $H_k$.\\ 
			If $\lambda_k \ge -\ecurv$ and $\|g_k\|=0$ set $\alpha_k=0,\ d_k=0$ and go to Step 10.
			\item If $\lambda_k < -\|g_k\|^{1/2}$, 
                              Compute a negative eigenvector $v_k$ such that
			\begin{equation} 
				H_k v_k = \lambda_k v_k,\ \|v_k\| = -\lambda_k,\ v_k^\top g_k \le 0,
			\end{equation}
			set $d_k=v_k$ and go to the line-search step.
             \item If $\lambda_k>\|g_k\|^{1/2}$, compute a Newton direction $d_k$ solution of
			\begin{equation} 
				H_k d_k = -g_k,
			\end{equation}
			go to the line-search step.
			\item If $d_k$ has not yet been chosen, compute it as a regularized Newton direction, solution of
			\begin{equation} 
				\left(H_k+(\|g_k\|^{1/2}+\ecurv) \mathbbm{I}_n\right)d_k = -g_k,
			\end{equation}
			and go to the line-search step.
			\item \textbf{Line-search step} Compute the minimum index $j_k$ such that the 
			step length\\ $\alpha_k:=\theta^{j_k}$ satisfies the decrease condition:
			\begin{equation} \label{eq:suff:cond2}
				\hat f_k(x_k+\alpha_k d_k) - \hat f_k(x_k) \; \le \; -\frac{\eta}{2}\alpha_k^2\|d_k\|^2.
			\end{equation}

			\item Set $x_{k+1}=x_k+\alpha_k d_k$.
			\item Set $k=k+1$.
		\end{enumerate}
	} 
\caption{ALAS, as implemented. \label{alg:alasi}}
\end{algorithm}

\subsection{Distribution of Steps}
\paragraph{IJCNN Dataset}

\revised{We have included most of the visual information for the runs on the IJCNN in the main part of the paper.
One additional consideration to verify the stability of the algorithm is to perform 
a sensitivity analysis of the hyperparameters of ALAS.
There are two hyperparameters in the ALAS algorithm, $\eta$ and $\epsilon$. In deterministic contexts, $\eta$ is typically
kept small to allow for fairly liberal step acceptance. We studied the impact of varying this hyperparameter on the 
performance of ALAS, reporting our results in Figure~\ref{fig:sens_IJCNN}. We confirm that indeed for reasonably
small values the precise value has limited impact on the convergence. This is by contrast with SGD wherein the learning
rate has a very significant impact on the performance and typically has to be tuned for each problem.
\begin{figure}
    \centering
    \subfigure{
    \includegraphics[width=7.5cm, height=4.1cm]{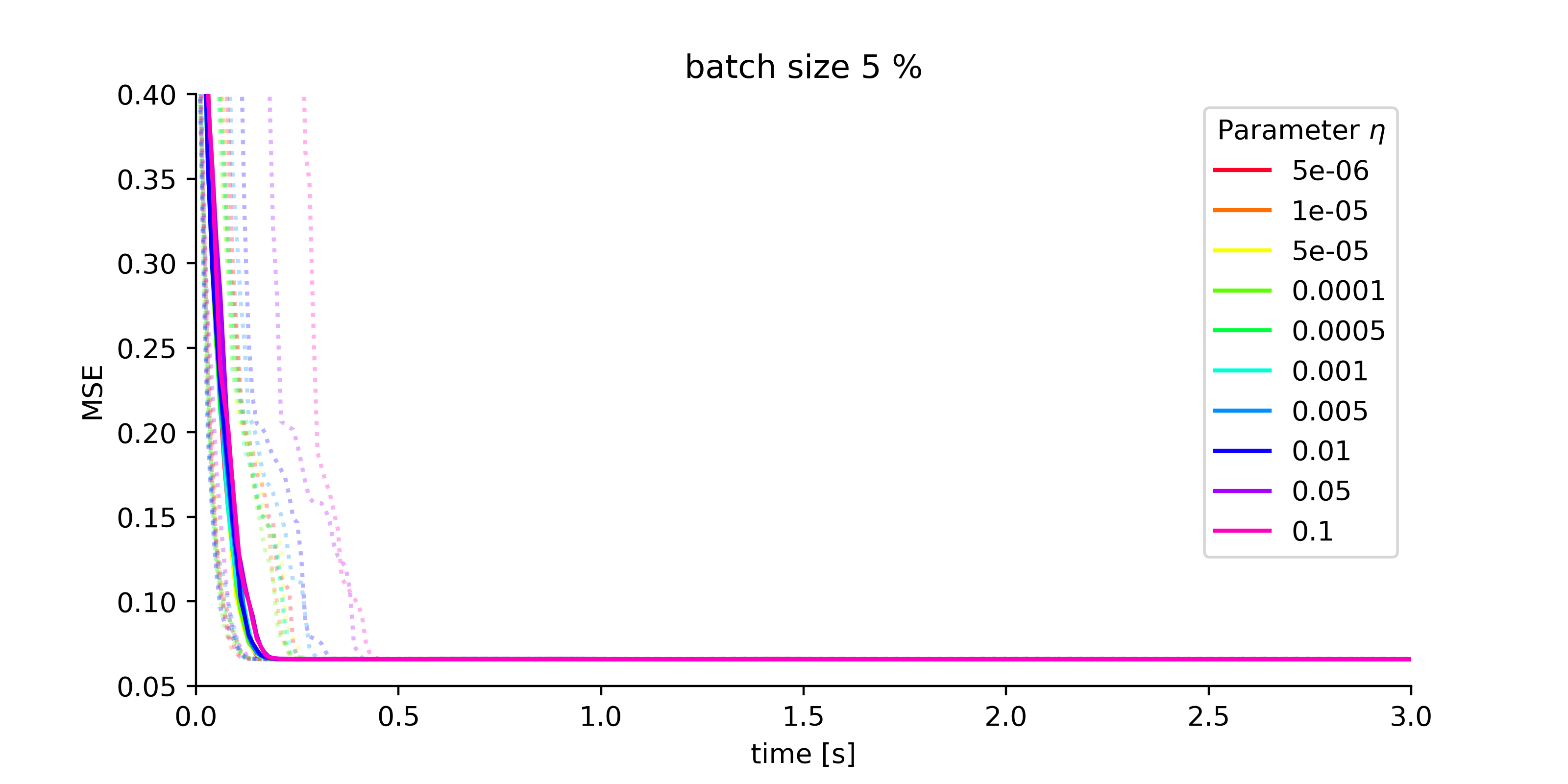}\label{fig:sens_IJCNN_nn_1}
    }
    ~ 
    \subfigure{
    \includegraphics[width=7.5cm, height=4.1cm]{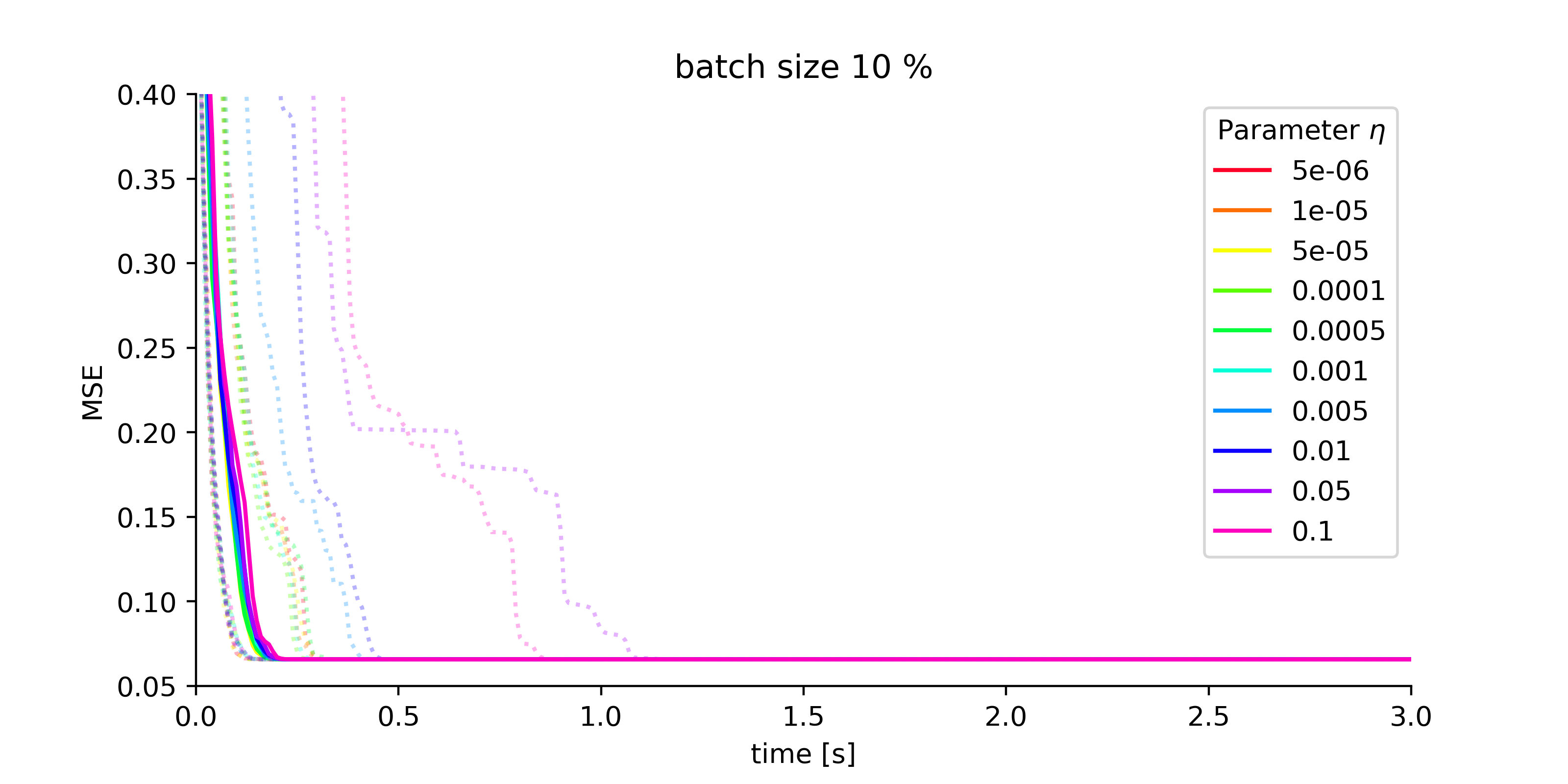}\label{fig:sens_IJCNN_nn_2}
    }
     \\
    \subfigure{
    \includegraphics[width=7.5cm, height=4.1cm]{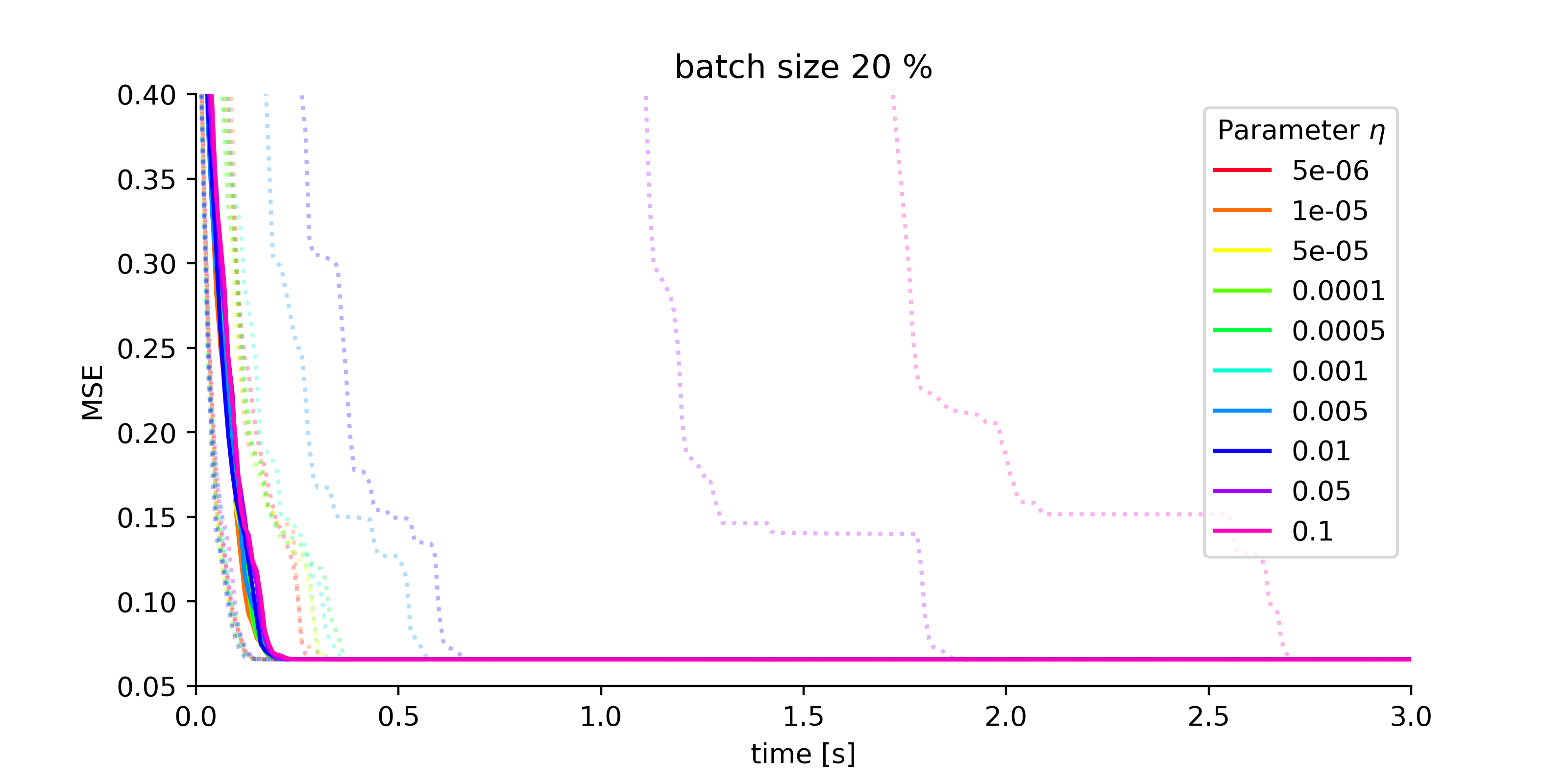}\label{fig:sens_IJCNN_nn_3}
    }
    ~ 
    \subfigure{
    \includegraphics[width=7.5cm, height=4.1cm]{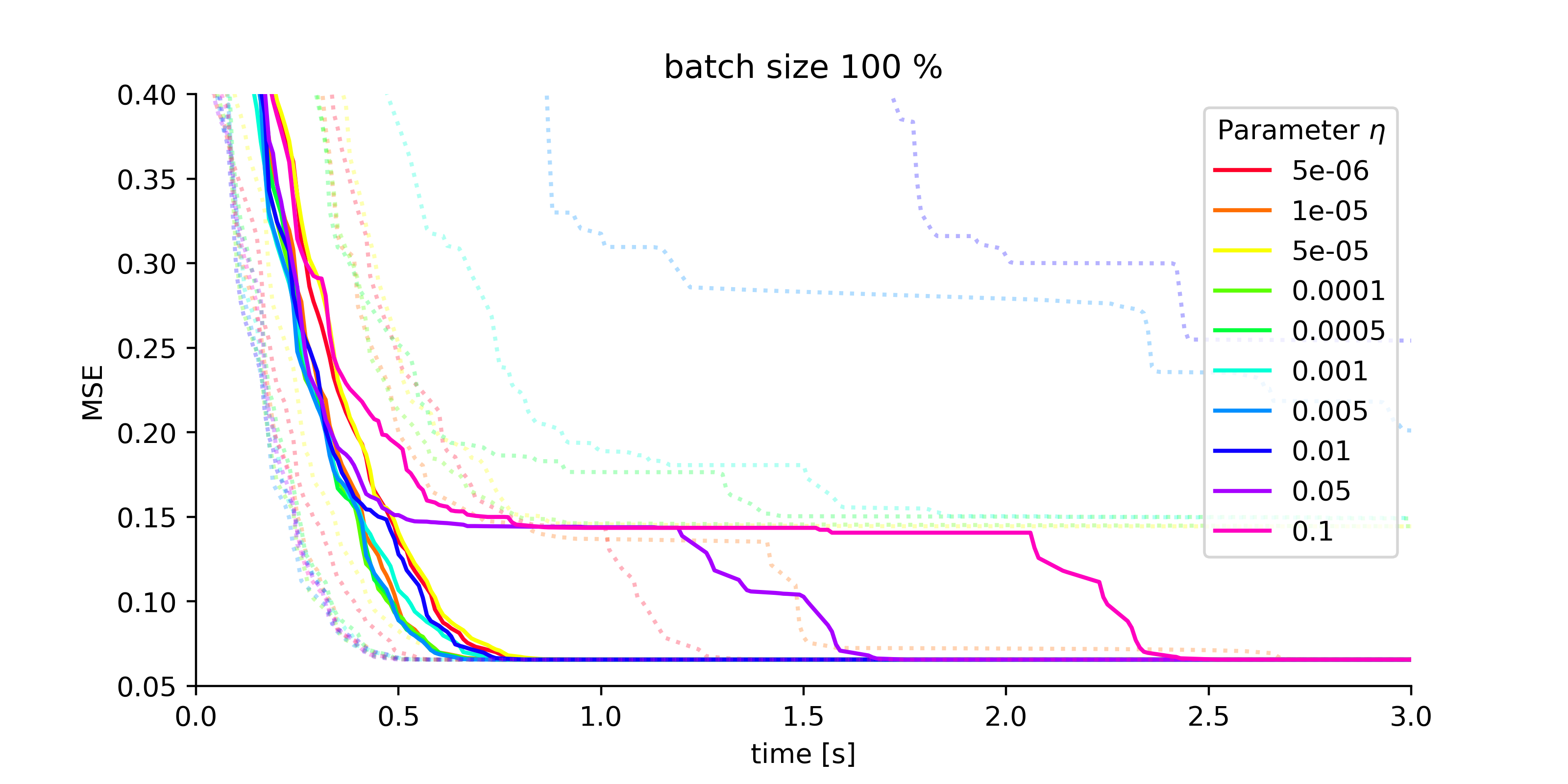}\label{fig:sens_IJCNN_nn_4}
    }
    \caption{The sensitivity of the performance of the ALAS algorithm on the IJCNN1 task as depending on the choice of $\eta$.}
    \label{fig:sens_IJCNN}
\end{figure}
The sensitivity of the Algorithm with respect to $\epsilon$ is shown in Figure~\ref{fig:sens_IJCNN_eps}. It can be seen that here, as well,
the Algorithm is not particularly sensitive. Although the Regularized Newton step, which depends on $\epsilon$ explicitly, is chosen
the most, it appears that it is modified in an appropriate space so as to mitigate any degradation of performance. 
\begin{figure}
    \centering
    \subfigure{
    \includegraphics[width=7.5cm, height=4.1cm]{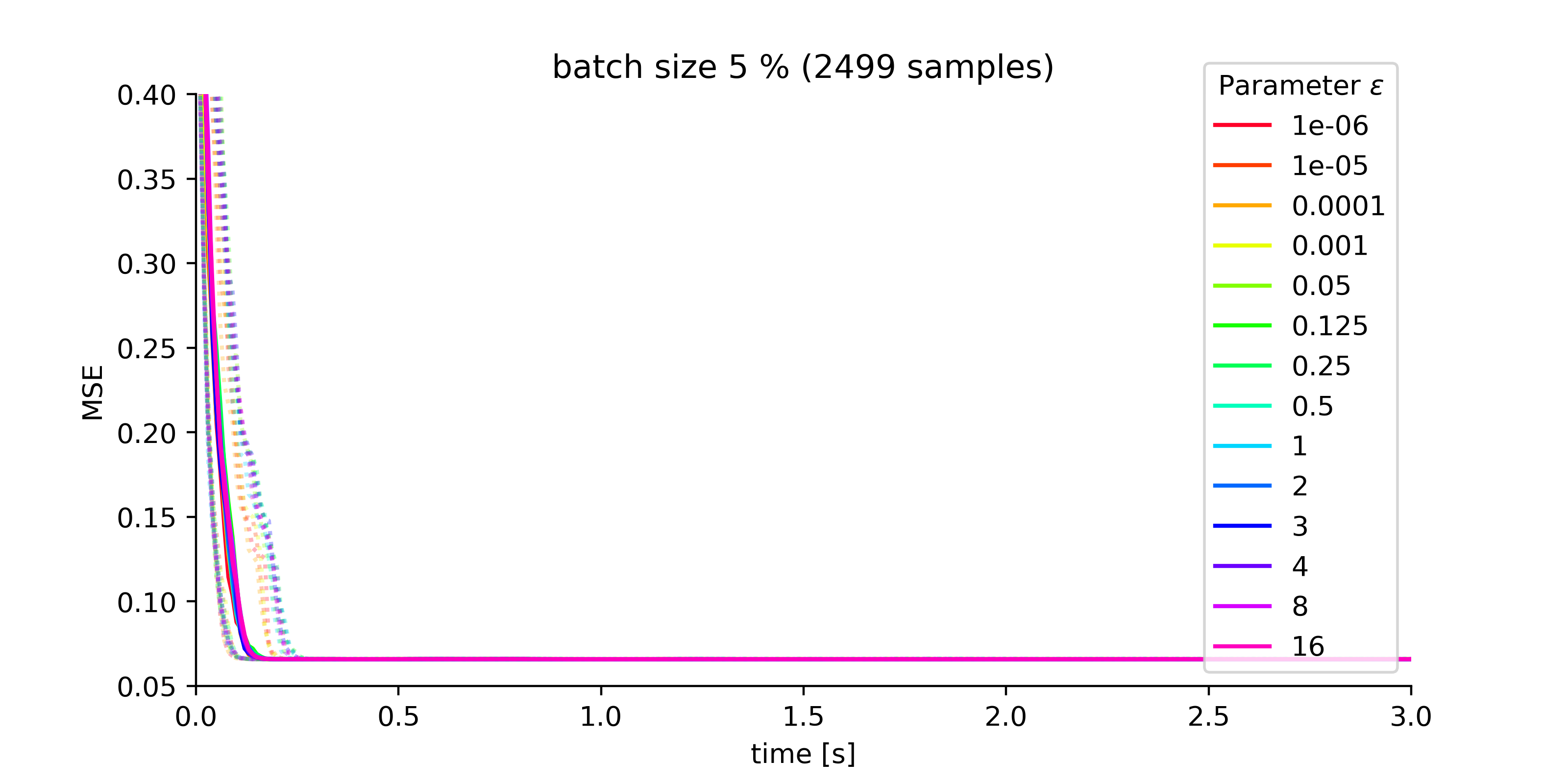}\label{fig:sens_IJCNN_nn_epsilon_1}
    }
    ~ 
    \subfigure{
    \includegraphics[width=7.5cm, height=4.1cm]{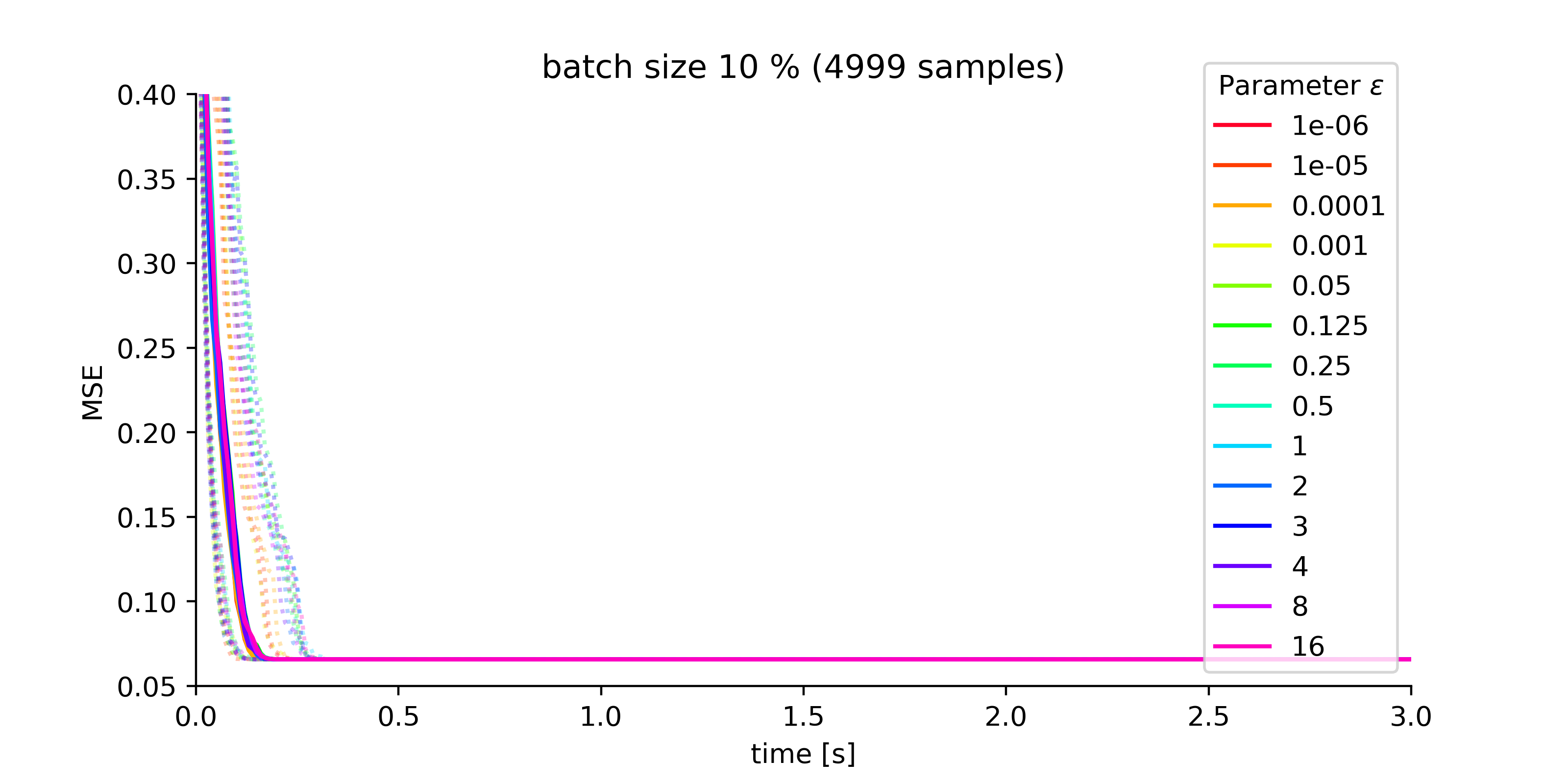}\label{fig:sens_IJCNN_nn_epsilon_2}
    }
     \\
    \subfigure{
    \includegraphics[width=7.5cm, height=4.1cm]{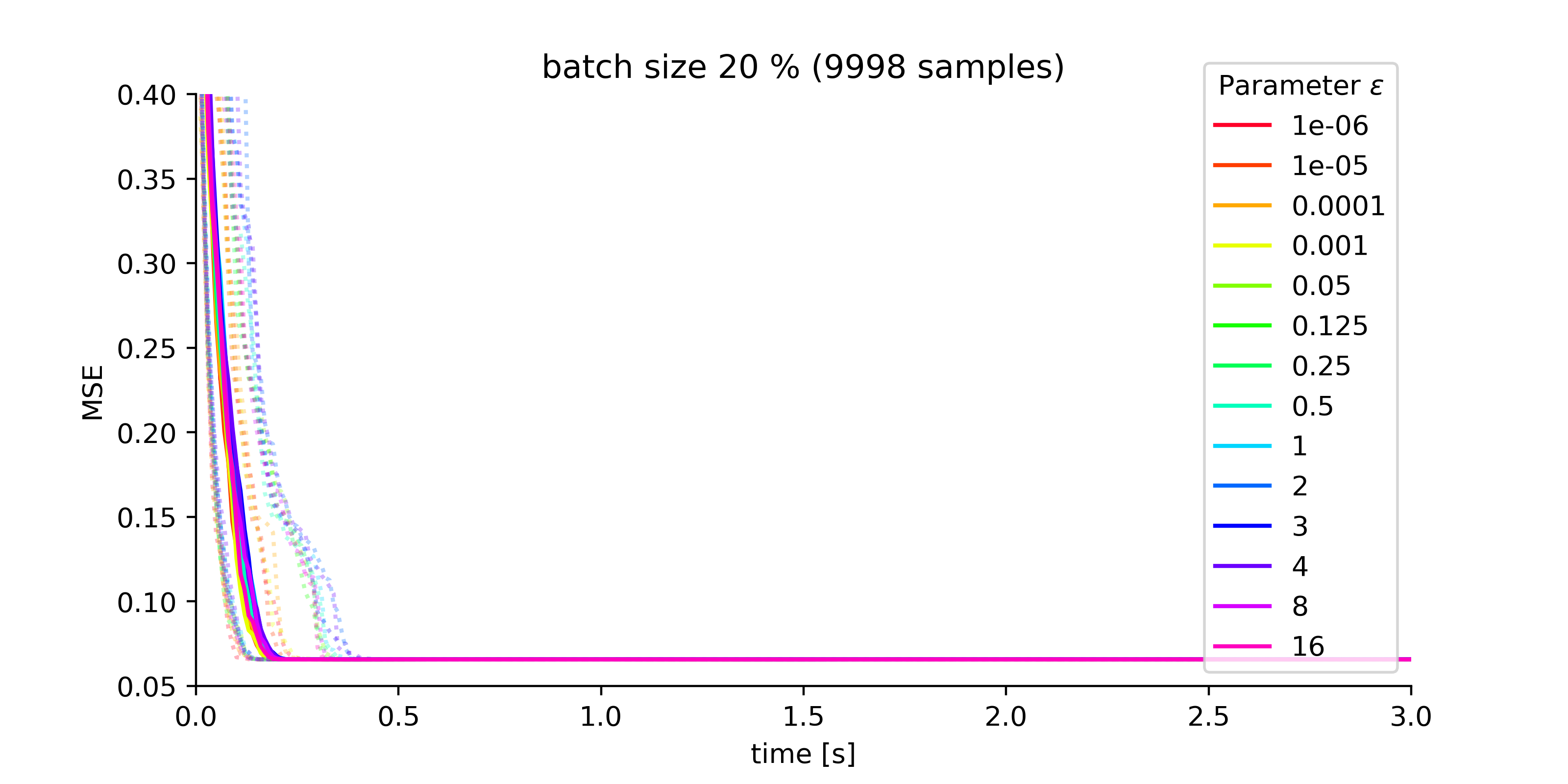}\label{fig:sens_IJCNN_nn_epsilon_3}
    }
    ~ 
    \subfigure{
    \includegraphics[width=7.5cm, height=4.1cm]{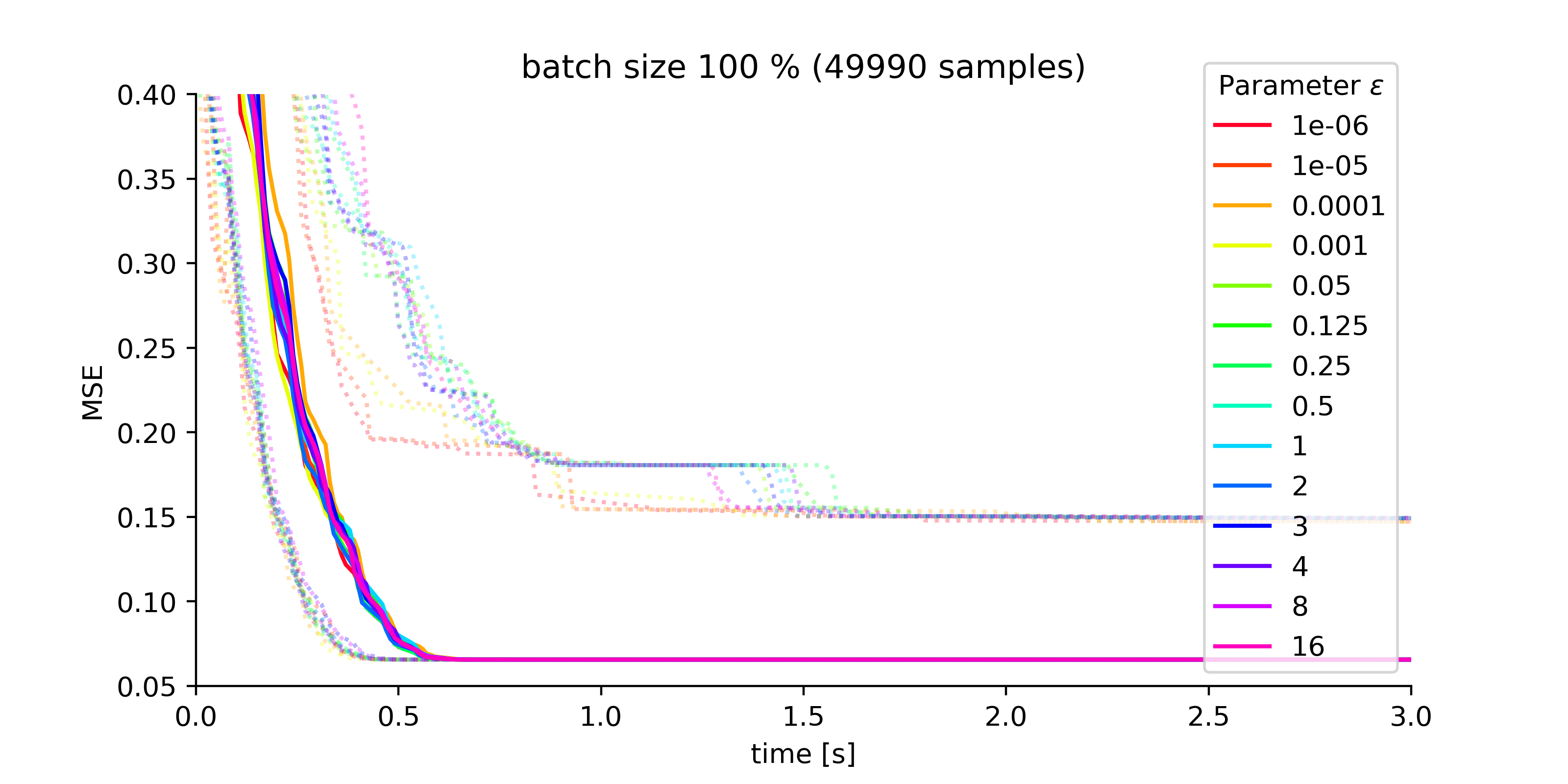}\label{fig:sens_IJCNN_nn_epsilon_4}
    }
    \caption{The sensitivity of the performance of the ALAS algorithm on the IJCNN1 task as depending on the choice of $\epsilon$.}
    \label{fig:sens_IJCNN_eps}
\end{figure}
}

\paragraph{A9A Dataset}

The distribution of the individual type of steps as described in Algorithm~\ref{alg:alasi} for selected runs for the A9A task is shown in Figure~\ref{fig:steps_A9A}. As in the case of IJCNN1 task, the most frequently used step type is the \textit{Regularized Newton}. We note that when a full sample is used, \textit{Negative curvature} steps are more common, suggesting that the problem is quite nonconvex. Note that in all of our runs (including those not reported here), other step choices were used less than 10 times.
The distributions of number of line-search iterations for selected runs of the ALAS algorithm are shown in Figure~\ref{fig:line_search_hist_A9A}.

\begin{figure}
    \centering
    \subfigure[123-1 5\%]{
    \includegraphics[width=7.5cm, height=4.1cm]{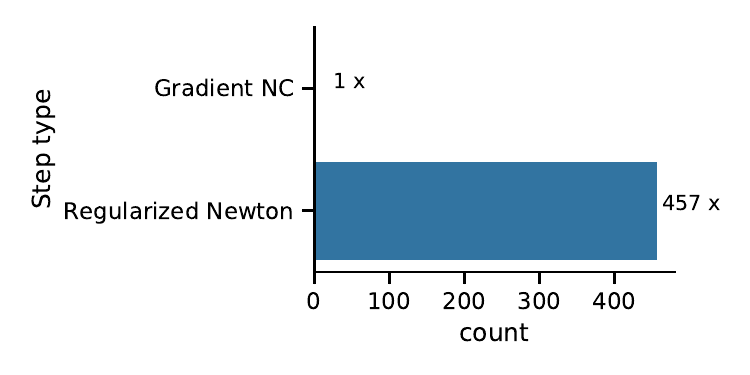}\label{fig:steps_A9A_nn_123-1_5}
    }
    ~ 
    \subfigure[123-1 100\%]{
    \includegraphics[width=7.5cm, height=4.1cm]{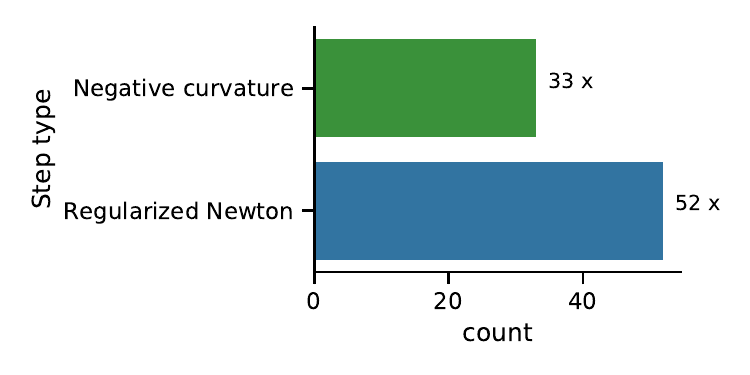}\label{fig:steps_A9A_nn_123-1_100}
    }
    \caption{The step type distribution of a single run of  ALAS algorithm on the A9A task for the 123-1 architecture and two sampling sizes.}
    \label{fig:steps_A9A}
\end{figure}

\begin{figure}
    \centering
    \subfigure[123-1 5\%]{
    \includegraphics[width=5cm, height=3.5cm]{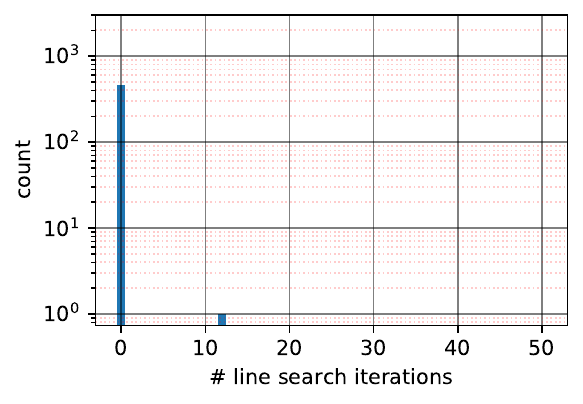}\label{fig:line_search_hist_A9A_nn_123-1_5}
    }
    \subfigure[123-1 100\%]{
    \includegraphics[width=5cm, height=3.5cm]{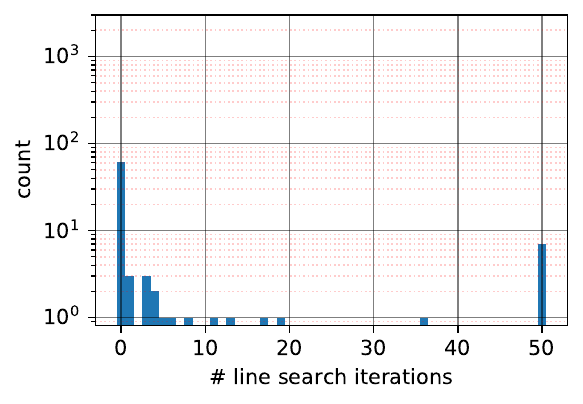}\label{fig:line_search_hist_A9A_nn_123-1_100}
    }
    \subfigure[123-2-1 100\%]{
    \includegraphics[width=5cm, height=3.5cm]{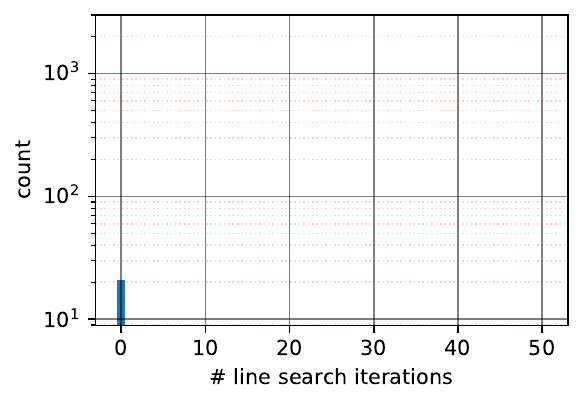}\label{fig:line_search_hist_A9A_nn_123-2-1_100}
    }
    \caption{Plots of the number of line-search iterations during each update for the A9A task for selected runs of the ALAS algorithm with different architectures and sampling sizes. The maximum number of line-search iterations was set to 50.}
    \label{fig:line_search_hist_A9A}
\end{figure}

\paragraph{MNIST Transfer Learning}

The distribution of the individual type 
 of steps as described in Algorithm~\ref{alg:alasi} for selected runs for the transfer learning task is shown in Figure~\ref{fig:steps_TransferMNIST01}. Again, the most frequently used step type was the \textit{Regularized Newton} 
  and the others occurred very rarely. The distribution of the number of line-search iterations for selected runs of the ALAS algorithm are shown in Figure~\ref{fig:line_search_hist_TransferMNIST01}; the line search was usually rather short (0 iterations of were dominating for most runs) but it was still used quite often.

\begin{figure}
    \centering
    \subfigure[8-2-1 5\%]{
    \includegraphics[width=7.5cm, height=4.1cm]{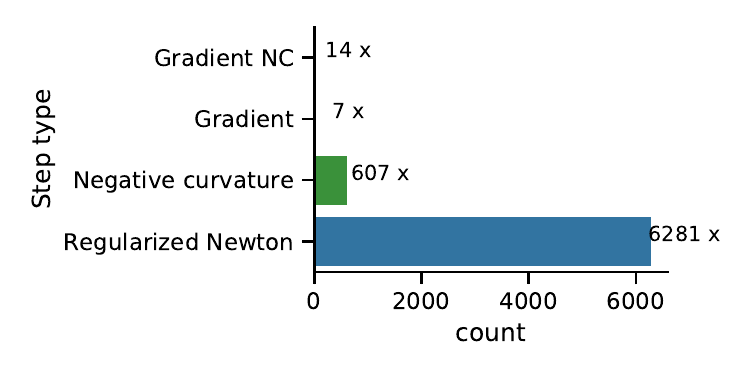}\label{fig:steps_TransferMNIST01_nn_8-2-1}
    }
    ~ 
    \subfigure[8-4-1 100\%]{
    \includegraphics[width=7.5cm, height=4.1cm]{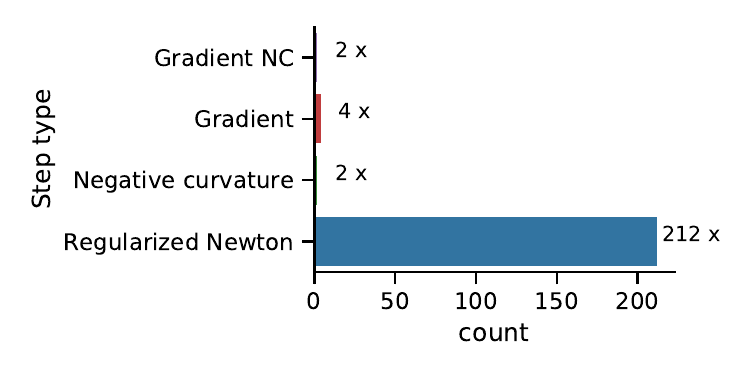}\label{fig:steps_TransferMNIST01_nn_8-4-1_100}
    }
    \caption{The step type distribution of a single run of ALAS algorithm on the transfer learning task.}
    \label{fig:steps_TransferMNIST01}
\end{figure}

\begin{figure}
    \centering
    \subfigure[8-1 100\%]{
    \includegraphics[width=5cm, height=3.5cm]{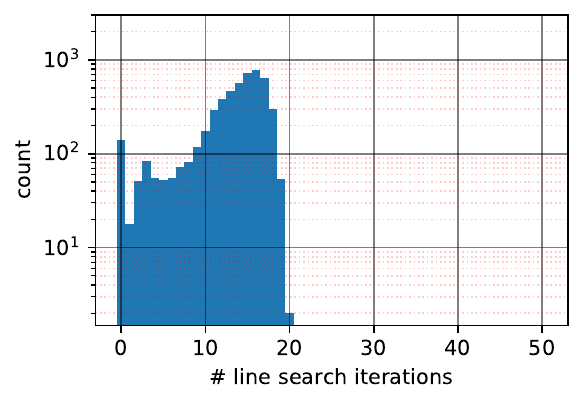}\label{fig:line_search_hist_TransferMNIST01_nn_8-1_100}
    }
    \subfigure[8-1-1 5\%]{
    \includegraphics[width=5cm, height=3.5cm]{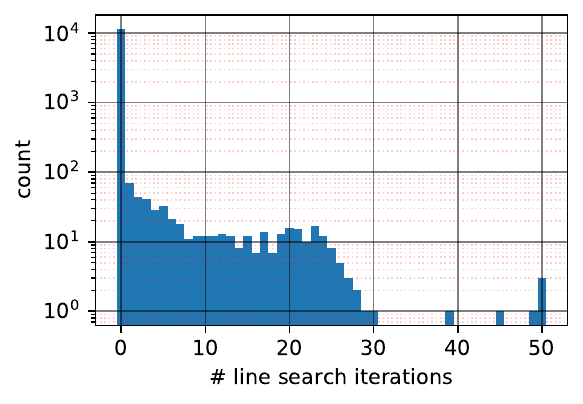}\label{fig:/line_search_hist_TransferMNIST01_nn_8-1-1_5}
    }
    \subfigure[8-1-1 10\%]{
    \includegraphics[width=5cm, height=3.5cm]{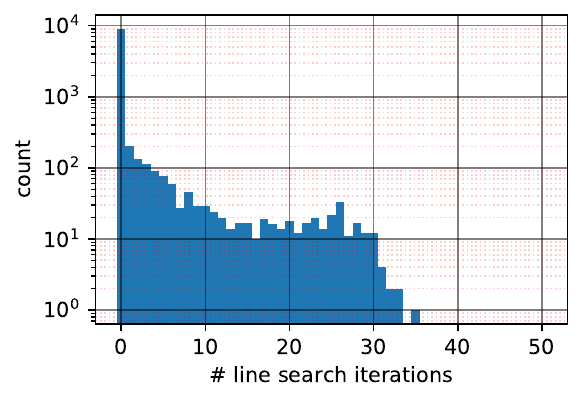}\label{fig:line_search_hist_TransferMNIST01_nn_8-1-1_10}
    }
    
    \subfigure[8-1-1 20\%]{
    \includegraphics[width=5cm, height=3.5cm]{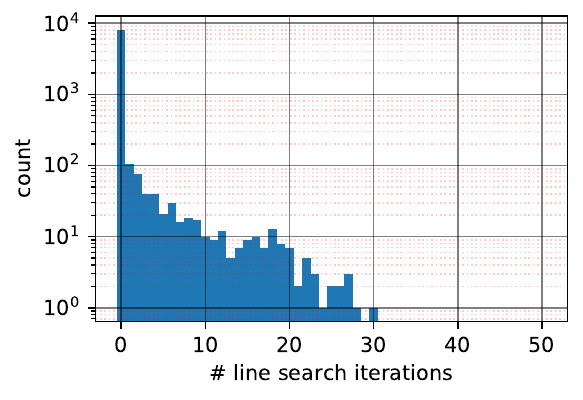}\label{fig:line_search_hist_TransferMNIST01_nn_8-1-1_20}
    }
    \subfigure[8-1-1 100\%]{
    \includegraphics[width=5cm, height=3.5cm]{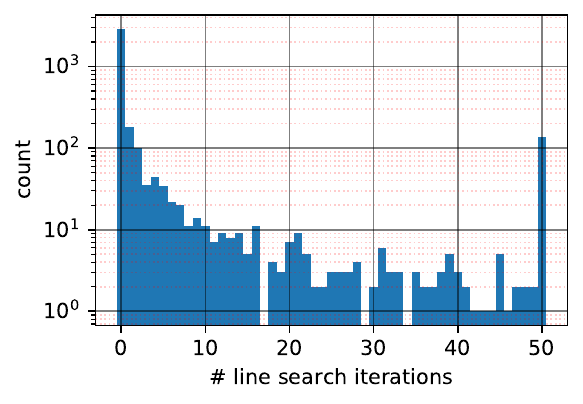}\label{fig:/line_search_hist_TransferMNIST01_nn_8-1-1_100}
    }
    \subfigure[8-2-1 100\%]{
    \includegraphics[width=5cm, height=3.5cm]{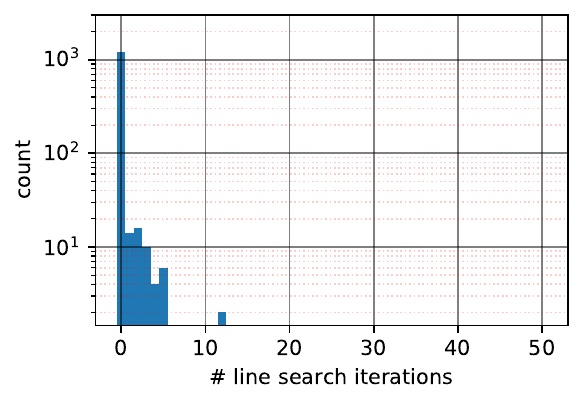}\label{fig:line_search_hist_TransferMNIST01_nn_8-2-1_100}
    }
    \caption{Plots of the number of line-search iterations during each update for the transfer learning task for selected runs of the ALAS algorithm with different architectures and sampling sizes. The maximum number of line-search iterations was set to 50.}
    \label{fig:line_search_hist_TransferMNIST01}
\end{figure}

\paragraph{Artificial NN}

The step type distribution for different runs is shown in Figure~\ref{fig:steps_NN1} --- the most frequent steps in general for this task were the \textit{Regularized Newton} (strongly dominating) and the \textit{Negative curvature} step.

\begin{figure}
    \centering
    \subfigure[2-1-1 20\%]{
    \includegraphics[width=7.5cm, height=4.1cm]{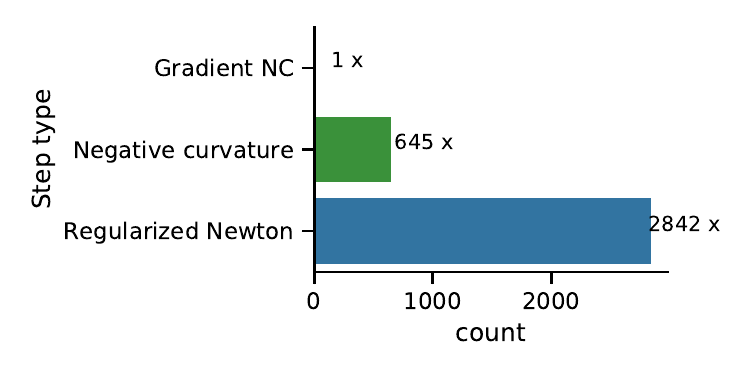}\label{fig:steps_NN1_nn_2-1-1_20}
    }
    ~ 
    \subfigure[2-4-2-1 5\%]{
    \includegraphics[width=7.5cm, height=4.1cm]{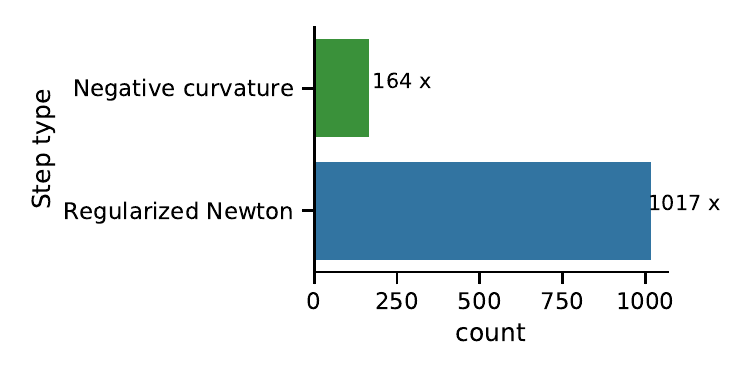}\label{fig:steps_NN1_nn_2-4-2-1_5}
    }
    \caption{The step type distribution of a single run of  ALAS algorithm on the NN1 task.}
    \label{fig:steps_NN1}
\end{figure}

\subsection{Additional plots for IJCNN1}

We tested SGD with several possible values for the learning rate, namely {1, 0.6, 0.3, 0.1, 0.01, 0.001}. Figure~\ref{fig:IJCNN_nn_22-1} shows the performance of a subset of those values: as one may expect, SGD is quite sensitive to the choice of the step size.

\begin{figure}
    \centering
    \subfigure[22-1 5\%]{
        \includegraphics[width=7.5cm, height=4.1cm]{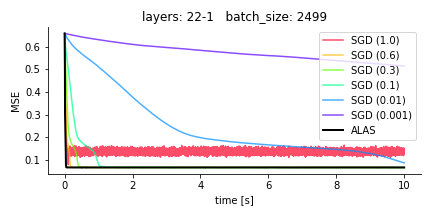}\label{fig:IJCNN_nn_22-1_0.05}
    }
    ~ 
    \subfigure[22-1 5\% (zoom)]{
        \includegraphics[width=7.5cm, height=4.1cm]{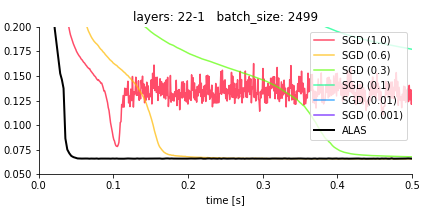}\label{fig:IJCNN_nn_22-1_0.05_zoom}
    }
    
    \subfigure[22-1 10\%]{
        \includegraphics[width=7.5cm, height=4.1cm]{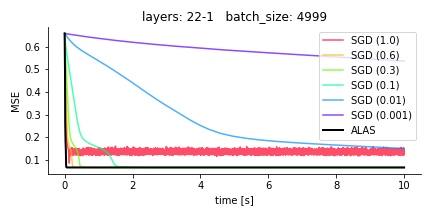}\label{fig:IJCNN_nn_22-1_0.10}
    }
    ~ 
    \subfigure[22-1 10\% (zoom)]{
        \includegraphics[width=7.5cm, height=4.1cm]{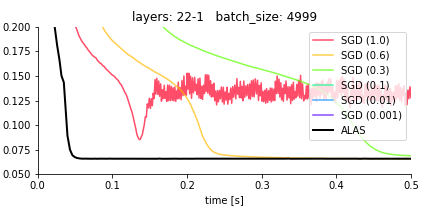}\label{fig:IJCNN_nn_22-1_0.10_zoom}
    }
    
    \subfigure[22-1 20\%]{
        \includegraphics[width=7.5cm, height=4.1cm]{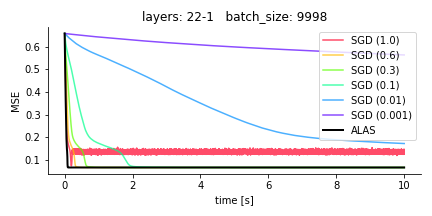}\label{fig:IJCNN_nn_22-1_0.20}
    }
    ~ 
    \subfigure[22-1 20\% (zoom)]{
        \includegraphics[width=7.5cm, height=4.1cm]{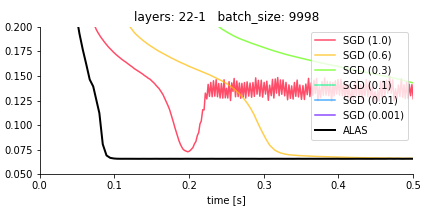}\label{fig:IJCNN_nn_22-1_0.20_zoom}
    }
    
    \subfigure[22-1 100\%]{
        \includegraphics[width=7.5cm, height=4.1cm]{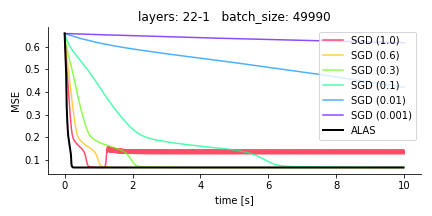}\label{fig:IJCNN_nn_22-1_1.00}
    }
    ~ 
    \subfigure[22-1 100\% (zoom)]{
        \includegraphics[width=7.5cm, height=4.1cm]{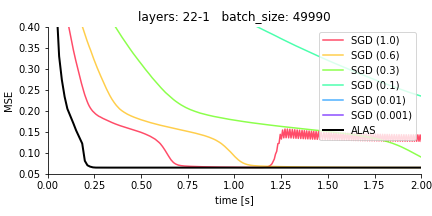}\label{fig:IJCNN_nn_22-1_1.00_zoom}
    }
    \caption{Comparison of  ALAS and SGD (with various learning rates) on the IJCNN1 dataset.}
    \label{fig:IJCNN_nn_22-1}
\end{figure}

\section{Additional numerical experiments}

\subsection{Experiment using the A9A dataset} 
\label{subsec:numA9A}
The second set of experiments was run on the A9A training dataset, also part of the LIBSVM library~\cite{libsvm}: for this classification dataset, we have $N=32,561$ samples, each of them possessing $n=123$ features.
For this task, we trained two neural networks without hidden layers  and with one hidden layer having a 
single neuron; the optimized function was the MSE as in the first experiment. Our batch sizes consisted of 100\%, 20\%, 10\% and 5\% of the dataset size, and 
the mean-squared error loss was used.  We tested four variants of SGD corresponding to the values 
{1, 0.6, 0.3, 0.1} for the learning rate, and selected the best variant for comparison with ALAS. The 
results are shown in Figures~\ref{fig:A9A_123-1} and \ref{fig:A9A_123-1-1}, as well as 
Table~\ref{tab:alg_comparison_a9a}.
Our method, ALAS, was significantly better for all sampling sizes on 
training the first network. However, for the second network with more optimization variables (learning parameters),  
ALAS appears to slow down relative to SGD. This appears due to our use of exact linear algebra techniques:  as explained in the introduction of this section, the results could be better for an inexact variant of 
ALAS.
\begin{table}[h!]
\centering
{\small
\begin{tabular}{c  c c c c | c c c c c }
\multicolumn{5}{c|}{Layers: 123-1} & \multicolumn{5}{c}{Layers: 123-2-1} \\ \hline
 alg. & $\skn_k$ & min loss & loss [8-10]s & iter. &  alg. & $\skn_k$ & min loss & loss [8-10]s  & iter.  \\ \hline
ALAS & 5\% & \textbf{0.1176} & \textbf{0.1182} & 458 & ALAS &5  \% & 0.1428 & 0.1437 & 20 \\
 SGD (0.1) & 5\% & 0.1224 & 0.1231 & 8260 &SGD (0.6) & 5 \% & \textbf{0.1066} & \textbf{0.1085} & 5992 \\
\hline
 ALAS & 10 \% & \textbf{0.1168} & \textbf{0.1176} & 360 & ALAS & 10  \% &  0.1464 & 0.1482 & 13 \\
 SGD (0.1) & 10 \% & 0.1245 & 0.1255 & 6168 &SGD (0.6) & 10 \% & \textbf{0.1094} &\textbf{0.1129}  &  3948\\
\hline
ALAS & 20 \% & \textbf{0.1163} & \textbf{0.1171} & 255& ALAS & 20  \% & 0.1436 & 0.1452 & 14 \\
 SGD (0.1) & 20 \% & 0.1396 & 0.1482 & 4152 &SGD (0.6) &  20\% &  \textbf{0.1163} & \textbf{0.1213} & 2449 \\
\hline
 ALAS & 100 \% &\textbf{ 0.1151} & \textbf{0.1151} & 85 & ALAS & 100 \% & 0.1550 & 0.1579 &  6\\
SGD (0.3) &100 \% & 0.2157 & 0.2786 & 645&SGD (0.6) & 100  \% & \textbf{0.1353} & \textbf{0.1365} & 819 \\
\hline
\end{tabular}
}
\caption{Results reached over the given time period $t =10 ~s$ on the A9A task. }
\label{tab:alg_comparison_a9a}
\end{table}

\begin{figure}
    \centering
    \subfigure[123-1 5\%]{
        \includegraphics[width=7.5cm, height=4.1cm]{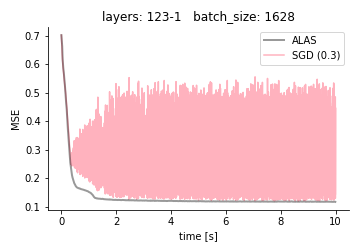}\label{fig:A9A_123-1_0.05}
    }
    ~ 
    \subfigure[123-1 10\%]{
        \includegraphics[width=7.5cm, height=4.1cm]{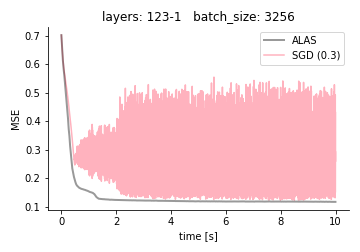}\label{fig:A9A_123-1_0.10}
    }
   \\
    \subfigure[123-1 20\%]{
        \includegraphics[width=7.5cm, height=4.1cm]{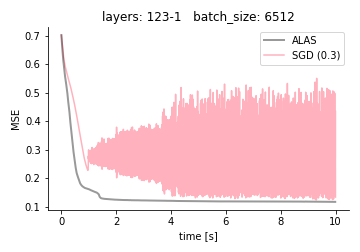}\label{fig:A9A_123-1_0.20}
    }
    ~ 
    \subfigure[123-1 100\%]{
        \includegraphics[width=7.5cm, height=4.1cm]{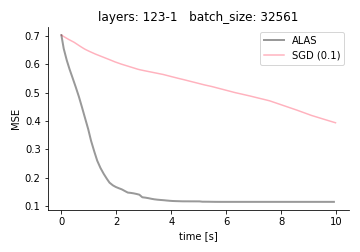}\label{fig:A9A_123-1_1.00}
    }
    \caption{Comparison of  ALAS and SGD (with best performing learning rate)  on the A9A dataset with a simple neural network with 123 input neurons, no hidden layer and an output neuron.}
    \label{fig:A9A_123-1}
\end{figure}

\begin{figure}
    \centering
    \subfigure[123-1-1 5\%]{
        \includegraphics[width=7.5cm, height=4.1cm]{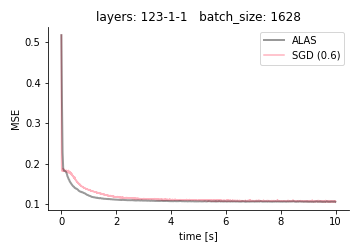}\label{fig:A9A_123-1-1_0.05}
    }
    ~ 
      \subfigure[123-1-1 10\%]{
        \includegraphics[width=7.5cm, height=4.1cm]{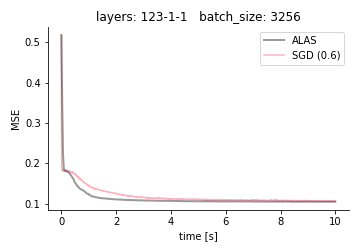}\label{fig:A9A_123-1-1_0.10}
    }\\
    \subfigure[123-1-1 20\%]{
        \includegraphics[width=7.5cm, height=4.1cm]{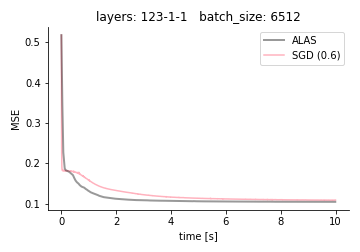}\label{fig:A9A_123-1-1_0.20}
    }
    ~ 
      \subfigure[123-1-1 100\%]{
        \includegraphics[width=7.5cm, height=4.1cm]{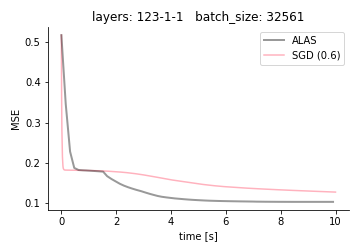}\label{fig:A9A_123-1-1_1.00}
    }
    \caption{Comparison of  ALAS and SGD (with best performing learning rate) on the A9A dataset with a simple neural network with 123 input neurons,  hidden layer with a single neuron and an output neuron.}
    \label{fig:A9A_123-1-1}
\end{figure}

\subsection{An artificial dataset}
\label{subsec:numNN1}
To illustrate the performance of ALAS on highly non-linear problems, we have created two artificial 
datasets, the first of which (thereafter called NN1) was generated using a neural network with random 
weights sampled from a normal distribution ($w_i \in \mathcal{N}(0,3)$) and two input neurons, two hidden 
layers with four and two neurons, respectively. We use hyperbolic tangent activation functions in all layers. 
To produce the actual training data, 50,000 points were sampled from a uniform distribution 
($\vec{x_i} \in \mathcal{R}^2, x_{ij} \in \mathcal{U}(0,1)$) and passed through the generated network to 
produce target values $\{y_i\}_i$. 

The optimization task was then to reconstruct the target values $\{y_i\}_i$ from the generated points 
$\{\vec{x_i}\}_i$ using different neural networks using the MSE as the loss function. We tried 
the values {1, 0.6, 0.3, 0.1 } for the learning rate of SGD, and we compared these variants with ALAS, 
using 100\%, 20 \%, 10 \%, and  5\% of the sample sizes. The results are shown in 
Table~\ref{tab:alg_comparison_nn1}: the algorithm ALAS performed best with a single exception when it got  
stuck in a worse local optimum than the SGD variant.

\begin{table}[h!]
\centering
{\small
\begin{tabular}{c  c c c c | c c c c c }
\multicolumn{5}{c|}{Layers: 2-1-1} & \multicolumn{5}{c}{Layers: 2-2-1} \\ \hline
 alg. & $\skn_k$ & min loss &  loss [8-10]s & iter.  & alg. & $\skn_k$ & min loss & loss [8-10]s & iter. \\ \hline
 ALAS & 5\% & $ \mathbf{	9.08e{-10}}$  & $\mathbf{ 1.12e{-9}}$ &  7471& ALAS & 5\% & $\mathbf{ 6.59e{-10}}$  & $ 4.000 $ & 5736 \\
 SGD (1.0) & 5\% & $  2.49e{-6}$ & $  2.76e{-6}$  & 15432& SGD (1.0) & 5\% & $ 1.60e{-6}$ & $\mathbf{ 1.78e{-6}}$  & 14105 \\ \hline
 ALAS & 10\% & $\mathbf{ 8.94e{-10}}$  & $\mathbf{ 1.01 e{-9}}$ &  5595& ALAS & 10\% & $ \mathbf{ 1.59e{-10}}$  & $ \mathbf{2.18e{-10}}$ & 2352 \\
 SGD (1.0) & 10\% & $ 2.77e{-6}$ & $  3.09e{-6}$  & 14015 & SGD (1.0) & 10\% & $  1.86e{-6}$ & $ 2.07e{-6}$  & 12134 \\ \hline
 ALAS & 20\% & $\mathbf{ 8.91e{-10}}$  & $\mathbf{ 9.39e{-10}}$ & 3353 & ALAS & 20\% & $ \mathbf{ 1.96e{-10}}$  & $\mathbf{ 2.23e{-10}}$ & 1267 \\
 SGD (1.0) & 20\% & $ 3.41e{-6}$ & $ 3.80e{-6}$  &  11533& SGD (1.0) & 20\% & $ 2.20e{-6}$ & $ 2.44e{-6}$  & 10283 \\ \hline
 ALAS & 100\% & $\mathbf{ 1.15e{-9}}$  & $\mathbf{ 1.211e{-9}}$ &  538& ALAS & 100\% & $ \mathbf{9.62e{-10}}$  & $\mathbf{9.99e{-10}}$ & 304 \\
  SGD (1.0) & 100\% & $ 7.51e{-6}$ & $ 8.42e{-6}$  & 5503 & SGD (1.0) & 100\% & $ 5.99e{-6}$ & $ 6.68e{-6}$  &  3833\\ \hline 
 \multicolumn{5}{c|}{Layers: 2-4-1-1} & \multicolumn{5}{c}{Layers: 2-4-2-1} \\ \hline
 ALAS & 5\% & $ \mathbf{9.21e{-10}}$  & $\mathbf{ 9.23e{-10}}$ &  1529 & ALAS & 5\% & $\mathbf{ 9.27e{-10}}$  & $ \mathbf{9.34e{-10}}$ & 922 \\
 SGD (1.0) & 5\% & $ 5.47e{-7}$ & $ 5.60e{-7}$  &  13224& SGD (1.0) & 5\% & $  1.68e{-6}$ & $ 1.86e{-6}$  &  11909\\ \hline
 ALAS & 10\% & $\mathbf{ 9.94e{-10}}$  & $ \mathbf{1.03e{-9}}$ &  818& ALAS & 10\% & $ \mathbf{	1.03e{-9}}$  & $ \mathbf{1.05e{-9}}$ &  493\\
 SGD (1.0) & 10\% & $ 5.75e{-7}$ & $ 5.88e{-7}$  & 10873 & SGD (1.0) & 10\% & $ 2.17e{-6}$ & $ 2.39e{-6}$  & 9192 \\ \hline
 ALAS & 20\% & $ \mathbf{1.26e{-9}}$  & $\mathbf{ 1.34e{-9}}$ & 485  & ALAS & 20\% & $\mathbf{ 	1.25e{-9}}$  & $\mathbf{ 1.30 e{-9}}$ & 262 \\
 SGD (1.0) & 20\% & $ 6.03e{-7}$ & $ 6.12e{-7}$  & 8223& SGD (1.0) & 20\% & $ 2.91e{-6}$ & $ 3.20e{-6}$  & 6850 \\\hline
 ALAS & 100\% & $ \mathbf{4.05e{-9}}$  & $\mathbf{4.43 e{-9}}$ & 94 & ALAS & 100\% & $ \mathbf{	3.59e{-9}}$  & $\mathbf{ 3.96e{-9}}$ & 55 \\
 SGD (1.0) & 100\% & $ 6.76e{-7}$ & $ 6.80e{-7}$  & 2819 & SGD (1.0) & 100\% & $ 9.76e{-6}$ & $ 1.09e{-5}$  & 2008  \\ \hline
\end{tabular}
}
\caption{Results reached over the given time period $t =10 ~s$ on the artificial dataset NN1. }
\label{tab:alg_comparison_nn1}
\end{table}

\subsection{A second artificial dataset}

The second artificial dataset, called NN2, was created using the same process as described in the main paper for the first. 
For this second dataset, we used a deeper neural network in order to 
introduce more nonlinearities: this network consisted of four input neurons, three hidden layers with up to 
four neurons, and one output layer producing the targets $\{y_i\}_i$. The networks used for experiments had one or two hidden layers --- first hidden layer had always four neurons and the second had one or two neurons if present. All neurons used a hyperbolic tangent as their activation function, the optimized function was the MSE. Few runs of the ALAS and the SGD algorithms are shown in Figures~\ref{fig:NN2_4-4-1} and \ref{fig:NN2_4-4-2-1}. The step type distribution for different runs is shown in Figure~\ref{fig:steps_NN2} --- the \textit{Regularized Newton} step type was the most frequent.
Our implementation of ALAS performed generally better than the best SGD variant, sometimes by a significant 
margin.

\begin{table}[h!]
\centering
{\small
\begin{tabular}{c  c c c c | c c c c c }
\multicolumn{5}{c|}{Layers: 4-4-1} & \multicolumn{5}{c}{Layers: 4-4-2-1}  \\ \hline
 alg. & $\skn_k$ & min loss & loss [8-10]s & iter. &  alg. & $\skn_k$ & min loss & loss [8-10]s  & iter.  \\ \hline
 ALAS & 5\% &  \textbf{0.0153} & \textbf{0.0163} & 1654  & ALAS & 5\% & 	\textbf{0.0150}  & \textbf{0.0154} & 1046 \\
 SGD (1.0) & 5\% & 0.0411  &0.0466 &  13185& SGD (1.0) & 5\% &  0.0167 & 0.0222 & 11814 \\ \hline
 ALAS & 10\% & \textbf{0.0145}  & \textbf{0.0156} & 1218& ALAS & 10\% & \textbf{0.0160 } & \textbf{0.0161}  &  650\\
 SGD (1.0) & 10\% & 0.0466  & 0.0513 & 10548 & SGD (1.0) & 10\% & 0.0185  & 0.0260 &  8764\\ \hline
 ALAS & 20\% & \textbf{0.0158}   & \textbf{0.0168} &  727& ALAS & 20\% & \textbf{0.0159}  & \textbf{0.0160} & 320 \\
 SGD (1.0) & 20\% &  0.0527 & 0.0562 & 7617 & SGD (1.0) & 20\% & 0.0242  & 0.0317 & 6123 \\ \hline
 ALAS & 100\% & \textbf{0.0122}  & \textbf{0.0133 }&  126& ALAS & 100\% & \textbf{0.0148}  & \textbf{0.0153} & 69 \\
 SGD (1.0) & 100\% & 0.0646  & 0.0651&  3068& SGD (1.0) & 100\% & 0.0384  & 0.0489 & 2222 \\ \hline
\end{tabular}
}
\caption{The comparison of the minimal (full) losses reached over the given time period $t =10 ~s$ on the artificial dataset NN2. The number in the parenthesis for SGD entries is the step size. The column \textit{loss [8-10]s} shows the median loss over the last two seconds.}
\label{tab:alg_comparison_nn2}
\end{table}

\begin{figure}
    \centering
    \subfigure[4-4-1 5\%]{
        \includegraphics[width=7.5cm, height=4.1cm]{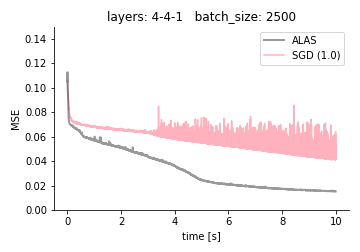}\label{fig:NN2_4-4-1_0.05}
    }
    ~ 
    \subfigure[4-4-1 10\%]{
        \includegraphics[width=7.5cm, height=4.1cm]{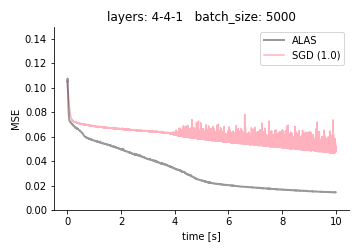}\label{fig:NN2_4-4-1_0.10}
    }
    \\
    \subfigure[4-4-1 20\%]{
        \includegraphics[width=7.5cm, height=4.1cm]{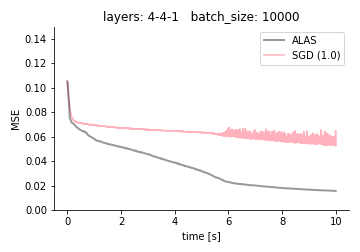}\label{fig:NN2_4-4-1_0.20}
    }
    ~ 
    \subfigure[4-4-1 100\%]{
        \includegraphics[width=7.5cm, height=4.1cm]{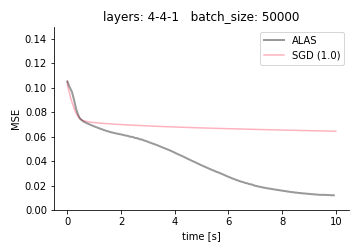}\label{fig:NN2_4-4-1_1.00}
    }
    \caption{Evaluation of the ALAS algorithm on artifical task NN2 compared to the SGD with best performing learning rate. Full losses are depicted.}
    \label{fig:NN2_4-4-1}
\end{figure}

\begin{figure}
    \centering
    \subfigure[4-4-2-1 5\%]{
    \includegraphics[width=7.5cm, height=4.1cm]{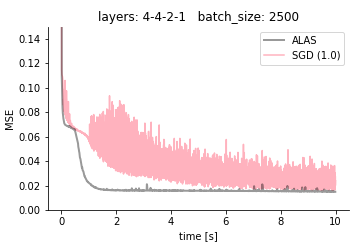}\label{fig:NN2_4-4-2-1_0.05}
    }
    ~ 
    \subfigure[4-4-2-1 10\%]{
    \includegraphics[width=7.5cm, height=4.1cm]{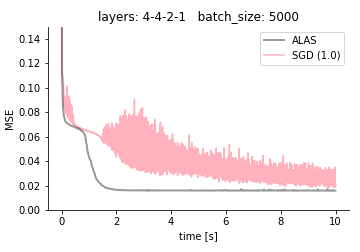}\label{fig:NN2_4-4-2-1_0.10}
    }
     \\
    \subfigure[4-4-2-1 20\%]{
    \includegraphics[width=7.5cm, height=4.1cm]{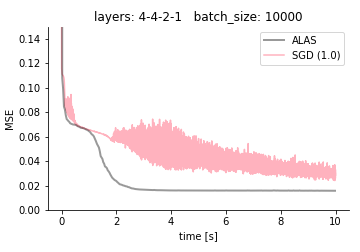}\label{fig:NN2_4-4-2-1_0.20}
    }
    ~ 
    \subfigure[4-4-2-1 100\%]{
    \includegraphics[width=7.5cm, height=4.1cm]{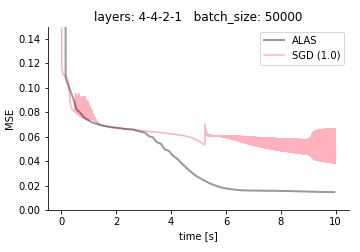}\label{fig:NN2_4-4-2-1_1.00}
    }
    \caption{Evaluation of the ALAS algorithm on artificial task NN2 compared to the SGD with best performing learning rate. Full losses are depicted.}
    \label{fig:NN2_4-4-2-1}
\end{figure}

\begin{figure}
    \centering
    \subfigure[4-2-1 10\%]{
    \includegraphics[width=7.5cm, height=4.1cm]{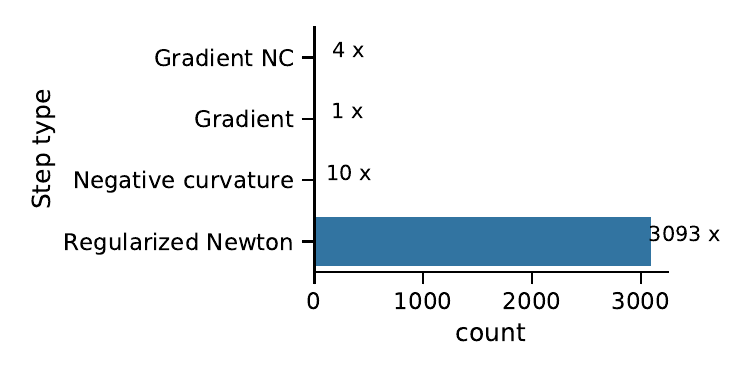}\label{fig:steps_NN2_nn_4-2-1_10}
    }
    ~ 
    \subfigure[4-4-2-1 100\%]{
    \includegraphics[width=7.5cm, height=4.1cm]{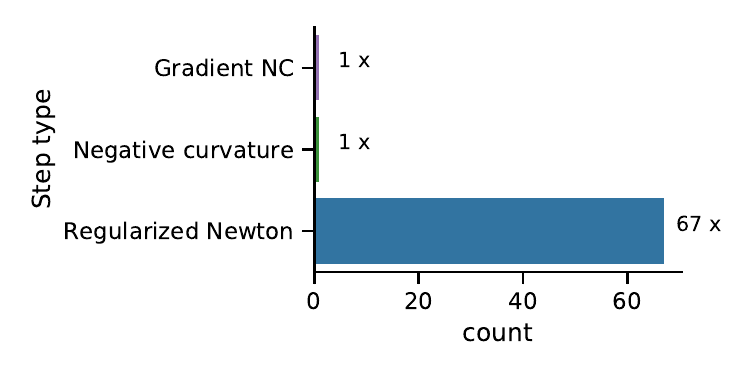}\label{fig:steps_NN2_nn_4-4-2-1_100}
    }
    \caption{The step type distribution of a single run of ALAS algorithm on the NN2 task.}
    \label{fig:steps_NN2}
\end{figure}


\end{appendices}

\end{document}